\documentclass[11pt]{article}
\usepackage{amsmath,amssymb,amsfonts,amsthm}
\usepackage{setspace}
\newenvironment{myabstract}{\par\noindent
{\bf Abstract . } \small }
{\par\vskip8pt minus3pt\rm}
\newcounter{item}[section]
\newcounter{kirshr}
\newcounter{kirsha}
\newcounter{kirshb}
\newenvironment{enumarab}{\setcounter{kirshb}{1}
\begin{list}{(\arabic{kirshb})}{\usecounter{kirshb}} }{\end{list}}
\newtheorem{theorem}{Theorem}[section]

\newtheorem{lemma}[theorem]{Lemma}
\newtheorem{corollary}[theorem]{Corollary}
\theoremstyle{definition}

\newtheorem{example}[theorem]{Example}
\newtheorem{definition}[theorem]{Definition}
\def\R{\mathbb{R}}

\def\C{{\mathfrak{C}}}
\def\Fm{{\mathfrak{Fm}}}

\def\At{{\bf At}}
\def\Nr{{\mathfrak{Nr}}}

\def\Sg{{\mathfrak{Sg}}}
\def\Fm{{\mathfrak{Fm}}}
\def\A{{\mathfrak{A}}}
\def\B{{\mathfrak{B}}}
\def\C{{\mathfrak{C}}}
\def\D{{\mathfrak{D}}}
\def\M{{\mathfrak{M}}}
\def\N{{\mathfrak{N}}}

\def\CA{{\bf CA}}

\def\QEA{{\bf QEA}}

\def\Df{{\bf Df}}

\def\Lf{{\bf Lf}}

\def\PA{{\bf PA}}

\def\K{{\bf K}}
\def\K{{\bf K}}

\def\RCA{{\bf RCA}}

\def\Rd{{\ Rd}}
\def\(R)RA{{\bf (R)RA}}
\def\RA{{\bf RA}}

\def\R{\mathbb{R}}

\def\Id{{\bf Id}}
\def\c #1{{\cal #1}}
 \def\CA{{\sf CA}}
\def\B{{\sf B}}
\def\G{{\sf G}}
\def\w{{\sf w}}
\def\y{{\sf y}}
\def\g{{\sf g}}

\def\r{{\sf r}}
\def\K{{\sf K}}

 \def\Cm{{\mathfrak{Cm}}}
\def\Nr{{\mathfrak{Nr}}}

\def\R{\sf R}
\def\Ra{{\mathfrak{Ra}}}
\def\Ca{{\mathfrak{Ca}}}
\def\set#1{\{#1\} }
\def\Ra{{\mathfrak{Ra}}}
\def\Nr{{\mathfrak{Nr}}}
\def\Tm{{\mathfrak{Tm}}}
\def\A{{\mathfrak{A}}}
\def\B{{\mathfrak{B}}}
\def\C{{\mathfrak{C}}}
\def\D{{\mathfrak{D}}}

\def\A{{\mathfrak{A}}}
\def\B{{\mathfrak{B}}}
\def\C{{\mathfrak{C}}}
\def\D{{\mathfrak{D}}}

\def\GG{{\mathfrak{GG}}}

\def\L{{\mathfrak{L}}}
\def\Rd{{\mathfrak{Rd}}}

\def\At{{\mathfrak{At}}}
\def\L{{\mathfrak{L}}}

\def\CA{{\bf CA}}
\def\RA{{\bf RA}}

\def\RCA{{\bf RCA}}

\def\G{{\bf G}}

\def\F{{\mathfrak{F}}}
\def\At{{\sf{At}}}
\def\N{\mathbb{N}}
\def\R{\mathfrak{R}}

\def\Cs{{\sf Cs}}

\def\c #1{{\cal #1}}

\def\pa{$\forall$}
\def\pe{$\exists$}

\def\ef{Ehren\-feucht--Fra\"\i ss\'e}

\def\nodes{{\sf nodes}}

\def\Ra{{\mathfrak{Ra}}}
\def\Nr{{\mathfrak{Nr}}}
\def\Z{{\cal Z}}
\def\CA{{\bf CA}}
\def\RCA{{\bf RCA}}
\def\c#1{{\mathcal #1}}

\def\set#1{ \{#1\}}

\def\Ca{{\mathfrak Ca}}

\def\pe{$\exists$}
\def\pa{$\forall$}
\def\Cm{{\mathfrak Cm}}
\def\Sg{{\mathfrak Sg}}

\def\Rl{{\mathfrak Rl}}
\def\N{{\cal N}}

\def\ls { L\"owenheim--Skolem}
\def\At{{\sf At}}

\def\rng{{\sf rng}}
\def\dom{{\sf dom}}
\def\Cm{{\sf Cm}}
\def\Mat{{\sf Mat}}
\def\w{{\sf w}}
\def\g{{\sf g}}
\def\y{{\sf y}}
\def\r{{\sf r}}
\def\bb{{\sf b}} 

\def\RQEA{{\sf RQEA}}
\def\ws{winning strategy}
\def\ef{Ehren\-feucht--Fra\"\i ss\'e}

 \def\CA{{\sf CA}}
\def\Cs{{\sf Cs}}
\def\RCA{{\sf RCA}}

\def\RA{{\sf RA}}
\def\PA{{\sf PA}}

\def\QEA{{\sf QEA}}

\def\y{{\sf y}}
\def\g{{\sf g}}
\def\bb{{\sf b}}
\def\r{{\sf r}}
\def\w{{\sf w}}
\def\Z{{\mathbb{Z}}}
\def\N{{\mathbb{N}}}

\def\c{{\sf c}}
\def\s{{\sf s}}
\def\t{{\sf t}}

\def\d{ Dedekind--MacNeille}
\def\Id{{\sf Id}}

\def\Df{{\sf Df}}
\def\Lf{{\sf Lf}}
\def\K{{\sf K}}
\def\nodes{{\sf nodes}}

\def\G{{\bold G}}

\def\Df{{\sf Df}}
\def\PA{{\sf PA}}
\def\Id{{\sf Id}}
\def\QEA{{\sf QEA}}
\def\s{{\sf s}}

\def\CA{{\sf CA}}
\def\K{{\sf K}}

\def\RCA{{\sf RCA}}
\def\RQEA{{\sf RQEA}}
\def\A{{\mathfrak{A}}}
\def\Cs{{\sf Cs}}
\def\t{{\sf t}}
\def\V{{\sf V}}

\def\G{{\mathfrak{G}}}
\def\GG{{\bold G}}
\def\de{Dedekind-MacNeille}
\def\Cm{{\mathfrak{Cm}}}
\def\M{{\sf M}}
\def\T{{\sf T}}
\def\Nr{{\sf Nr}}
\def\H{{\cal H}}
\title{A brief history of algebraic logic, from neat embeddings to games theory and rainbow constructions}
\author{Tarek Sayed Ahmed\\
Department of Mathematics, Faculty of Science,\\
Cairo University, Giza, Egypt.
 }
\date{}
\begin{document}
\maketitle

\begin{myabstract}  
We take a long magical tour in algebraic logic, starting from classical results on
neat embeddings due to Henkin, Monk and Tarski, all the way
to recent results in algebraic logic using so--called rainbow constructions invented by Hirsch and Hodkinson. 
Highlighting the connections with graph theory, model theory, and finite combinatorics,  this
article aspires to present topics of broad interest in a
way that is hopefully accessible to a large audience.  
The  paper has a survey character but it contains new approaches to old ones. 
We aspire to make our survey fairly 
comprehensive, at least in so far as Tarskian algebraic logic, specifically,  the theory of cylindric algebras, is concerned.
Other topics, such as abstract algebraic logic, modal logic and the so--called (central) finitizability problem 
in algebraic logic will be dealt with;  the last in some detail.  
Rainbow constructions are used to solve problems adressing classes of cylindric--like algebras 
consisting of algebras having a neat embedding property.  The hitherto obtained results generalize 
seminal results of Hirsch and Hodkinson on non--atom canonicity, 
non--first order definabiity and non--finite 
axiomatizability, proved for  
classes of representable cylindric algebras of finite dimension$>2$. We show that such results remain valid 
for cylindric algebras  possesing relativized  {\it clique guarded} 
representations that are {\it only locally} well behaved. The paper is written in a way that makes it  
accessible to non--specialists  curious about the state of the art in Tarskian algebraic logic. 
Reaching the boundaries of current research, the paper also aspires to be 
informative to the practitioner, and even more, stimulates her/him to carry on 
further research in main stream 
algebraic logic.
\end{myabstract}

\section{Introduction}

\subsection{Algebraic Logic}

The topic of this paper is algebraic logic in a broad sense. Initiated at the beginning of the 20th century, formal logic and its study using mathematical machinery
known as metamathematical investigations, or simply metamathematics, is a crucial addition to the collection 
of mathematical catalysts. 
Traced back to the works of Frege, Hilbert, Russell, Tarski, G\"odel and others,
one of the branches of pure mathematics that metamathematics  has precipitated is algebraic logic.
This paper aspires to give a compact, yet fairly comprehensive survey, of the  
history of the subject up to the present day.

Algebraic logic starts from certain special logical considerations, abstracts from them, places them in a general algebraic context
and via this generalization makes contact with other branches of mathematics (like set theory and topology).
It can neither be maintained nor overemphasized that algebraic logic is more algebra than logic,
nor that it is more logic than algebra; in this paper we argue that algebraic logic, particularly the theory of cylindric algebras,
has become sufficiently interesting and deep
to acquire a distinguished status among other subdisciplines of mathematical logic, and indeed of
pure mathematics.
In algebraic logic a  deeper understanding of  both (universal) algebra and logic is not hindered by the largely irrelevant details of a particular logical system; 
it is instead  guided structurally by clear--cut causality.
Hypotheses are kept as general as possible and introduced on a by--need basis, 
and thus results and proofs are modular and easy to track down. 
Access to highly non--trivial results is also considerably facilitated.

The continuous interplay between the specific and the general in algebraic logic brings a
large array of new results for particular non--conventional approaches, unifies several known results,
produces new results in well-studied conventional areas,
and finally reveals previously unknown causality relations. 
Algebraic logic  can be viewed as a playground
where several actors, most notably model theory, set theory, finite combinatorics and graph theory play.
It is an  interesting play which
brought and is likely to bring more significant changes and development to the actors, and in a way, 
it also revolutionized our way of  thinking about logic. 
{\it Algebraizing predicate logic}, 
a task primarly initiated by Tarski proved an extremely rewarding task, with significant early contributions from the 
(also) Polish logicians Rasiowa and Sikorsi,
as well as from Halmos who invented the theory of polyadic algebras 
which is a `cousin' to cylindric algebras.

The history of logic, evolving into {\it mathematical logic or metamathematics} 
has been long and winding. From the beginning
of the contemporary era of logic there were two approaches to the subject, one
centered on the notion of logical equivalence and the other, reinforced by Hilbert's work on metamathematics,
centered
on the notions of assertion and inference.
It was not until much later that logicians started to think about
connections between these two ways of looking at logic. Tarski
\cite{Tar1935} gave the precise connection between Boolean algebra
and the classical propositional calculus. His approach builds on
Lindenbaum's idea of viewing the set of formulas as an algebra with
operations induced by the logical connectives. Logical equivalence
is a congruence relation on the formula algebra. This is the
so-called {\it Lindenbaum--Tarski} method.
When Tarski applied this method to the predicate
calculus, it led him naturally to the concept of cylindric algebras.

Henkin began working with Tarski on the subject of cylindric algebras in the 1950's, and
a report of their joint research appeared in 1961. By then Monk had
also made substantial contributions to the theory. The three
planned to publish a comprehensive two--volume treatise on the theory
of cylindric algebras. The first volume treated cylindric
algebras from a general algebraic point of view, while the second
volume contained other topics, such as the representation
theory, to which Andr\'eka and N\'emeti contributed a great deal,  and connections between cylindric--like
algebras and logic. The second volume though entitled `Cylindric Algebras, Part 2' also dealt with other algebraic formalisms 
of first order logic like Halmos' polyadic algebras and relation algebras.

We can find that the theory of cylindric--like algebras is explicated primarily in three substantial monographs:
Henkin, Monk and Tarski \cite{HMT1, HMT2}, and Henkin, Monk,
Tarski, Andr\'eka and N\'emeti \cite{HMTAN81}.

This covers the development of the subject till the mid--eighties of the last century.
The recent \cite{1}, referred to as `Cylindric Algebras, Part 3', the notation of which we follow here,
gives a representative picture of the research in the area over the last
thirty years,  emphasizing the bridges that  cylindric algebra theory--in the wide sense--has built
with other areas, like combinatorics, graph theory, data base theory, stochastics and other fields.
Not confined to the walls of pure mathematics, algebraic logic
has also found bridges to such apparently remote areas as general relativity, and hyper-computation.
Techniques from \cite{HHbook} are also used but we our treatment will be self--contained; all basic notions and terminolgy will be recalled
in due time.

This paper also surveys, refines and adds to the latest developments
of the subject, but on a smaller scale, of course. In particular, we focus more on `pure' Tarskian
algebraic logic.  We emphasize that the paper is not only purely expository. 

The paper  gives a survey of new recent results with fairly complete proofs, 
and it contains new ideas
as well as new approaches to old ones. 
The results  throughout the whole paper are not presented in a chronological  order as they appeared historically;
rather we move back and forth in time
between deep results in algebraic logic. The temporal (historical) development of a certain topic (particularly in mathematics) 
does not neccesarily coincide with the most logical  one,  for newly discovered results often 
shed light on older ones, resulting in a deeper understanding of both.  In (algebraic) logic new paradigms usually 
present themselves in the
form of new systems competing with old ones. This suggests 
a fresh look at existing logical systems, 
rather than their speedy overthrow.

Both developments do 
not even necessarily coincide with the simplest, which we aspire to achieve in this paper.

For example, take the following three branches in algebraic logic:
(1)  algebraization of logical calculi (deductive systems) leading up to abstract algebraic logic, 
(2)  algebraic approach to first order logic largely due to 
Tarski, manifesting itself in the study  of cylindric  algebras (the central topic of the 
present paper), and finally 
(3)  algebraic approach to non--classical propositional logics like intuitionistic and modal logic. 

The work of Gabbay and Maksimova  on amalgamation and interpolation categorized in (3), 
though historically originate in connection with the theory of polyadic and 
cylindric algebras through the work of pioneers 
including Diagneault, Johnson, Pigozzi and Comer investigating such properties for concrete 
classes of cylindric and polyadic algebras categorized in (2), can be classified 
today as the part of abstract algebraic  logic dealing with definability issues  
categorized in  (1).

The purpose of this paper is twofold. Apart from giving  a general overview to some of the fundamental ideas and methods of applying 
algebraic machinery to logic, we also intend to present recent developments from algebraic logic and logic
in an integrated format that is accessible to the non--specialist and informative for the practioner.  
Our intention is to unify, illuminate and generalize several existing results 
scattered in the literature, hopefully stimulating further research.
We hope that after reading this paper  the reader will get more than a glimpse 
of rapid dissemination of the latest research in the field, and also contribute to it 
if she/he wishes.

However, algebraic logic today is a huge subject.
Therefore our treatment generally is bound to be selective, focused on some but not all topics.
The topics we focus on are
{\it mainstream Tarskian algebraic logic}.  
Occasionally (but not always) few details can be omitted from long proofs. Our 
excuse will be that we wish to emphasize certain concepts 
while saying no more than absolutely necessary.  
To make up for all this, we include many references for those who
wish to dig deeper. 
We focus on {\it cylindric algebras}. Notation used is common or /and  self--explonatory; 
if not, then it will 
be explained at its first occurrence. Our notation is consistent with the notation in \cite{1, HMT1, HMT2}.
In particular, for an ordinal $\alpha$, $\CA_{\alpha}$ denotes the class of cylindric algebras of dimension
$\alpha$.
Although the paper addresses a variety of topics in the textbooks \cite{HMT1, HMT2, HHbook}, we require only familiarity  
with the minimal basics of cylindric algebra theory.  In this respect, 
the paper is fairly self--contained. All required definitions and terminology 
will be recalled in due time whenever needed.
Throughout the paper we make the following convention. We denote infinite ordinals by $\alpha, \beta\ldots$ 
and finite ordinals by $n, m\ldots$. Ordinals which are arbitary meaning that they could be finite or infinite 
will be denoted by  $\alpha,\beta\ldots$. 
Also 
algebras will be denoted by Gothic letters, and when we write $\A$ for an algebra, then we shall be tacitly assuming that 
$A$ denotes its universe, that is 
$\A=\langle A, f_i^{\A}\rangle_{i\in I}$ where $I$ is a non--empty set and $f_i$ $(i\in I)$ are the operations in the signature of $\A$ 
interpreted via $f_i^{\A}$ in $\A$. For better readability, we omit the superscript $\A$ and we write simply 
$\A=\langle A, f_i\rangle_{i\in I}$.

The paper is divided into two parts. The first part is a general survey on recent results in algebraic logic mostly in the last three decades 
building (possibly new) bridges between 
them. The second part is more technical addressing rainbow constructions. Such constructions are used 
to solve problems on classes of algebras 
whose members have a neat embedding 
property. The 
main theorem in the second part is theorem \ref{can}. 

\section*{Part 1}

\section{Cylindric algebras}

A {\it cylindric algebra} consists of a Boolean algebra endowed with an
additional structure consisting of distinguished elements and
operations, satisfying a certain system of equations. The
introduction and study of these algebras has its motivation in two
parts of mathematics: the deductive systems of first-order logic,
and a portion of elementary set theory dealing with spaces of
various dimensions, better known as cylindric set algebras; such algebras also have a geometric twist,
reflected in the terminology `cylinder'. If we are working in $3$ dimensions,
and we apply the unary operations of cylindrifiers (algebraizing existential quantifiers) to a `circle',
then we are forming the cylinder based on this circle.

Cylindric set algebras are algebras whose elements are relations of a certain pre-assigned arity, endowed with set--theoretic operations
that utilize the form of elements of the algebra as sets of sequences.
For a set $V$, ${\cal B}(V)$ denotes the Boolean set algebra $\langle \wp(V), \cup, \cap, \sim, \emptyset, V\rangle$.
We will primarily deal with a set $V$ consisting of $\alpha$--ary sequences, for some ordinal $\alpha$.

Let $U$ be a set and $\alpha$ an ordinal; $\alpha$ will be the dimension of the algebra.
For $s,t\in {}^{\alpha}U$ write $s\equiv_i t$ if $s(j)=t(j)$ for all $j\neq i$.
For $X\subseteq {}^{\alpha}U$ and $i,j<\alpha,$ let
$${\sf C}_iX=\{s\in {}^{\alpha}U: \exists t\in X (t\equiv_i s)\}$$
and
$${\sf D}_{ij}=\{s\in {}^{\alpha}U: s_i=s_j\}.$$

$\langle{\cal B}(^{\alpha}U), {\sf C}_i, {\sf D}_{ij}\rangle_{i,j<\alpha}$ is called {\it the full cylindric set algebra of dimension $\alpha$}
with unit (or greatest element) $^{\alpha}U$.
We follow the conventions of \cite{HMT2} where the cylindric operations in set algebras are denoted by capital letters, while
the cylindric operations in abstract algebras are denoted by small letters.
Examples of subalgebras of such set algebras arise naturally from models of first order theories.
Indeed, if $\M$ is a first order structure in a first
order signature $L$ with $\alpha$ many variables, then one manufactures a cylindric set algebra based on $\M$ as follows.
Let
$$\phi^{\M}=\{ s\in {}^{\alpha}{\M}: \M\models \phi[s]\},$$
(here $\M\models \phi[s]$ means that $s$ satisfies $\phi$ in $\M$), then the set
$\{\phi^{\M}: \phi \in Fm^L\}$ is a cylindric set algebra of dimension $\alpha$, where $Fm^L$ denotes the set of first order formulas taken in 
the signature $L$.  To see why, we have:
\begin{align*}
\phi^{\M}\cap \psi^{\M}&=(\phi\land \psi)^{\M},\\
^{\alpha}{\M}\sim \phi^M&=(\neg \phi)^{\M},\\
{\sf C}_i(\phi^{\M})&=\exists v_i\phi^{\M},\\
{\sf D}_{ij}&=:(x_i=x_j)^{\M}.
\end{align*}
$\Cs_{\alpha}$ denotes the class of all subalgebras of full set algebras of dimension $\alpha$.
Recall that $\CA_{\alpha}$ stands for the class of cylindric algebras of dimension $\alpha$.
This last (equationally defined class) is obtained from cylindric set algebras by a process of abstraction and is defined by a {\it finite} schema
of equations given in \cite[Definition 1.1.1]{HMT1} that holds of course in the more concrete set algebra.
\begin{definition} By \textit{a cylindric algebra of dimension} $\alpha$, briefly a
$\CA_{\alpha}$, we mean an
algebra
$$ {\A} = \langle A, +, \cdot,-, 0 , 1 , {\sf c}_i, {\sf d}_{ij}\rangle_{\kappa, \lambda < \alpha}$$ where $\langle A, +, \cdot, -, 0, 1\rangle$
is a Boolean algebra such that $0, 1$, and ${\sf d}_{i j}$ are
distinguished elements of $A$ (for all $j,i < \alpha$),
$-$ and ${\sf c}_i$ are unary operations on $A$ (for all
$i < \alpha$), $+$ and $.$ are binary operations on $A$, and
such that the following postulates are satisfies for any $x, y \in
A$ and any $i, j, \mu < \alpha$:
\begin{enumerate}
\item [$(C_1)$] $  {\sf c}_i 0 = 0$,
\item [$(C_2)$]$  x \leq {\sf c}_i x \,\ ( i.e., x + {\sf c}_i x = {\sf c}_i x)$,
\item [$(C_3)$]$  {\sf c}_i (x\cdot {\sf c}_i y )  = {\sf c}_i x\cdot  {\sf c}_i y $,
\item [$(C_4)$] $  {\sf c}_i {\sf c}_j x   = {\sf c}_j {\sf c}_i x $,
\item [$(C_5)$]$  {\sf d}_{i i} = 1 $,
\item [$(C_6)$]if $  i \neq j, \mu$, then
 ${\sf d}_{j \mu} = {\sf c}_i
 ( {\sf d}_{j i} \cdot  {\sf d}_{i \mu}  )  $,
\item [$(C_7)$] if $  i \neq j$, then
 ${\sf c}_i ( {\sf d}_{i j} \cdot  x) \cdot   {\sf c}_i
 ( {\sf d}_{i j} \cdot  - x) = 0
 $.
\end{enumerate}
\end{definition}

${\A}\in \CA_{\omega}$ is {\it locally finite}, if the dimension set of every element $x\in A$
is finite. The dimension set of $x$, or $\Delta x$ for short, is the set
$\{i\in \omega: {\sf c}_ix\neq x\}.$ Locally finite algebras correspond to Tarski--Lindenbaum algebras of (first order) formulas;
in such  algebras the dimension set of (an equivalence class of)  a formula reflects  the number of (finite) set of free variables in this formula.
Tarski proved that every locally finite $\omega$-dimensional
cylindric algebra is representable, i.e. isomorphic
to a subdirect product of set algebra each of dimension $\omega$.

Let $\Lf_{\omega}$ denote the class of
locally finite cylindric algebras.
Let $\RCA_{\omega}$ stand for the class of isomorphic copies
of subdirect products
of set algebras each of dimension $\omega$,
or briefly,
the class of $\omega$ dimensional representable
cylindric algebras. Then Tarski's theorem reads
$\Lf_{\omega}\subseteq \RCA_{\omega}$.
This representation theorem is non-trivial; in fact it is equivalent
to G\"odel's celebrated Completeness Theorem \cite[\S 4.3]{HMT2}.

\subsection{Neat embeddings and Monk's result}

Soon in the development of the subject,
it transpired that the class $\Lf_{\omega}$, the algebraic counterpart
of first order logic,  had
some shortcomings when treated as the
sole subject of research in an autonomous algebraic theory. One such assumption is the fixed dimension $\omega$.
The other is local finiteness.
The restrictive character of these two notions
becomes obvious when we turn our attention to cylindric set algebras.
We find that there are algebras of all dimensions,
and set algebras that are not locally finite
are easily constructed. For these reasons the original conception of a cylindric algebra
was extended. The restriction to dimension $\omega$
and local finiteness were removed,
and the class $\CA_{\alpha}$, of cylindric algebras of dimension $\alpha$,
where $\alpha$ is any ordinal, finite or transfinite, was introduced. The logic corresponding to ${\sf RCA}_{\omega}$,
allowing infinitary predicates, is a more basic algebraizable  (in the standard Blok--Pigozzi sense) 
formalism of $L_{\omega, \omega}$.

Three pillars in the development of the subject, and even one can say {\it the} three pillars in the development of the subject are 
Tarski's result that every locally finite cylindric algebra is representable,
Henkin characterization of  the variety of representable algebras of any dimension
via neat embeddings, in his celebrated Neat Embedding Theorem \cite[Theorem 3.2.10]{HMT2}, and Monk's  proof that the variety of
representable algebras of dimension $>2$ cannot
be axiomatized by a finite schema \cite{Monk}. Monk's result had a shattering effect on the
development of the subject.
Not only did it create employment for many algebraic logicians in the past,
but it still has its enormous repercussions till the very present day, most likely creating
employment for others in the future.

The last two results involve the central notion of neat reducts:
\begin{definition} 
Let  $\alpha<\beta$ be ordinals and $\B\in \CA_{\beta}$. Then the {\it $\alpha$--neat reduct} of $\B$, in symbols
$\Nr_{\alpha}\B$, is the
algebra obtained from $\B$, by discarding
cylindrifiers and diagonal elements whose indices are in $\beta\sim \alpha$, and restricting the universe to
the set $Nr_{\alpha}B=\{x\in \B: \{i\in \beta: {\sf c}_ix\neq x\}\subseteq \alpha\}.$
\end{definition}
Let $\alpha$ be any ordinal.  If $\A\in \CA_\alpha$ and $\A\subseteq \Nr_\alpha\B$, with $\B\in \CA_\beta$ ($\beta>\alpha$), 
then we say that $\A$ {\it neatly embeds} in $\B$, and 
that $\B$ is a {\it $\beta$--dilation of $\A$}, or simply a {\it dilation} of $\A$ if $\beta$ is clear 
from context.  We also say that $\A$ has {\it a neat embedding property} and if $\beta\geq \alpha+\omega$, we say that $\A$ has {\bf the} neat embedding property.
For $\sf K\subseteq {\sf CA}_\beta$, and $\alpha<\beta$, ${\sf Nr}_\alpha{\sf K}=\{\Nr_\alpha\B: \B\in {\sf K}\}\subseteq \CA_\alpha$.

One can show that $\A\in \CA_{\alpha}$ has the neat embedding property 
$\iff \A\in \RCA_{\alpha}\iff \A\in \bold S{\sf Nr}_\alpha\CA_{\alpha+\omega}$, cf. \cite[Theorem 2.6.35]{HMT1}.
The last equivalence is Henkin's celebrated neat embedding theorem.
For $2<n<\omega$, what Monk proved is that for any $k\in \omega$, there is an algebra 
$\A_k\in \bold S{\sf Nr}_n\CA_{n+k}\sim \RCA_n$, 
such that the ultraproduct $\Pi_{k/U}\A_k/U\in \RCA_n$
for any non--principal ultrafilter on $U$.  This implies that $\RCA_n$ is 
not finitely axiomatizable.

If the variety of representable cylindric algebras of dimension at least three had turned out to be 
axiomatized by a finite schema, 
algebraic logic would have evolved along a significantly  different path than it did in the past 
$45$ years, or so. This would have undoubtfully marked the end of the abstract class $\CA_{\alpha}$ $(\alpha$ an ordinal) 
as a separate subject of research;  after all why bother about abstract algebras, if a few nice extra axioms can lead us from those
to concrete algebras consisting of genuine relations, with set theoretic operations uniformy defined over these relations. However, 
due to Monk's non--finitizability result, together with its improvements by various algebraic logicians (from Andr\'eka to Venema) 
$\CA_{\alpha}$ was here to stay and its ``infinite distance' from $\RCA_{\alpha}$, when $\alpha>2$,
became an important research topic.  

{\bf Monk's non-finite axiomatizability result marked the end of an era 
and the beginning of a new one.}

\section{Atom--canonicity, \d\ completions and canonical extensions}

\subsection{Some history}  Let us start by recalling some results from the old era; from Stone's classical representability 
result for Boolean algebras all the way to  
{\it canonical extensions} and {\it completions} of {\it Boolean algebras with operators} 
of which cylindric algebras are a special case; 
a topic initiated by Tarski and 
his student Jonsson.

In 1930 Stone proved Stones representability theorem for Boolean algebras. Stone showed  that every Boolean algebra $\B$ can be embedded into a complete and atomic
Boolean set algebra, namely, the Boolean algebra of the class of all
subsets of the set
of ultrafilters in $\B$. 

This {\it canonical extension} of $\B$ was characterized
topologically 
by Jonsson and Tarski in 1950 but they went further. They developed a theory of Boolean algebras
{\it with operators} (similar in spirit to the theory of groups with
operators), proved
that every Boolean algebra with operators $\B$ can be embedded into a canonical
complete and atomic Boolean algebra with operators $\A$, having the signature 
as $\B$ {\it called the  canonical extension} of $\B$, and they established a number of preservation theorems concerning
equations and
universal Horn sentences that are preserved under the passage from $\B$ to $\A$. In retrospect, the representation theorem of Jonsson and Tarski
is the first completeness theorem for multimodal logic using canonical models.
Jonsson and Tarski concluded that the canonical extension of every abstract relation algebra is a relation algebra 
and similarly the canonical extension of an abstract cylindric algebra is a
cylindric algebra. However, they did not settle the question of whether e.g. the canonical extension of a
{\it representable cylindric algebra} is again representable.  This problem was
settled in the affirmative
sometime in the 1960s, by Monk (unpublished).

MacNeille and Tarski showed in the 1930s that every Boolean algebra has 
another natural complete extension - and in fact it is a minimal complete
extension
that is formed using Dedekind cuts,  now known as the  {\it \d\ completion}. 
Monk developed the theory
of \d\ completions extensions of Boolean algebras with complete operators in analogy with
the Jonsson--Tarski theory of canonical extensions of Boolean algebras
with operators. In particular, Monk  proved an analogous preservation theorem, and
concluded
that the minimal complete extension of every relation and cylindric algebra is again
a cylindric algebra \cite{HMT1}. 
But, unlike the case with canonical extensions, Monk was
unable to settle the question of whether the \d\ 
completion of a representable finite--dimensional
cylindric algebra of dimension $>2$
is representable or not.  

Contrary to canonical extensions (and possibly expectations), this question was finally settled negatively 
by Hodkinson in 1997 \cite{Hodkinson} by showing that for $2<n<\omega$, $\sf RCA_n$ is not {\it atom--canonical} (to be recalled below).
This readily implies that ${\sf RCA}_n$ is not closed under \de\ completions.
This result is substantially generalized in \cite{mlq} by showing that for any $2<n<\omega$, for any 
$k\geq 3$ the variety $\bold S{\sf Nr}_n\CA_{n+k}$ is not atom--canonical giving Hodkinson's aforementioned result when $k=\omega$ as a special case, by observing
that ${\sf RCA}_n=\bold S{\sf Nr}_n\CA_{\omega}$. 
The result in {\it op.cit} will be recalled with a complete proof 
in theorem \ref{can}  where it will be shown 
that for $2<n<\omega$, there is an atomic countable and simple $\A\in \RCA_n$, such that its \d\ completion, 
namely, the {\it complex algebra of its atom structure} is outside $\bold S{\sf Nr}_n\CA_{n+3}$.

\subsection{Atom structures and complex algebras}

We recall the notions of {\it atom structures} and {\it complex algebra} in the framework of Boolean algebras 
with operators of which $\CA$s  are a special case \cite[Definition 2.62, 2.65]{HHbook}. 
The action of the non--Boolean operators in a completely additive (where operators distribute over arbitrary joins componentwise) 
atomic Boolean algebra with operators, $(\sf BAO)$ for short, is determined by their behavior over the atoms, and
this in turn is encoded by the {\it atom structure} of the algebra.

\begin{definition}\label{definition}(\textbf{Atom Structure})
Let $\A=\langle A, +, -, 0, 1, \Omega_{i}:i\in I\rangle$ be
an atomic $\sf BAO$ with non--Boolean operators $\Omega_{i}:i\in I$. Let
the rank of $\Omega_{i}$ be $\rho_{i}$. The \textit{atom structure}
$\At\A$ of $\A$ is a relational structure
$$\langle At\A, R_{\Omega_{i}}:i\in I\rangle$$
where $At\A$ is the set of atoms of $\A$ 
and $R_{\Omega_{i}}$ is a $(\rho(i)+1)$-ary relation over
$At\A$ defined by
$$R_{\Omega_{i}}(a_{0},
\cdots, a_{\rho(i)})\Longleftrightarrow\Omega_{i}(a_{1}, \cdots,
a_{\rho(i)})\geq a_{0}.$$
\end{definition}
\begin{definition}(\textbf{Complex algebra})
Conversely, if we are given an arbitrary first order structure
$\mathcal{S}=\langle S, r_{i}:i\in I\rangle$ where $r_{i}$ is a
$(\rho(i)+1)$-ary relation over $S$, called an {\it atom structure}, we can define its
\textit{complex
algebra}
$$\Cm(\mathcal{S})=\langle \wp(S),
\cup, \setminus, \phi, S, \Omega_{i}\rangle_{i\in
I},$$
where $\wp(S)$ is the power set of $S$, and
$\Omega_{i}$ is the $\rho(i)$-ary operator defined
by$$\Omega_{i}(X_{1}, \cdots, X_{\rho(i)})=\{s\in
S:\exists s_{1}\in X_{1}\cdots\exists s_{\rho(i)}\in X_{\rho(i)},
r_{i}(s, s_{1}, \cdots, s_{\rho(i)})\},$$ for each
$X_{1}, \cdots, X_{\rho(i)}\in\wp(S)$.
\end{definition}
It is easy to check that, up to isomorphism,
$\At(\Cm(\mathcal{S}))\cong\mathcal{S}$ alway. If $\A$ is
finite then of course
$\A\cong\Cm(\At\A)$. 
For algebras $\A$ and $\B$ having the same signature expanding that of Boolean algebras, 
we say that $\A$ is dense in $\B$ if $\A\subseteq \B$ and for all non--zero $b\in \B$, there is a non--zero 
$a\in A$ such that $a\leq b$.

An atom structure will be denoted by $\bf At$.  An atom structure $\bf At$ has the signature of $\CA_\alpha$, $\alpha$ an ordinal, 
if  $\Cm\bf At$ has the signature of $\CA_\alpha$, in which case we say that $\bf At$ is an $\alpha$--dimensional atom structure. 

\begin{definition}\label{canonical} 
(a) Let $V$ be a variety of $\sf CA_n$s. Then $V$ is {\it atom--canonical} if whenever $\A\in V$ and $\A$ is atomic, then $\mathfrak{Cm}\At\A\in V$.

(b) The {\it  \d\ completion} of  $\A\in \CA_n$, is the unique (up to isomorphisms that fix $\A$ pointwise)  complete  
$\B\in \CA_n$ such that $\A\subseteq \B$ and $\A$ is {\it dense} in $\B$.
\end{definition}

If $\A\in \CA_n$ is atomic,  then 
$\mathfrak{Cm}\At\A$ 
is the {\it \d\ completion of $\A$.}
If $\A\in \CA_n$, then its atom structure will be denoted by $\At\A$ with domain the set of atoms of $\A$ denoted by $At\A$.
By the same token for a class $\sf L\subseteq \CA_n$ consisting of atomic algebras, we denote 
the class of first order structures $\bold K=\{\At\A: \A\in {\sf L}\}$ by $\At(\sf L)$, so that  if ${\bf At}\in \bold K$, 
then there exists $\A\in \sf L$, such that $\At\A=\bf At$. 

We deal only with atom structure having the signature of $\CA_n$.
{\it Non atom--canonicity} can be proved 
by finding {\it weakly representable atom structures} that are not {\it strongly representable}.

\begin{definition} Let $\alpha$ be an ordinal. An atom structure $\bf At$ of dimension $\alpha$ is {\it weakly representable} if there is an atomic 
$\A\in \RCA_{\alpha}$ such that $\At\A=\bf At$.  The atom structure  $\bf At$ is {\it strongly representable} if for all $\A\in \CA_{\alpha}$, 
$\At\A=\bf At \implies \A\in {\sf RCA}_{\alpha}$.
\end{definition}
It can be proved for $\CA_{\alpha}$s, where $\alpha$ is any ordinal, that $\At\A$ is weakly representable $\iff$ the {\it term algebra}, in symbols 
$\Tm\At\A$ is representable.
 $\Tm\At\A$ is the subalgebra of $\Cm\At\A$ generated by the atoms. 
On the other hand, $\At\A$ is strongly representable $\iff$ $\Cm\At\A$ is representable.
Fix $2<n<\omega$. These two notions (strong and weak representability) do not coincide for cylindric algebras as proved by Hodkinson \cite{Hodkinson}.
In theorem \ref{can}, we generalize Hodkinson's result by showing that there are two atomic $\CA_n$s sharing the same atom structure, one is representable and the other is even outside 
$\bold S{\sf N}r_n\CA_{n+3}(\supsetneq \RCA_n$). In particular, there is an algebra outside $\bold S{\sf Nr}_n\CA_{n+3}$ having a dense representable 
subalgebra.

\section{Monk algebras, the good and the bad}

Graphs will be frequently used in what follows. We recall some basics. A \textit{(directed) graph} is a set $G$ (of \textit{nodes} or
\textit{vertices}) endowed with a binary relation $E$, the edge
relation. A pair $(x, y)$ of elements of $G$ is said to be an
\textit{edge} if $xEy$ holds. A directed graph is said to be
{\it complete} if $(x, y)$ is an edge for all nodes $x, y$. A graph is
said to be {\it undirected} if $E$ is symmetric and irreflexive. An
undirected graph is {\it complete} if $(x, y)$ is an edge for all {\it distinct}
nodes $x, y$.  Finite ordinals can (and will) 
be viewed as complete irreflexive graphs the obvious way, cf. theorem \ref{can}.

A \textit{clique} in an undirected graph with set of
nodes $G$ is a set $C \subseteq G$ such that each pair of distinct
nodes of $C$ is an edge.

\begin{definition}\label{graph} Let $\G=(G, E)$ be an undirected graph ($G$ is the set of vertices and $E$ is an ireflexive symmetric 
binary relation on $E$), and
$C$ be a non-empty set of `colours'.
\begin{enumerate}

\item A subset $X\subseteq G$ is said to be an {\it independent set} if $(x,y)\in E$
for all $x,y\in X$.

\item A function $f:G\to C$ is called a {\it $C$ colouring} of  $\G$ if
$(v,w)\in E$ implies that $f(v)\neq f(w)$. 

\item The {\it chromatic number} of $\G$, denoted by $\chi(\G),$ is
the size of the smallest finite set $C$ such that there exists
a $C$ colouring of $\G$, if
such a $C$ exists, otherwise $\chi(\G)=\infty.$

\item A {\it cycle} in $\G$ is a finite sequence $\mu=(v_0,\ldots v_{k-1})$ (some $k\in \omega$) of
distinct nodes, such that $(v_0, v_1), \ldots (v_{k-2}, v_{k-1}), (v_{k-1}, v_0) \in E$. The {\it length} 
of such a cycle is $k.$

\item The {\it girth} of $\G$, denoted by $g(\G)$, is the
length of the shortest cycle
in $\G$ if $\G$ contains cycles, 
and $g(\G)=\infty$ othewise.
\end{enumerate}
\end{definition}

\subsection{From good to bad}

Fix $2<n<\omega$. Because $\RCA_n$ is a variety, an atomic algebra $\A\in \RCA_n\iff$ 
all equations axiomatizing $\RCA_n$ holds in $\A$.
From the point of view of  $\At\A$, the atom structure of $\A$, each equation in the signature of $\RCA_n$ 
corresponds to a certain universal monadic second--order statement, 
in the signature of $\At\A$ where the universal quantifiers are restricted 
to ranging over the set of atoms that lie below elements of $\A$. Such a statement fails  $\iff \At\A$  is partitioned 
into finitely many $\A$--definable sets (sets definable using atoms of $\A$ as parameters) with certain 
`bad' properties. Call this a {\it bad partition}. 
A bad partition of a graph is a 
{\it finite colouring}, namely, a partition of its sets of nodes into finitely many independent sets. 
A typical Monk argument  is to construct for finite dimension $>2$ a sequence of non--representable algebras based on graphs with bad partitions
(having finite chromatic number) converging to one that is based on a graph having {\it infinite chromatic number},
hence, representable. 
It follows immediately that the variety of representable algebras of dimension $>2$ 
is not finitely axiomatizable.

We call the non--representable algebras in the sequence {\it bad} Monk algebras. Based on 
graphs  having finite chromatic number, they are not representable as subdirect products of set algebras.  This is proved by an application
of  Ramsey's theorem. But these bad algebras converge to a {\it good} 
infinite Monk algebra that is based on a graph that has
infinite chromatic number. 
This (limit) algebra is {\it good in the sense that it permits a representation}; it is a subdirect product of set algebras.
We borrow the terminolgy `bad and good' from \cite{strongca} where such notions are applied to 
the {\it graph
used in the construction of the Monk algebra}. A  graph is {\it good} if it has infinite chromatic number, othewise 
it is, as mentioned above, bad, that is, it has a finite colouring.

Constructing algebras based on graphs having arbitrarily large chromatic number converging to one that is based
on a graph having infinite chromatic number seems a plausible thing to do, and
indeed it can be done, witnessing non--finite axiomatizability
of the class of representable algebras of finite dimension $>2$. Monk's original proof  \cite{Monk}, refined by Maddux \cite{Maddux} 
can be seen this way. Monk used {\it finite} bad Monk algebras converging to a good {\it infnite one}.
The idea involved in the construction of a bad finite Monk algebra is not so hard.
Such algebras are finite, hence atomic, more precisely their Boolean reducts are atomic.
The atoms are given colours, and
cylindrifications and diagonals are defined by stating that monochromatic triangles
are inconsistent. If a Monk's algebra has many more atoms than colours,
it follows from Ramsey's Theorem that any representation
of the algebra must contain a monochromatic triangle, so the algebra
is not representable.

\subsection{From bad to good} 

Conversely, one can form an  {\it anti--Monk ultraproduct},
of  a sequence $(\A_i: i\in \omega)$ of {\it good} infinite Monk algebras 
(based on graphs with infinite chromatic number) converging to an infinite bad atomic algebra $\A$, namely, one that
is based on a graph that is
{\it only $2$--colourable} \cite{strongca}. This last algebra is representable, but only {\it weakly}.
This means that its subalgebra generated by the atoms is representable, but its \d\ minimal completion, 
which is the complex algebra of its atom structure,  namely, $\Cm\At\A$, 
is not. In this sense, this limit $\A$ is not {\it strongly} representable. But every element of the sequence $\A_i$ $(i\in \omega)$ is strongly representable, in the same 
sense, meaning that $\Cm\At\A_i$ is representable,
The aforementioned technique of proof, showing that the class of strongly representable cylindric algebras is not elementary
is due to Hirsch and Hodkinson \cite{strongca}.

The {\it ultraproduct} that Monk used in his 1969 seminal result is a `reverse' to this one,
and is more intuitive, since as indicated it is plausible that a sequence of graphs having arbitrarily finite chromatic numbers getting larger and larger, 
converges to one that has
infinite chromatic number, but the `other way round' is hard to visualize, 
let alone implement.  So how did Hirsch and Hodkinson prove their result?
Fix $2<n<\omega$. Recall that a $\CA_n$  atom structure $\bf At$ {\it strongly representable} $\iff$ the complex algebra $\Cm\bf At$ is representable.
So that an atomic algebra $\A$ is strongly representable (in the sense that $\Cm\At\A\in {\sf RCA}_n$) $\iff$ its atom structure is strongly representable.
One shows that an atom structure is strongly representable $\iff$ it has {\it no bad partition using any sets at all}. So, here, we want to find atom structures, with no bad partitions, 
with an ultraproduct that {\it does have a bad partition.}

From a graph $\Gamma$, one can create an atom structure that is strongly representable $\iff$ the graph is good, namely,
it has has no finite colouring; this atom structure is denoted by $\rho(\mathfrak{I}(\Gamma))$ in \cite{HHbook2} 
and its complex algebra $\Cm(\rho(\Gamma))$ is denoted by $\mathfrak{M}(\Gamma)$ in \cite[Proposition 3.6.8]{HHbook2}. 
So the problem reduces to finding a sequence of graphs with no finite colouring, 
with an ultraproduct that does have a finite colouring. We want graphs of {\it infinite chromatic numbers}, having an ultraproduct
with {\it finite chromatic number}. It is not obvious, {\it a priori}, that such graphs actually exist.
And here Erdos' probabilistic methods offer solace. Graphs like this can be found using the probabilistis methods of Erd\"os, for those methods
render finite graphs of arbitrarily large chromatic number and girth \cite{Erdos, strongca}. 
By taking disjoint unions, we obtain graphs of infinite chromatic number (no bad partitions) and arbitarly large girth. A non--principal 
ultraproduct of these has no cycles, so has chromatic number 2 (bad partition), witness the proof of \cite[Corollary 3.7.2]{HHbook2}.

{\bf Both constructions from `bad to good' and from `good to bad' will be encountered in some detail in \S 9.} 

\section{Further repercussions of the seminal result of Monk back in 1969 and its refinements; two complementary 
paths}

Monk's seminal result proved in 1969 \cite{Monk}, showing that the class of representable cylindric algebras
of any dimension $>2$ is not finitely axiomatizable, had a
shattering effect on  algebraic logic, in many respects. In fact, it changed the history of the development
of the subject, and inspired immensely fruitful research, that involved dozens of publications due to many (algebraic) logicians,
starting from Andr\'eka \cite{Andreka} all the way to Venema \cite{v}.
The conclusions drawn from this result, were that either the
extra non--Boolean basic operations of cylindrifiers and diagonal elements were, due to possibly a historical accident,  not properly chosen, or that the notion of
representability was inappropriate; for sure it was concrete enough, but perhaps this is precisely the reason; it is far {\it too concrete.}
Research following both paths, either by changing the signature or/and altering the notion of concrete representability have been pursued
ever since, with amazing success.  Indeed there are  two conflicting but complementary facets
of such already extensive  research referred to in the literature, as `attacking the representation problem'.
One is to sharpen Monk's result proving {\it negative} (non--finite axiomtizability) results,  
the other is try to avoid it, proving {\it positive} results.

In the first path one delves deeply in investigating the complexity of potential axiomatizations for existing varieties
of representable algebras.
Splitting techniques  \cite{Andreka} proved efficient to show {relative non--finite axiomatizability results} in the following sense:
Let $\K$ be a variety having signature $t$, and let $\sf L$ be a variety having signature $t'\subseteq t$, such that if 
$\A\in \K$ then the reduct of $\A$ obtained by discarding the operations in $t\sim t'$, $\Rd_{t'}\A$ for short,  is in $\sf L$.
We say that a set of first order formulas
$\Sigma$ in the signature $t$ {\it axiomatizes $\K$ over $\sf L$}, if  
for any algebra $\A$ in the signature $t$ whenever 
$\A\models \Sigma$ and $\Rd_{t'}\A\in \sf L$, then $\A\in \K$. This means that $\Sigma$ `captures' the properties of the operations in $t\sim t'$. 
A {\it relative non--finite axiomatizability result} is typically of the form:  There is no set `of a special form' of 
first order formulas satisfying a `finitary condition' that axiomatizes
$\K$ over $\sf L$.  Such special forms may be equations, or universal formulas. By finitary, we exclusively mean that $\Sigma$ is finite (this makes no sense if the signature at hand 
is infinite), or $\Sigma$ is a finite schema in the sense of Monk's schema \cite{Monk}, \cite[Definition 5.6.11-5.6.12]{HMT2},  
or $\Sigma$ contains only finitely many variables. The last two cases apply equally well to varieties having infinite signature
like $\RCA_{\omega}$.

A typical result of the last form is that for $2<n<\omega$, there is no set of universal formulas containing 
only finitely many variables that axiomatizes 
the variety of representable (quasi--) polyadic equality algebras of dimension $n$ over the variety 
of representable (quasi--)polyadic equality algebras of the same dimension \cite{Andreka}, a result that can lifted 
to the transfinite. We give (more than) an outline of  the proof in the infinite dimensional case deferred to an appendix. 

\subsection{Second path}

The second path is to try and sidestep
such wild unruly complex axiomatizations of $\RCA_{\alpha}$ for $\alpha>2$, often referred to as {\it taming methods}.
Those taming methods can either involve passing to (better behaved) expansions of the algebras considered,
or even completely change the signature  baring in mind that the essential operations like cylindrifiers are
term definable, or else change the very  notion of representability involved, as long as it remains concrete enough.
The borderlines are difficult to draw, we might not know what is {\it not} concrete enough, but we can
judge that a given representability notion is satisfactory, once we have one.
This type of investigations, when one seeks to find a variety of `representable algebras'  that is {\it finitely axiomatizable by a recursive set of equations},
is known in the literature as the {\it finitizability problem} \cite{n, Bulletin, Sain, mlq}, to be dealt with below in some detail.

To get a grasp of how  difficult the {\it representability problem for $\CA_{\alpha}$s seemed to be in the late sixties of the last century
we quote  Henkin, Monk and Tarski \cite[pp.416]{HMT1} 
addressing the problem for the time being from the `algebra side':

`{\it There are two outstanding open problems, one of them is the problem of providing a simple
intrinsic characterization for all representable $\CA$s, the second problem is to find a notion
of representability for
which a general representation theorem could be obtained which
at the same time would be close to geometrical representation in the concrete character and intuitive simplicity.
It is by no means clear that a satisfactory solution
of either of these problem will ever be found or
that a solution is possible'!} (Our exclamation mark).}

{\bf Classical (Tarskian) representation:} Let $\alpha$ be an ordinal. Recall that ${\sf Cs}_{\alpha}$ denotes the class of {\it set algebras of dimension $\alpha$}, that is, 
algebras with top elements a {\it cartesian square}; a set of   
the form $^{\alpha}U$ for some non--empty set $U$. When $\alpha<\omega$, set algebras are simple, and the class
$\sf RCA_{\alpha}$ is by definition ${\bf SP}\Cs_{\alpha}$.
The definition of representability, without any change 
in its formulation, is extended to algebras of infinite dimension. 

In this case, however, an intuitive justification is less clear
since cylindric set algebras of infinite dimension 
are not in general subdirectly indecomposable. 
In fact, for $\alpha\geq \omega$
no intrinsic property is known which singles 
out the algebras isomorphic  to ${\sf Cs}_{\alpha}$s
among all representable $\CA_{\alpha}$'s, as opposed to the finite dimensional case
where such algebras can be intrinsically characterized by the property of being simple; 
for $\alpha<\omega$, $\A\in \sf RCA_{\alpha}$ is simple $\iff\A\in \bold I{\sf Cs}_{\alpha}$ where $\bold I$ denotes the operation of closing under isomorphic copies.

But in any case (even for $\alpha\geq \omega)$, members of $\RCA_{\alpha}$
can be still represented as algebras consisting of genuine 
$\alpha$--ary relations over
a disjoint union of Cartesian squares, the class consisting of all such algebras
is denoted by ${\sf Gs}_{\alpha}$, with ${\sf Gs}$ standing for 
{\it generalized set algebras}; these are algebras whose top elements 
are {\it disjoint unions} of cartesian squares (of dimension $\alpha$). 

The class of generalized cylindric set algebras, 
just as that of ordinary cylindric set algebras,
has many features that make it well qualified to represent
$\CA_{\alpha}$. The construction of the algebras in this (bigger) class 
retains its concrete character, 
all the fundamental operations and distinguished elements
are unambiguously defined in set-theoretic terms, and the definitions are uniform
over the whole class; 
geometric intuition underlying the construction 
gives us good insight into the structures of the algebras.
Thus there is (geometric) 
justification that 
${\sf Gs}_{\alpha}$ consists of the standard models
of $\CA$--theory.  Its members consist of {\it genuine $\alpha$--ary} relations, and the operations are {\it set--theoretically concretely 
defined} utilizing the form of these relations as sets 
of sequences. But contrary to what was hoped for, by the results of Andr\'eka and Monk \cite{Monk, Andreka} any axiomatization of this class of
(representable) algebras has be infinite and extremely complex if $\alpha>2$. The axioms stipulated by Tarski were not enough to obtain a Stone like representability result.
{\it Infinitely many} more axioms were needed.
We will see later in \S7, cf. theorem \ref{c},  
that removing the condition of {\it disjointness} is highly rewarding, 
in the context of the so--called finitizability problem, where we can (and will) obtain a Stone--like representability result expressed 
via a finite set of equations enforcing representability.
The last procedure is an instance of the technique of {\it relativization}.

{\bf Relativization:} One can find well motivated appropriate notions of semantics (a notion of representation) by first locating them while giving up classical semantical prejudices.
It is hard to give a precise mathematical underpinning to such intuitions. What really counts at the end is a completeness 
theorem stating a natural fit between chosen intuitive concrete-enough, but perhaps not {\it excessively} concrete,
semantics and well behaved, hopefully recursive, axiomatizations.
Go\"dels completeness theorem ties just one choice of predicate logical validity in `standard set theoretic modelling'.  
One could be impressed by the beauty of ensnaring `intuitive validity' by means of exact 
mathematical notions.  But on the other hand, one can argue that 
from equally natural semantic points
of view, other logical equilibria
arise with different sets of validities.

Fix the dimension $n$ to be finite $>1$. A technique that proved extremely potent  in obtaining positive results in both algebraic and modal logic 
is that of {\it relativization}, the slogan being {\it relativization turns negative results positive}. The technique originated 
from research by Istvan N\'emeti  dedicated to the class
of {\it relativized set algebras} whose top elements are arbitrary sets of $n$--ary sequences, $n<\omega$, that are 
not necessarily {\it squares} of the form $^nU$, where $U$ a non--empty set and $n<\omega$ is the dimension.
So the top element of a relativized set algebra of dimension $n$ is an {\it arbitrary subset}  $V$ of ${}^nU$ with operations defined like cylindric set algebras of dimension $n$, but 
relativized to $V$. 
N\'emeti proved, in a seminal result,  that the universal theory of such algebras is decidable.
From the modal perspective, such top elements are  called {\it guards} `guarding the semantics'.
The corresponding multimodal
logic exhibits nice modal behaviour and is regarded as the base for
proposing the so--called {\it guarded fragments} of first order logic by Andr\'eka et al. \cite{v}. 
Relativized semantics  has led to many other nice modal
fragments of first order logic, like the 
loosely guarded, 
clique guarded and packed fragments of first order logic \cite{v}, \cite[Definitions 19.1, 19.2, 19.3,  p. 586-589]{HHbook}.
The most severe relativization is that one deals with set algebras whose top element $V$ 
is an arbitrary set of $n$--ary sequences and operations defined like cylindric set algebra 
of dimension $n$ relativized to $V$ (dealing with the class $\sf Crs_n$).
Less severe relativizations are obtained by imposing certain closure properties on the top element $V$ 
without losing the positive properties of ${\sf Crs_n}$ like e.g. the decidability of its equational theory. 
We say that $V\subseteq {}^nU$ is {\it locally square} if whenever $s\in V$ and
$\tau: n\to n$, then $s\circ \tau\in V$. Let ${\sf G}_n$ be the class of set 
algebras whose top elements are locally square and operations are defined like cylindric set algebra 
of dimension $n$ relativized to the top element $V$, together with with the unary substitution operators denoted by 
${\sf S}_{[i, j]}$ $(i, j<n)$, where for $X\subseteq V$, then ${\sf S}_{[i,j]}X=\{s\in V: s\circ [i, j]\in V\}$, where $[i, j]$ is the transposition that swaps 
$i$ and $j$. So ${\sf G}_n$ has the same signature as {{\it  polyadic equality algebras of dimension $n$}.

It is proved in \cite{f} that the class ${\sf G}_n$ is a finitely axiomatizable variety obtaining a Stone--like representability 
result for algebras of $n$--ary relations. In \cite{mlq} this result is reproved using games. 
A precursor to the result in \cite{f} is the classical Andr\'eka--Resek--Thompson 
result \cite{AT} which says that every $n$--dimensional algebra that has the same signature as $\CA_n$, satisfying a certain finite set of equations together
with the the so--called {\it merry go round identies}, briefly, {\sf MGR} is representable by set algebras whose
top elements are {\it diagonizable} in the following sense: 
If $V$ ($\subseteq {}^nU$)  is the top element of  a given set algebra
of dimension $n$, then $s\in V\implies s\circ [i|j]\in V$ where $[i|j]$ is the replacement that sends $i$ to $j$ 
and keeps everything else fixed. As before, the operations are like those of cylindric set algbras of dimension
$n$  relativized to $V$.
The procedure of weakening the threatening condition of commutativity of cylindrifiers {\it globally}, like in the case of {\it several 
relativized set algebras including 
${\sf Crs}_n$ and ${\sf G}_n$}
($1<n<\omega$)  (thereby guarding semantics)
proved highly rewarding. In \cite{mlq} it is shown that the universal theories of such varieties is decidable.  
This technique of relativization also works for infinite dimensions as will be shown in theorem \ref{c}.

In all cases finite and infinite dimensions, one can obtain positive results on finite axiomatizabily, atom--canonicity, 
canonicity, and 
complete representations 
for the hitherto obtained varieties of representable (modal) algebras, cf. \cite{mlq}.
So far, decidability of the universal theory of the variety of representable algebrasis known only for finite dimensions as shown next.
Recall that $\mathfrak{B}(V)$ is the Boolean algebra $(\wp(V), \cup, \cap, \sim, \emptyset, V)$.
We denote the the class of set algebras of the form
$(\mathfrak{B}(V), {\sf C}_i, {\sf D}_{ij})_{i, j<n}$ where $V$ is diagonalizable by ${\sf D}_n$ and that consisting of algebras of the form 
$(\mathfrak{B}(V), {\sf C}_i, {\sf D}_{ij}, {\sf s}_{[i,j]})_{i, j<n}$ where $V$ is locally square by ${\sf G}_n$.
\begin{theorem}\label{m}
\cite{AT, f, Fer, ans}.  
Fix  $n>1$. Then  ${\sf D}_n$ and ${\sf G}_n$ are varieties that are axiomatizable by a finite schemata. 
In case $n<\omega$, both varieties are finitely axiomatizable and have a decidable universal (hence equational) theory.
\end{theorem}
\begin{proof} 
We prove only 
decidability of the universal theory of ${\sf G}_n$. The proof depends on the decidability
of the loosely guarded fragment of first order logic.
 For $\A\in {\sf G}_n$,   let $\L(\A)$ be the first order signature consisting of an
$n$--ary relation symbol for each element
of $\A$. 
Then we show that for every $\A\in {\sf G}_n$, for any  $\psi(x)$ a quantifier free formula of the signature of $\sf G_n$
and $\bar{a}\in \A$ with $|\bar{a}|=|\bar{x}|$, there is a loosely guarded $\L(\A)$
sentence $\tau_{\A}(\psi(\bar{a}))$ whose relation symbols are among $\bar{a}$
such that for any relativized representation $M$ of $\A$,
$\A\models \psi(\bar{a})\iff M\models \tau_{\A}(\psi(\bar{a}))$. 

Let $\A\in \sf G_n$ and 
$\bar{a}\in \A$. We start by the terms. Then by induction we complete the translation to quantifier free formulas.
For any tuple $\bar{u}$ of distinct $n$ variables, and term $t(\bar{x})$ in the signature
of $\sf G_n$,
we translate $t(\bar{a})$ into a loosely guarded formula
$\tau_\A^{\bar{u}}(t(\bar{a}))$ 
of the first order language having signature
$L(\A)$.
If $t$ is a  variable, then $t(\bar{a})$ is $a$ for some $a\in \rng(\bar{a})$, and we
let $\tau_\A^{\bar{u}}(t(a))=a(\bar{u}).$
For ${\sf d}_{ij}$ one sets
$\tau_{\A}^{\bar{u}}(t)$
to be ${\sf d}_{ij}^{\A}(\bar{u})$ and the constants $0$ and $1$ are handled analogously.
Now assume inductively that $t(\bar{a})$ and $t'(\bar{a})$
are already translated.
We suppress $\bar{a}$ as it plays no role here. For all $i, j<n$ and $\sigma:n\to n$, define
(for the clause ${\sf c}_i$,  $w$ is a new variable):
\begin{align*}
\tau_{\A}^{\bar{u}}(-t)&=1(\bar{u})\land \neg \tau_{\A}^{\bar{u}}(t),\\
\tau_{\A}^{\bar{u}}(t+t')&=\tau_{\A}^{\bar{u}}(t)+\tau_{\A}^{\bar{u}}(t),\\
\tau_{\A}^{\bar{u}}({\sf c}_it)&=1(\bar{u})\land \exists w[(1(\bar{u}^i_w) \land \tau_{\A}^{{\bar{u}^i_w}}(t)],\\
\tau_{\A}^{\bar{u}}({ \s}_{\sigma} t)&= 1(\bar{u}) \land(\tau_{\A}^{\bar{u}\circ\sigma}(t)).
\end{align*}
Let $M$ be a relativized representation of $\A$,
then $\A\models t(\bar{a})=t'(\bar{a})$
$\iff M\models \forall\bar{u} [\tau_{\A}^{\bar{u}}(t(\bar{a})) \longleftrightarrow  \tau_{\A}^{\bar{u}}(t'(\bar{a}))].$
For terms $t(\bar{x})$ and $t'(\bar{x})$ and $\bar{a}\in \A$, choose pairwise distinct variables
$\bar{u}$, that is for $i<j<n$, $u_i\neq u_j$
and define
$\tau_{\A}(t(\bar{a})=t'(\bar{a})): = \forall \bar{u} [1(\bar{u})\to (\tau_{\A}^{\bar{u}}(t(\bar{a}))\longleftrightarrow
\tau_{\A}^{\bar{u}}(t'(\bar{a})))].$
Now extend the definition to the Boolean operations as expected, thereby completing the translation of any quantifier free formula
$\psi(\bar{a})$ in the signature of $\sf G_n$  to the $L(\A)$ formula 
$\tau_{\A}(\psi(\bar{a}))$.

Then it is easy to check that, for any
quantifier free formula $\psi(\bar{x})$ in the signature of $\sf G_n$ 
and $a\in \A$, we have:
$$\A\models \psi(\bar{a})\iff M\models \tau_{\A}(\psi(\bar{a})),$$
and the last is a loosely guarded $\L(\A)$  sentence.
By decidability of the loosely guarded fragment
the required result follows. 
\end{proof}

{\bf Legitimacy of relativization:} A mathematical theory of `meaning' leads 
to new insights and clarifications. Virtually all logics considered by logicians permit a semantical approach. 
In some cases, such as intuitionistic logic, there are philosophical reasons to prefer one semantics over another. 
But research in algebraic logic has shown that this is the case too with {\it  classical and modal logic}, and furthermore, in these two cases, 
the motivation of altering semantics  is more often than not {\it more basic and practical}.  
The move of altering semantics (the most famous is Henkin's move changing second order semantics)
has radical philosophical repercussions, taking us away from the conventional
Tarskian semantics captured by G\"odel--like axiomatization. The latter completeness proof is effective
but highly undecidable; and this property is inherited by finite variable fragments
of first order logic when the number of variables present in the signature is at least three, as long as we insist on Tarskian (square) semantics.

For higher order logics  Henkin's approach uses general models with restricted ranges of quantification. 
Standard models remain a limiting case where all mathematically possible sets
are present. This makes higher order logic many sorted first order logic treating
individual, sets and predicates a par.
The move from higher order to first order brings clear gains in complexity. But though effectively axiomatizable predicate
logic is undecidable by Church's theorem. We can gain more by {\it relativization}. By relativization  we experience immediate practical benefits,  
obtaining quite expressive {\it decidable} 
quantifier logics. This dynamic viewpoint enriches and unifies our view of a mutiplicity of disciplines 
sharing a cognitive slant.
Standard predicate logic has arisen historically by making several ad--hoc semantic decisions
that could well have gone in a different way. Its not all about {\it one canonical}  completeness theorem but rather about {\it several} completeness theorems
obtained by varying both the semantic and syntactical parameters seeking an optimal fit.
This can be implemented from a classical relativized representability theory,
like that adopted in the monograph \cite{HMT1}, though such algebras were treated in {\it op.cit} off main stream,
and they were only brought back to the front of the scene by the work of Resek, Thompson,
Andr\'eka \cite{AT} and last but not least Ferenczi \cite{Fer}.

{\bf Defining the notion of a representation:} 
Inspite of the infinite discrepancy between abstract and concrete algebras, 
there are (other) means  to control the `notion of a representation'. Representations of an algebra can be described in {\it a first 
order theory in a two--sorted language.} 
The first sort in a model of this {\it defining theory}
is the algebra itself, while the second sort is a representation of it. 
The defining theory specifies the relation between the two, 
and its axioms depend {\it on what kind of representation}, be it relativized, ordinary, complete, etc,  is considered.

{\it Thus the representable algebras are those models of the first sort 
of the defining theory, with the second sort providing the representation.}

Many classes in algebraic logic can be seen
this way. The defining theory is usually finite, simple 
and essentially recursively enumerable; and if we are lucky it will be recursive. 
The class of all structures that arise as the first sort of a model of a two-sorted 
first order theory is an old venerable notion in model theory introduced 
by Maltsev
in the forties of the 20th century. Ever since it was studied by Makkai and others. 
It is known as a pseudo--elementary class.  
What is meant here is  
a $PC_{\Delta}$ class \cite{HHbook} 
but expressed in a two sorted language. 
The term pseudo--elementary class strictly 
means $PC_{\Delta}$ 
when the second sort is empty, but the two notions were proved to be 
equivalent by Makkai \cite{HHbook}.
Any elementary class is pseudo--elementary, 
but the converse is not true;  
the class of $\alpha$--dimensional neat reducts
of $\beta$--dimensional cylindric algebras for $1<\alpha<\beta\cap \omega$
is an example \cite{Sayedneat}. 
Other examples are the class of strongly representable atom structures 
and the class of completely representable algebras, both of finite dimension $>2$,  as proved by Hirsch and 
Hodkinson \cite{Sayedneat, HH, strongca},  

According to Hirsch and Hodkinson \cite{HHbook} a 
fairly but not completely general {\it definition} of 'a notion of representation'
is just the second sort of a model of a two-sorted 
(more often than not recursively enumerable) first order theory, 
where the first sort of the theory
is the algebra.  Put in this form,  Hirsch and Hodkinson apply model--theoretic finite forcing
to representing various abstract classes of algebras, using possibly relativized representations. 
Model--theoretic forcing,  
with precursors Henkin's proof of the classical completeness
theorem for first order logic and its algebraic version; the 
neat embedding theorem (proved also by Henkin),  typically involves constructing a model of a first order theory by a {\it game.}

{\bf Axiomatizing the class of representable algebras using games:} 
Games are a highly structured activity with powerful concrete intuitions concerning moves and strategies. 
This view has been around as an undercurrent in logic for
quite some time, 
emerging  often in various contexts.
The game builds the model  with  elements being produced by the second player
called $\exists$ (verifier), in his response to criticism by the first 
player, called $\forall$ (falsfier).
The falsifier draws objects from the domain of discourse which can be tested for certain facts.
Semantical notions of validity serve as a touch--stone of adequacy for proof theoretic or game theoretic proposals,
but the latter provides more vivid ideas and elaborate concepts
about structuring of arguments and procedures
of reasoning. The analysis of winning strategies of the players involved
during the play  
gives fine and delicate information, and allows one to delve deeper into the analysis.
The approach of Hirsch and Hodkinson is basically to combine the forcing games with 
the two sorted approach mentioned above to representations.

For example, games can be used to give an explicit (necessarily 
infinite) recursive axiomatization of the class $\RCA_n$ for $2<n<\omega$. We give a rough outline now.
The idea is that a \ws\ for \pe\ in a $k$--rounded
game for $k<\omega$,  call it $\bold G_k$, can be coded in a first order sentence $\rho_k$ in the signature of $\CA_n$. 

This game is played on so-called {\it networks on $\A$}. A pre-network $N$ on $\A$ is a map 
$N: {}^n\Delta\to \A$, where $\Delta$ is a finite set of nodes, denoted by $\nodes(N)$.  A network is roughly a `finite approximation to a representation' 
so it  is a pre--network that satisfes certain consistency conditions. For example, for $\bar{x}, \bar{y}\in {}^n\Delta$, 
$N(\bar{x})\leq {\sf d}_{ij}\iff x_i=x_j$ and $\bar{x}\equiv_i \bar{y}$ $\implies$ $N(\bar{x})\cdot {\sf c}_iN(\bar{y})\neq 0$. 
The $k$--rounded game between \pa\ and \pe\ where the board of play 
consists of networks on $\A$ is denoted by $\bold G_k(\A)$.
Suppose that we are at round $0<t<\omega$ and $t<k$.  Then $\forall$ challenges \pe\ by 
choosing a previously played network $N_t$ an edge $\bar{x}$ of $N_t$ an index $i<n$ and $a\in \A.$ 
\pe\ can either {\it reject this move} by playing the pre--network $N_{t+1}$ that is like $N_t$ except that $N_{t+1}(\bar{x})=
N_t(\bar{x})\cdot -{\sf c}_ia$.
If she accepts, then \pe\ has to respond with a network $N_{t+1}$ 
{\it refining $N_t$; in this repect being a finer finite approximation to a representation}. In particular, $N_{t+1}\supseteq N_t$ (as maps). 
In her response, \pe\  plays a pre-network defined as follows: 
$\nodes(N_{t+1})=\nodes(N_t)$ plus a new node $z$. Let $\bar{z}$ be given by $\bar{z}\equiv_i\bar{x}$, $z_i=z$. The labels of 
the $n$--tuples (hyperedges) of $\nodes(N_{t+1})$ are given by
\begin{itemize}
\item $N_{t+1}(\bar{z})=a\cdot \prod_{j,k<n, z_j=z_k}{\sf d}_{jk}$

\item $N_{t+1}(\bar{x})=N_t(\bar{x})\cdot {\sf c}_ia$

\item $N_{t+1}(\bar{y})=N_t(\bar{y})$ for all $y\in {}^nN_t\sim \{\bar{x}\}$

\item $N_{t+1}(\bar{y})=\prod_{j,k<d, y_j=y_k}{\sf d}_{jk}$ for all $y\in {}^nN_{t+1}\sim (\{\bar{z}\}\cup {}^nN_t).$
\end{itemize}
If $N_{t+1}$ is a {\it network}, in this case we say that \pe\ survives round $t$. 
If \pe\ survives every round $t< k$, 
then \pe\ wins; otherwise \pa\ wins. 
There are no draws.
For a countable algebra $\A$, $\A\models \rho_k$ $\iff$ \pe\ has a \ws\ in $\bold G_k(\A)$. 
One can translate the set of such  sentences $\{\rho_k: k\in \omega\}$ to a set 
of equations $\Sigma$ using that $\RCA_n$ is a discriminator variety. If $\A$ is countable and $\A\models \Sigma$, then this 
means that \pe\ has a \ws\ (expressed by) $\rho_k$ in $\bold G_k(\A)$ for all $k\in \omega$. 
Using a compactness argument, one shows that \pe\ has a \ws\ in $\bold G_{\omega}(\A)$.
Then \pe\ can use her \ws\ in 
$\bold G_{\omega}(\A)$, 
to build a representation of $\A$. 
The countability condition here is essential, so that all possible moves by \pa\ can be scheduled
and it can be assumed that \pe\ has succeeded to respond to 
every possible move (challenge) by \pa. 

\pe\ builds the required representation as the `limit' of the play follows: Consider a play
$N_0\subseteq N_1\subseteq \ldots $ of $\bold G_{\omega}(\A)$ in which $\exists$ uses his \ws\
and $\forall$ plays every possible legal move.
The {\it limit} of the play is a representation of $\A$ defined as follows: 
Let $N=\bigcup_{t<\omega}N_t$ and define $h:\A\rightarrow\wp(N)$ as follows:
$$h(a)=\{\bar{x}\in N: \exists t<\omega(\bar{x}\in N_t \& N_t(\bar{x})\leq a)\}.$$

One then readily deduces that ${\sf Mod}\Sigma ={\sf RCA}_n$, 
for if $\C\models \Sigma$, then by 
the downward Tarski--\ls\ theorem there is an elementary subalgebra of $\C$, $\A$ say, such that
$\A\models \Sigma$.
By the above reasoning $\A$ is representable.
But $\RCA_n$ is a variety, hence closed under elementary equivalence,  
so $\C$ is representable, too, cf. \cite[\S.8.3]{HHbook}.
A concrete instance of such games will be encounterd in theorem \ref{b}.

{\bf Representing all $\CA_n$s (after all) by twisting set algebras?} To understand the ``essence'' of representable algebras, one often deals
with the non-representable ones, the ``distorted images'' so to speak. 
Simon's result in \cite{Simon}, of ``representing" non-representable algebras, 
seems to point out that this distortion is, after all, not completely chaotic.   
This is similar to studying non-standard models of arithmetic, 
that do shed light on the standard
model. From the main result in \cite{AT} (the famous Andr\'eka--Resek--Thompson theorem), 
it follows that if $2<n<\omega$, and $\A\in \CA_n$
satisfies the so--called  {\it merry go round identities}, $\sf MGR$ for short, 
then $\A$ is representable as a relativized set algebra having top element 
a set
$V\subseteq {}^nU$ such that if $s\in V$, and $i, j<n$ then $s\circ [i|j]\in V$. 

Let $\alpha$ be an ordinal $\geq 3$. Instead of asking ``What is missing from $\CA$'s to be representable?", Henkin turned around the question and asked how much set 
algebras needed to be distorted to provide a representation of all $\CA_{\alpha}$'s. And, strikingly, the anwser provided by Simon is 
``not very much", at least for the lowest value of $\alpha$, 
for which Monk's seminal result and its improvements apply, namely, 
$\alpha=3$.
Simon \cite{Simon} proved that any abstract 
$3$-dimensional cylindric algebra satisfying 
$\sf MGR$ can be obtained from a ${\sf Cs}_3$ (a cylindric set algebra of dimension $3$) 
by the methods called {\it twisting and
dilation}, studied in
\cite[pp. 86--91]{HMT2}. This, in a way, does add to our understanding of the distance between
the abstract notion of cylindric algebra 
and its concrete one, at least in the case of dimension $3$. 
However, Simon had to broaden Henkin's notion of twisting to exhaust the class 
$\CA_3$.  He also showed that Henkin's more restrictive notion of twisting 
does not fit the bill; there are abstract $\CA_3$'s satisfying the $\sf MGR$
that  cannot be obtained by the methods of 
relativization, dilation and twisting, the latter 
understood in the sense of Henkin. Simon's twisting is a stronger ``distortion" of the original algebra, and so its scope
is wider, it can ``reach" more algebras. 
The analogous problem for higher dimension is still an intriguing open problem. The idea is that given any abstract $\A\in \CA_3$ 
then there  is a set algebra $\C\in {\sf Cs}_3$, such that $\A$ can be obtained from $\C$ by applying a finite 
number of operations among 
twisting, dilating, relativizing and forming subalgebras, not neceassrily all (e.g. if $\A\in \sf Cs_3$), 
and not necessarily in this order. 

Let $\Ca_{\alpha}$ denote the class of $\CA_{\alpha}$ atom structures. Recall that the the idea behind dilations
is that in a $\CA$ if it is an atom can be inserted in a certain position, as long as its addition does not contradict the $\CA$ axioms.
Twisting, originally, consists of starting from a complete atomic $\CA_\alpha$ $\A$, selecting atoms $a,b\in \A$ 
and an ordinal $k<\alpha$ and then redefining ${\sf c}_k$ on 
$a$ and $b$ by interchanging the action of ${\sf c}_k$ on $a$ and $b$, in part, ``twisting".
Twisting is used to ``distort" atom structures. It produces $\Ca_{\alpha}$'s from 
$\Ca_{\alpha}$'s, and it typically kills $\sf MGR$. 
However, in some circumstances it {\it can reproduce $\sf MGR$!}.
In both dilation and twising, one starts out with a complete and atomic $\CA_{\alpha}$s, adjoins new elements and /or changes the operations to get a new, 
complete atomic $\CA_{\alpha}$ with certain prescribed 
properties. 

Now here is an ouline of how to `represent' every $\CA_3$ \cite{Simon}: 
One starts with a $\CA_3$ $\A$ and checks if $\sf MGR$ holds in it. If it does then by the 
Andr\'eka--Resek-Thompson theorem \cite{AT}, 
$\A$ is a subalgebra of a relativized $\Cs_3$, and we are done. Else, 
one embeds $\A$ into its complete and atomic canonical extension $\A^+$ 
in order to be able to repair the failure of $\sf MGR$ by twisting the atom structure of $\A^+$. 
But one has to apply dilation first by inserting atoms 
to adjust the twisting parameters, so that twisting can be applied in the first place.
Twisting is then applied to the algebra $\B$ obtained by dilating $\A^+$ to get a $\CA_3$, $\A_3$ in which  $\sf MGR$ holds.
Then one uses the Andr\'eka-Resek--Thompson result to 
represent $\A_3$ as a relativized $\sf Cs_3$.
So here twising is used in a more constructive way; 
by twisting an algebra in which $\sf MGR$ does not hold, we get one where
$\sf MGR$ holds.  

{\bf Since the effect of twisting can always be undone by twisting the twisted algebra, the procedure we have described shows that $\A$ can be
obtained from a subalgebra of a relativized set algebra $\A_3\in \CA_3$ by applying dilation, twisting, relativization 
and the operation of forming 
subalgebras, not necessarily in this order.}

\section{The interaction with modal logic; non--orthodox rules}

Obtaining positive  representability results can also be implemented  
from a modal perspective, in even a more natural way, a task primarily  initiated by Yde
Venema in his dissertation.
This approach has enriched the subject
considerably, dissolving the so robust persistent non--finite axiomatizability 
results in algebraic logic initiated by Monk.
The correspondence between modal logic and algebraic logic allows one, on a first level, to study the algebras 
to understand the deductive systems. But often metalogical properties end up having logical counterparts. 
Though for a long time the algebraic and logical strands have been carried out in
relative isolation, now the tie between 
these two approaches is considerably strengthened. 
Algebraising modal logic allows strong new techniques to bear on metalogical problems. 
Furthermore, algebraic semantics turns out better behaved than Kripke frame based semantics. For example 
there is an (algebraic) completeness theorem  for every normal  modal logic, dually there is a representation 
theorem for Boolean algebras with operators extending Stone's reprsentability result, due to 
Jo\"nsson and Tarski. There is no analogous result for frames.

At a deeper level of this interplay, is e.g the result that is undecidable to tell whether a given finite $\CA_n$, for $2<n<\omega$, is representable 
used to obtain strong undecidablity  results for multimodal logics between the product multimodal logics $\bf K ^n$ and $\bf S5^n$ \cite{k}.
Conversely, in his pioneering paper `Modal Cylindric Algebras' \cite{v}, 
Venema crossed the bridge from the other side. 
Standard Hilbert style axiomatizations, being excluded by Monk's resullt, Venema succeeded to obtain a sound and complete 
proof system for finite variable fragments of first order logic (disguised in a modal formalism) 
with at least three variables;  using instead non--orthodox derivation rules involving the 
so--called difference operator as illustrated in some detail 
next.

On the one hand, one can view modal logics as fragments of first order logic. But on the other hand,  one can also
turn the glass around and give first order logic with $n$ variables  a modal formalism, by viewing assignments as worlds, and 
giving the existential quantifier the most
prominent citizen in first order logic the following familar modal pattern:
$$M, s\models \exists x_i\phi \Longleftrightarrow (\exists s)(s\equiv _i t)\&M, s\models \phi.$$
(Recall that $s\equiv _i t$ means that $s$ and $t$ agree off of $i$, that is
$t(j)=s(j)$ for all $j\neq n$). Here existential quantifiers are viewed as modalities.
A representation theorem for algebras like the one mentioned in theorem \ref{m} is {\bf the dual} of representing abstract state frames by what
Venema calls assignment frames (sets of sequences).  In the aforementioned theorem the frames had domains diagonizable
or  locally square 'set of sequences'.  The Stone--like 
representability result proved for such algebra in \cite{AT, Fer}, and later by games in \cite{mlq}
are in essence algebraic. 

But in some other cases like in the `squareness case' it might be easier to work on {\bf the  frame level}.
Indeed it {\it is easier}. Fix $2<n<\omega$. 

Venema worked on the frame level {\it of squares}, 
namely, atom structures of the form $({}^nU, T_i, D_{ij})_{i, j<n}$,  
where the accessibility relation corresponding to cylindrfiers and diagonal elements 
are defined by: For $s, t\in {}^nU$, $sT_it\iff\s\equiv_i t$ and $D_{i,j}=\{s\in {}^nU: s_i=s_j\}$.
This (modal) view enabled him to use 
non--orthodox derivation rule common in modal logic, 
to get a completeness theorem for finite 
variable fragments of first order logic.

The 
completeness result achieved by Venema for what he calls {\it cylindric modal logic} (a dual formalism of 
finite variable fragments of first order logic),  has its roots in 
the prophetic monograph \cite{HMT2}. 
Such techniques are the {\it frame version} of the notion of so--called {\it richness} 
\cite[Definition 3.2.1]{HMT2}. A {\it rich} algebra has a rich 
supply of elements, that are sufficient to satisfy
a certain set of equations \cite[Lemma 3.2.3, Theorem 3.2.5]{HMT2} to be recalled in a moment.  
And it is precisely the (dual of the) richness condition that can be transformed into  
a non--othodox derivation rule as will be shown in a moment. But such a rule, as it stands on its own,  is only {\it sound}, but it is also
`potentially complete'.  One  can add only finitely many axioms to get completeness. 
From the modal point of view  the unorthodox completeness theorem is based on a {\it special characterization} of the $n$--squares  $^nU$ (domain the Kriple frames). 

{\it The starting point for this characterization is the observation that {\bf the inequality relation} on such $n$--squares  can be obtained in
as a certain composition of the accessibility relations (corresponding to cylindrifiers and diagonal elements), 
using the fact that the difference operator on squares is term definable in a `nice' way.}  

A relation $R^n$ as {\it a new  accessibilty relation} is added to abstract 
Kripke frames (atom structures of $\CA_n$s). Its aim is to enforce representability of the complex algebra based on this frame. 
This relation cn be modally expressible (else there would have been no problem!), but it can be reflected by 
an operator $D_n$ which acts as {\bf a difference operator}. 

{\it The crucial key idea here is that 
a frame is representable as a square $\iff$ $R^n$ is an inequality relation}. 

This is a complicated 
condition, that as indicated cannot be modally expressed, but it can be `captured' if we 
consider  the {\it difference
operator} $D$. This operater $D$ increases 
the expressive power of modal languages. 
For example, it can express frame properties that are 
not expressible modally, such as  irreflexitivity via $\Diamond p\to Dp$. 
This definability of $R^n$ using a `difference--like operator', denote it by $D_n$, leads to a new 
non--orthodox derivation rule, 
proving completeness. What makes this rule non-orthodox, as opposed to orthox is 
that the property of {\it irreflexitivy of the inequality relation} stimulated by the difference operator is not modally 
definable, but `$D_n$ definable'.  

Now a carefuly scrutiny reveals that $D_n$ is precisely 
the {\it frame version} of the equations in \cite[Theorem 3.2.5]{HMT2}.
Using the notation in \cite{v} (1):
$$D_n\phi=\bigvee_i \bigvee_{j\neq i} 
\Diamond_j({\sf d}_{ij}\land \Diamond_0\ldots \Diamond_{i-1}\Diamond_{i+1}\ldots \Diamond_{n-1}\phi).$$
Here $D_n$ is defined to make $R^n$ its accessibility relation, that is for any $n$--square frame $\M$ and $u\in \M$,
$$\M, u\models D_n\phi\iff (\exists v )(R^n(uv)\& v\models \phi).$$
The equations in the hypothesis of \cite[Theorem 3.2.5]{HMT2} are (2):
$$\bigwedge_{k, l\in n}{\sf c}_k[x\cdot y\cdot {\sf c}_k(x\cdot -y)]-{\sf c}_l({\sf c}_kx
\cdot -{\sf d}_{kl})=0.$$
(2) is the dual of (1)  once one identifies the cylindrifier ${\sf c}_i$ the diamond modality $\Diamond_i$ ($i<n$)
 together with some easy manipulations. 

\subsection{Rectangular density}

Continue to fix $2<n<\omega$. There is another algebraic expression of such  non--orthdox derivation rules  by so--called {\it density conditions}.
This approach also has its roots in the monograph \cite{HMT2}, cf. \cite[Theorem 3.2.14]{HMT2} where 
it is proved that if $\A\in \CA_n$ is an {\it atomic rectangulary 
dense algebra}, that is, $\A$ is atomic and its atoms are 
{\it rectangles}, then $\A$ is representable. 
Here a rectangle is an $n$--dimensional rectangle; it is an  element that satisfies
$\prod_{i<n} {\sf c}_i x=x$. Notice that  $\geq$ always holds, because $x\leq {\sf c}_ix$ for any $i<n$. 
When $n=2$, and the algebra is a set algebra with top element $U\times U$, and $X\subseteq U\times U$
a $2$--dimensional rectangle is just the familiar geometric  
rectangle obtained by cylindrifying 
on both dimensions.
This is a nice sufficient condition for representability; in fact it can be proved that an algebra is representable $\iff$ it can be embdded in 
an atomic one whose atoms are rectangles \cite[Theorem 3.2.16]{HMT2}. The bad news is that
this characterization cannot be expressed in a $\forall\exists$ formula due to the involvement of `atomicity'. 
From the logical point of view this `complicated notion' of atomicity is not warranted, 
because it does not lend itself to derivation 
rules even non--orthodox ones.
The  escape of this impasse 
was accomplished 
by simply {\it removing the condition} of atomicity by Andr\'eka  et al. \cite{AGMNS}, 
and proving that the representability result survives this omission. 
Every {\it rectangulary dense} algebra is reprsentable, where 
by {\it rectangularly density} is meant that below every element there is a rectangle, that is not necessarily an atom.
Algebras considered might not be atomic. 
By the above discussion one cannot help but to ask the purely algebraic question.
How are  the notions of {\it rectangular density} and {\it richness},  related, if at all?

{\it As illustrated next, both are algebraic expressions of a Henkin construction; the last  achieving completeness for $L_{\omega, \omega}$.
But instead they reflect algebraically a complete non--orthodox 
proof system for $L_n$ using a variant of the difference operator.}

The rectangles in a rectangulary dense algebra can be associated with so--called $0$--thin elements \cite[Definition 3.2.1]{HMT2}. 
Such elements have a double facet, a geometric one and a metalogical one.
The metalogical interpretation is that these elements  {\it abstract the notion of  individual constants} \cite[Remark 3.2.2]{HMT2}.
Geometrically, in a set algebra $0$--thin elements are obtained by fixing the first component of assignments by a constant.
That is, if $\A$ is a set algebra with top element $^nU$ say,  and $X\in \A$ is $0$ thin, then 
$X=\{s\in {}^{\alpha}U: s_0=u\}$  for some $u\in U$. {\it Thinnes here means that literally there is  a thin line between the dimension of $X$
and the dimension of $\A$.}
There is  `enough supply' of such elements in a rectangularly dense algebra. 
This algebraic notion of {\it richness} actually  
reflects the notion of {\it rich theories in Henkin constructions}.

Rich theories occuring in Henkin's completeness proof eliminate existential 
quantifiers in existential formulas via individual constants more commonly referred to as {\it witnesses}. 
Algebraically, every cylindrifier is elimintated, or {\it witnessed} by a $0$--thin element \cite[Definition 3.2.1]{HMT2}.
But then by algebraising the rest of Henkin's proof, we get that rich algebras are representable. This connection 
manifests itself blatantly in the proof of \cite[Theorem 3.2.5]{HMT2} where the base of the representation
actually consists only of $0$-thin elements, whereas  
the generic canonical 
models in Henkin constructions consist of individual constants.
More succintly, richness and rectangular 
density are {\it saturation conditions}. Geometrically:  rectangular density 
means that below every non--zero element there is a rectangle, while 
richness means that below every 
element there is  {\it a square}, a special kind of rectangle as the name suggests.
The latter notion strikes one as weaker,
but both notions are sufficiently strong to enforce representability, and both notions {\it can be 
transferred to non--orthodox derivation 
rules using the difference operator} achieving completeness when 
added to the cylindric (modal)
axioms.  In fact, it can (and was) proved in \cite{AGMNS} that both notions are essentially equivalent.
In {\it op.cit} it is proved that $\A$ is rectangularly dense (and quasi--atomic) $\iff$
$\A$ is rich.

For diagonal free cylindric algebras, ${\sf Df}_n$ for short,  one 
cannot express the difference operator because there is no `equality' in the signature, so that
we {\it cannot express non--equality.}  Rectangular density 
can be formulated in the language of $\sf Df_n$ but richness cannot be. 
So here one uses rectangular density to prove representability.
These ideas were implemented by Venema \cite{r} extending the results in \cite{AGMNS} to ${\sf Df}_n$s.
The required representation (completeness theorem) was attained 
using {\it Rectangular games}  played on so--called {\it crystal networks}.
In the present `diagonal--free' context the technique used by Andr\'eka et al. \cite{AGMNS} 
based on \cite[Theorem 3.2.16]{HMT2}  does not work.

Dually, from the (multi) modal logic perspective
one takes the non--orthodox  Gabbay--like irreflexive (density) rule:
$p\land \tau(\neg \phi\land p))\to \phi$, if $p\notin \phi$, then 
deduce $\phi$, \cite[p.1563]{v}, 
where  $$\tau(\chi)=\neg \Diamond_0\ldots \neg \Diamond_{n-1}[(\bigwedge_{i\in n}\Diamond_0\ldots \Diamond_{i-1}
\Diamond_{i+1} \ldots \exists \Diamond_{n-1}\chi)\land
\neg \chi,$$
$\chi$ a formula. 
Together with the $\sf S5$ axioms for each $i<n$ this gives a complete and sound (non--orthodox) 
proof system $\vdash$. This is a {\it translation of the algebraic property rectangular density}, 
in the sense that if $T$ is a theory,  then the Tarski--Lindenbaum  quotient algebra $\Fm_T$, where the quotient 
is taken with respect to $\vdash$ is rectangularly dense, hence representable. 

The representability result `{\it rectangular density in a ${\sf Df}_n$ 
$\implies$ representability}' was proved using games played on networks 
defined next:
\begin{definition} Let $2<n<\omega$ and $\A\in {\sf Df}_n$.
\begin{enumarab}
\item An {\it $\A$ pre--network} is 
 a pair $N=(N_1, N_2)$
where $N_1$ is a finite set of nodes, and  $N_2:N_1^n\to \A$ is a total map. 
$N$ is {\it atomic} if $\rng N\subseteq \At\A$.
We write $N$ for any of $N, N_1, N_2$ relying on context, we write $\nodes(N)$ for $N_1$. 
\item A pre-network $N$ is said to be a network if it satisfies the following consistency condition: for $\bar{x}, \bar{y}\in N$ and $i<n$, 
$\bar{x}\equiv_i \bar{y}\implies N(\bar{x})\cdot {\sf c}_i N (\bar{y})\neq 0.$
\end{enumarab}
\end{definition}
We define a game.  But first a piece of notation. For a function $f$ and $i\in \dom(f)$, $g=f_i^u$ denotes the function 
with the same domain as $f$, such that $f\upharpoonright \dom(f)\sim \{i\}=g\upharpoonright \dom (f)\sim \{i\}$ and 
$g(i)=u$.  
A play of the game consists of a sequence $N_0\subseteq N_1\subseteq \ldots$ of networks so that there are $\omega$ rounds.
Suppose we are at round $t$, with the network $N_t$ the outcome of the play so far. 
$\forall$ makes a move by 
\begin{enumarab}
\item  Choosing $a\in \A$. $\exists$ must respond with a network
$N_{t+1}\supseteq N_t$ such that either $N_{t+1}(\bar{x})\leq a$ or $N_{t+1}(\bar{x})\leq -a$,

\item $\forall$ may choose an edge $\bar{x}$ of $N_t$ an index $i<n$ and $b\in \A$ with $N_t(\bar{x})\leq {\sf c}_ib$.
$\exists$ must respond with a network $N_{t+1}\supseteq N_t$ such that for some $z\in N_{t+1},$
$N_{t+1}(\bar{x}^i_z)=b$.
\end{enumarab}
Fix $2<n<\omega$. Our aim is to show that if $\A$ is countable and rectangularly dense then \pe\ has a \ws\, and so $\A$ is representable.
A rectangle in an $\A\in {\sf Df}_n$, and a rectangularly dense $\sf Df_n$ are defined like above (the definitions are free from diagonal 
elements).

\begin{definition} A network $N$ is said to be {\it rectangular} if $N(\bar{x})$ is a non-zero 
rectangle for every $s\in {}^nM$.
$M$ is a {\it crystal network} if
${\sf c}_iM(s)={\sf c}_iM(t)$ whenever $s\equiv_i t$ for any $i<n$. 
\end{definition}
\begin{lemma}\label{venema}\cite[Lemmas 1, 2, p.1557, p.1560]{r}. 
Let $\A\in \sf Df_n$ be countable and rectangularly dense.
\begin{enumarab}
\item If $M$ is a crystal network, and $r\in \A$ is a non--zero rectangle below $M(t)$ for some $t\in {}^nM$, then there is a 
crystal extension $M'$ of $M$ such that $M'(t)=r$.
\item If $M$ is  crystal network and ${\sf c}_iM(t)={\sf c}_ir$ for some $t\in {}^nM$ 
and some rectangle $r\in \A$, $r>0$, 
then there is a crystal extension $M'$ of $M$ and an element $t'\in {}^nM'$ such
that $t\equiv_i t'$ and $M'(t)=r$.
\end{enumarab}
\end{lemma}
\begin{proof} For the first part define like on \cite[p.1558]{v}, $M'(t)=M(t)\cdot {\sf c}_{(\Delta(t,u))}r,$ where $\Delta(t,u)=\{i<n: t_i\neq u_i\}$.
Then $\nodes(M')=\nodes(N')$.
For the second part: If there is already a tuple in $^nM$ such that $t\equiv _i u$ and $M(u)={\sf c}_ir$, then $M$ is as required.
Else, choose $k\notin M$ and let $K=M\cup \{k\}$ be defined as in \cite[p.1560]{v}. 
\end{proof}
Then it can be shown \cite[Lemma 3]{r}, that in any round of the game, if the last network played is a crystal network
then \pe\ can survive this round responding 
with a crystal network. From this we conclude using the arguments in \cite{r} having at our disposal
lemma \ref{venema}:.
\begin{theorem}\label{b}
Let $2<n<\omega$. Let $\A$ be a countable rectangularly dense $\sf Df_n$. Then \pe\ has a \ws\ in $G_{\omega}$, hence
$\A\in \sf RDf_n$.
\end{theorem}
\begin{proof} Like the arguments used above in constructing a representation as a limit of the play 
on networks  cf. \cite[Theorem 2, Lemma 3]{r} using the crucial lemma \ref{venema}.
\end{proof}

The notions of density can be retrieved `modally' using the difference operator, 
but they cannot be retrieved by finite Hilbert style axiomatizations.
Such  non--orthox derivation rules can be traced back to the work of Gabbay and ever since
have been  frequently used by modal logicians, though some people argue  that they capture extra 
variables or` witnesses' from the back door and this is
inimical to the modal nature.
These extremely liberal Gabbay--style inference systems
typically correspond to classes that are inductive, i.e., axiomatized by 
$\forall\exists$--formulas.
like the class of rectangularly dense
cylindric algebras \cite{AGMNS} and its diagonal free reducts.

{\bf We hasten to add that the `game technique' used by Venema \cite{r} 
works for all other cylindric--like algebras dealt with in \cite{AGMNS} like Pinter's substitution algebras and quasi--polyadic algebras with and 
without equality,  but {\it not} `the other way round.'
The main technique in \cite{AGMNS} which reduces the problem to the  `atomic case' already proved by Henkin et al. \cite{HMT2} (for cylindric algebras), 
does not settle the $\sf Df$ case.} 

But in any case the 
important thing here  is that we have a sound and complete proof system, involving finite many axioms (and non--orthodox rules) 
achieving completeness relative to
square semantics. Such a a completeness theorem  circumvents 
the dominating so resilient
non--completeness theorems (obtained algebraically) 
by Monk, and considerably sharpened and refined 
by pioneers, to name a few, Andr\'eka, Biro, Hirsch, Hodkinson and Maddux \cite{Andreka, HHbook, Maddux, n,  Bulletin, v}.
Next we describe other efficient ways to avoid strong non--finite axiomatizability results.

\section{Finitizability using guarding conjuncted with the semigroup approach} 

We start by a strong incompleteness theorem for the so--called {\it typeless logics} 
cf. \cite[\S 4.3]{HMT2}  indicating a `severe infinite'  mismatch between syntax and the intended classical semantics 
(in any attempt to give an algebraizable formalism of 
$L_{\omega, \omega}$), as long as we insist on full fledged commutativiy of quantifiers. 
Here the adjective `typless' indicates that in the formation rules of formulas, one drops the rank function specifying  the arity of relation symbols; they all have the same rank, namely, $\omega$.
This is necessary if we want to algebraize $L_{\omega, \omega}.$ 

In a while we will manipulate {\it both syntax and semantics}. 
We will expand  the syntax (vocabularly) with finitely many connectives, and reformulate it in such a way that the connectives are {\it finite} without 
{\it losing the  
expressive power of the previous formalism; the other (infinitely many) connectives (like all cylindrifiers  $\exists_i: i<\omega)$) 
are still there; they are definable from the finitely many new connectives}.  We simultaneously {\it guard} (Tarskian square) semantics.
In this way, we 
shall be able to obtain a strong completeness theorem for another algebraizable typless 
variant of $L_{\omega, \omega}$ via a complete finite Hilbert style axiomatization involving only `type free' 
valid formulas. 
In this process, we weaken commutativity of quantifiers and the intended semantics is accordingly `widened'
allowing {\it generalized models} where assignments are restricted to the {\it admissable ones}, in the sense of the coming 
definition \ref{adm}. It is worthwhile sacrifizing (but not fully) 
this `semantical' precarious property of full--fledged commutativity of quantifiers 
seeing as how the harvest we reap is a strong completeness theorem 
for certain {\it tamed} finitary logics of infinitary relations having quite strong 
expressive power and, even more,  
unlike $L_{\omega, \omega}$ possibly having a decidable validity problem.

Let us start from the beginning, namely, by negative results. 
We start by defining  the `classical' more basic algebraizable typeless extension of $L_{\omega, \omega}$ with usual Tarskian square semantics
in the context of an algebraizable logic (in the standrad Blok--Pigozzi sense). Such logics are dealt 
with in \cite[\S 4.3]{HMT2}. Our formulation is slightly more updated.
By a {\it a logical system, a logic for short} we understand a quadruple $(F, \bold K, {\sf mng}, \models)$ where $F$ is a set (of formulas) in a certain signature,
$\bold K$ is a class of structure $\sf mng$ is a function with
domain $F\times \bold K$ and $\models \subseteq F\times F.$
Intuitively, $\bold K$ is the class of structures for the signature at hand 
$\sf mng(\phi, M)$ is the interpretation of $\phi$ in $M$ (possibly relativized),
and $\models$ is the pure semantical
relation determined by $\bold K$.  Such a definition is too broad; now we restrict  a(n algebraisable) logic to be:

\begin{definition} A logic  $(F, \bold K,  {\sf mng}, \models)$ with formula algebra $\F$
of signature $t$ is {\it algebraizable} if:
\begin{enumarab}

\item A set ${\sf Cn}\L$ the logical connectives fixed and each $c\in Cn\L$ finite rank
determining  the signature $t,$

\item There is set $P$ called atoms such that $\F$ is the term algebra or
absolutely free algebra over $P$
with signature $t$,

\item ${\sf mng}_M=\langle {\sf mng}(\phi, M): \phi\in F\rangle$ is an endomorphism on $\F$,

\item There is a derived binary connective $\leftrightarrow$
and a nullary connective
$\top$ that is compatible with the meaning functions,
so that for all $\psi, \phi\in F$, we have ${\sf mng}(\phi)={\sf mng}(\psi)\iff M\models \phi\leftrightarrow\psi$
and $M\models \phi\iff M\models \phi\leftrightarrow \top,$

\item For each endomorphism $h$ of $\F$, $M\in \bold K$, there is an $N\in \bold K$
such ${\sf mng}_N={\sf mng}_M\circ h,$  so that validity is preserved by endomorphisms.
\end{enumarab}
\end{definition}

Item (5) is what guarantees that instances of valid formulas remain valid for a homomorphism applied to a formula
$\phi$ amounts to replacing the atomic formulas in $\phi$ by
formula schemes.
This is a crucial property for a logic to allow {\it algebraization}. 
We are ready to describe the algebraizable modification of $L_{\omega, \omega}$ (first order {\it with equality}.)
To simplify matters, we consider the variables to have order type $\omega$ (the least infinite ordinal).
A {\it signature} is a triple $\Lambda=(\omega, R, \rho)$ such that $R$
and  $\rho$ are functions with common domain a cardinal $\beta\leq \omega$,
and $\rho(i)\leq \omega$ for all $i\in \beta$;
$R_i$ is a relation symbol of arity $\rho(i)$ (i.e the arity can be infinite.)

Fix a signature
having $\omega$ many relation symbols each of arity $\omega$.
The atomic formulas are of the form $v_k=v_l$ or 
$R_l(v_0, \ldots, v_i,\ldots)_{i<\omega}$, $l\in \omega$.
Variables can appear only in their natural order.
Such atomic formulas are called {\it restricted} \cite{HMT1}. For first order logic with equality 
the notion of restricted formulas is only an apparent 
restriction because any (usual) first order formula is (semantically) equivalent to a restricted one. 

For the above typless formalism using restricted formulas,   one can dispense with the use of variables altogether, since they only appear in their natural order. 
No information is lost this way; one mihght as well write simply $R$ instead of the atomic formula $R(x_0, x_1\ldots, x_i\ldots )_{i<\omega}$, 
($R$ a relation symbol of arity $\omega$). 
This formalism, like the case when we have only finitely many variables,  readily lends itself to a (infinite)--dimensional 
propositional modal formalism,  if we view the infinitely many existensial quantifiers as diamonds.

The formation rules of formulas are like (ordinary) first order logic; we use the logical connectives $\land$ for conjunction
and we denote negation by $\neg$. If $\phi$, $\psi$ are formulas, $i<\omega$, then $\phi\land \psi$, $\neg \phi$, and $\exists v _i\phi$, 
briefly $\exists_i \phi$\footnote{$\Diamond_i$ if one prefers the modal formalism} 
are formulas. 

The set of (restricted) formulas is denoted by $Fm_r$.
A structure is pair ${\M}=(M, \bold R)$ where $M$ is a non--empty set, $\bold R:\omega\to {}^{\omega}M$, 
and the interpretation of $R_i$ in $\M$,  in symbols $R_i(\bar{x})^{\M}$ is $\bold R_i$ $(i<\omega$).
The class of al structures is denoted by $\bold K$.For a sequence $s\in {}^{\omega}M$, $\M\in \bold K$, and a formula $\phi$
we write ${M}\models \phi[s]$
if $s$ satisfies ${\phi}$ in $\M$; this too is defined exactly like in first order logic. Having defined the basic semantical notions 
$\models$ is defined the usual way (like first order logic). 
We denote $\{s\in ^{\omega}M: \M\models \phi[s]\}$ by $\phi^{\M}$.
For $\phi\in \Fm_r$ and $\M\in \bf K$,  ${\sf mng}(\phi, \M)=\phi^{\M}$.
Let $\L_{\omega}=(Fm_r, \bold K, {\sf mng}, \models).$
$\L_{\omega}$ is referred to as a {\it finitary logic of infinitary relations.}
For provability we use the basic proof system in \cite[p. 157, \S 4.3]{HMT2} which is a natural extension of a complete calculas for $L_{\omega, \omega}$
expressed in terms of restricted formula.
To formulate our next strong incompleteness result we need.
\begin{definition}
A {\it formula schema} in a logic $\L$ is a formula $\sigma(R_1,\ldots, R_k)$
where $R_1,\ldots, R_k$ are relation symbol.
An {\it $\omega$ instance} or even simply an instance of $\sigma$
is a formula of the form $\sigma(\chi_1, \ldots, \chi_k)$
where $\chi_1, \ldots, \chi_k$
are formulas and each $R_i$ is replaced by $\chi_i$.
\end{definition}
Here  {\it type--free valid formula schema} is a new notion of validity defined by Henkin et 
al. \cite[Remark 4.3.65, Problem 4.16]{HMT2}, \cite[p. 487]{HHbook}. 
\begin{theorem}\label{negative}
\begin{enumarab}

\item For any $k\geq 1$, there is no finite schemata of $\L_{\omega}$ whose set $\Sigma$ of instances satisfies
$$\Sigma \vdash_{\omega+k} \phi\Longleftrightarrow \vdash _{\omega+k+1}\phi.$$

\item For any $k\geq 1$, there is no finite schemata of $\L_{\omega}$ whose set $\Sigma$ of instances satisfies
$$\models \phi\implies \Sigma\vdash_{\omega+k}\phi.$$
\end{enumarab}
\end{theorem}
\begin{proof}
Observe that if $\phi$ is a formula of $\L$ and $\tau(\phi)$
is its corresponding term in the language of $\CA_{\omega}$ as defined in \cite[Definition 4.3.55]{HMT2},
then $\bold S\sf Nr_{\omega}\CA_{\omega+k}\models \tau(\phi)={\bf 1}\Longleftrightarrow \vdash_{\omega+k}\phi.$
Thus the existence of such a schemata
would imply finite axiomatizability by equations of
$\bold S\sf Nr_{\omega}\CA_{\omega+k+1}$ over $\bold S\sf Nr_{\omega}\CA_{\omega+k}$
since every schema translates into an equation in the language of
$\CA_{\omega}$. We show that this cannot happen.
We use (the notation and) idea in \cite[Theorem 3.1]{t}. In fact, we prove more than needed allowing the order type of variables to be any infinite ordinal $\alpha$.
We show that for any  positive $k\geq 1$, the variety $\bold S\sf Nr_{\alpha}\CA_{\alpha+k+1}$ is not axiomatizable by a finite schema 
over $\bold S{\sf Nr}_{\alpha}\CA_{\alpha+k}$. This gives the required result.

We start by the finite dimensional case, then we lift the construction to the transfinite. 
Fix $2<m<n<\omega$. Let $\mathfrak{C}(m,n,r)$ be the algebra $\Ca(\bold H)$ where $\bold H=H_m^{n+1}(\A(n,r), \omega)),$
is the $\CA_m$ atom structure consisting of all $n+1$--wide $m$--dimensional
wide $\omega$ hypernetworks \cite[Definition 12.21]{HHbook}
on $\A(n,r)$  as defined in \cite[Definition 15.2]{HHbook}.   Then $\mathfrak{C}(m, n, r)\in \CA_m$.
Then for any $r\in \omega$ and $3\leq m\leq n<\omega$, $\C(m,n,r)\in {\sf Nr}_m{\sf CA}_n$, $\C(m,n,r)\notin {\bold  S}{\sf Nr}_m{\sf CA_{n+1}}$
and $\Pi_{r/U}\C(m,n,r)\in {\sf RCA}_m$, cf. \cite[Corollaries 15.7, 5.10, Exercise 2, pp. 484, Remark 15.13]{HHbook}.
Take $$x_n=\{f\in H_n^{n+k+1}(\A(n,r), \omega); m\leq j<n\to \exists i<m, f(i,j)=\Id\}.$$
Then $x_n\in \C(n,n+k,r)$ and ${\sf c}_ix_n\cdot {\sf c}_jx_n=x_n$ for distinct $i, j<m$.
Furthermore (*),
$I_n:\C(m,m+k,r)\cong \Rl_{x_n}\Rd_m {\C}(n,n+k, r)$
via the map, defined for $S\subseteq H_m^{m+k+1}(\A(m+k,r), \omega)),$ by
$$I_n(S)=\{f\in H_n^{n+k+1}(\A(n,r), \omega):  f\upharpoonright {}^{\leq m+k+1}m\in S,$$
$$\forall j(m\leq j<n\to  \exists i<m,  f(i,j)=\Id)\}.$$
We have proved the (known) result for finite ordinals $>2$.

To lift the result to the transfinite,
we proceed like in \cite{t}, using the same lifting argument in {\it op.cit}.
Let $\alpha$ be an infinite ordinal. Let $I=\{\Gamma: \Gamma\subseteq \alpha,  |\Gamma|<\omega\}$.
For each $\Gamma\in I$, let $M_{\Gamma}=\{\Delta\in I: \Gamma\subseteq \Delta\}$,
and let $F$ be an ultrafilter on $I$ such that $\forall\Gamma\in I,\; M_{\Gamma}\in F$.
For each $\Gamma\in I$, let $\rho_{\Gamma}$
be an injective function from $|\Gamma|$ onto $\Gamma.$
Let ${\C}_{\Gamma}^r$ be an algebra similar to $\CA_{\alpha}$ such that
$\Rd^{\rho_\Gamma}{\C}_{\Gamma}^r={\C}(|\Gamma|, |\Gamma|+k,r)$
and let
$\B^r=\Pi_{\Gamma/F\in I}\C_{\Gamma}^r.$
Then we have $\B^r\in \bold {\sf Nr}_\alpha\CA_{\alpha+k}$ and
$\B^r\not\in \bold S{\sf Nr}_\alpha\CA_{\alpha+k+1}$.
These can be proved exactly like the proof of the first two items in \cite[Theorem 3.1]{t}. The second part uses that the element $x_n$ is {\it $m$--rectangular} 
and {\it $m$--symmetric}, in the sense that 
for all  $i\neq j\in m$, ${\sf c}_ix_n\cdot {\sf c}_jx_n=x_n$ and ${\sf s}_i^jx_nx \cdot  {\sf s}_j^ix_n =x_n$ (This last two conditions are not entirely indepedent \cite{HMT2}). 
This is crucial to guarantee that in the algebra obtained after relativizing 
to  $x_n$, we do not lose commutativity of cylindrifiers.  The relativized algebra stays inside $\CA_n$.

We know
from the finite dimensional case that $\Pi_{r/U}\Rd^{\rho_\Gamma}\C^r_\Gamma=\Pi_{r/U}\C(|\Gamma|, |\Gamma|+k, r) \subseteq \Nr_{|\Gamma|}\A_\Gamma$,
for some $\A_\Gamma\in\CA_{|\Gamma|+\omega}=\CA_{\omega}$.
Let $\lambda_\Gamma:\omega\rightarrow\alpha+\omega$
extend $\rho_\Gamma:|\Gamma|\rightarrow \Gamma \; (\subseteq\alpha)$ and satisfy
$\lambda_\Gamma(|\Gamma|+i)=\alpha+i$
for $i<\omega$.  Let $\F_\Gamma$ be a $\CA_{\alpha+\omega}$ type algebra such that $\Rd^{\lambda_\Gamma}\F_\Gamma=\A_\Gamma$.
Then $\Pi_{\Gamma/F}\F_\Gamma\in\CA_{\alpha+\omega}$, and we have proceeding like in the proof of item 3 in \cite[Theorem 3.1]{t}:
$\Pi_{r/U}\B^r=\Pi_{r/U}\Pi_{\Gamma/F}\C^r_\Gamma
\cong \Pi_{\Gamma/F}\Pi_{r/U}\C^r_\Gamma
\subseteq \Pi_{\Gamma/F}\Nr_{|\Gamma|}\A_\Gamma
=\Pi_{\Gamma/F}\Nr_{|\Gamma|}\Rd^{\lambda_\Gamma}\F_\Gamma
=\Nr_\alpha\Pi_{\Gamma/F}\F_\Gamma.$
But $\B=\Pi_{r/U}\B^r\in \bold S{\sf Nr}_{\alpha}\CA_{\alpha+\omega}$
because $\F=\Pi_{\Gamma/F}\F_{\Gamma}\in \CA_{\alpha+\omega}$
and $\B\subseteq \Nr_{\alpha}\F$, hence it is representable (here we use the neat embedding theorem).
The rest follows using a standard L\'os argument.

Item (2) follows also from the above proof using the same reasoning since the above proof also gives that $\RCA_{\omega}$ cannot be axiomatized by a finite schema over
$\bold S\sf Nr_{\omega}\CA_{\omega+k}$ for ay $k>0$.
\end{proof}

Henkin, Monk and Tarski formulated the so--called {\it finitizability problem $\sf FP$} this way:

{\it Devise an algebraic version of predicate
logic in which the class of representable algebras forms a finitely based variety}
\cite{Andreka, AT, f,  Sain,  Fer,  Bulletin,  v, r,  HHbook,  mlq, HMT2}.

The $\sf FP$ in a nut shell, seeks a Stone--like representability 
result for algebras of relations having infinite rank, equivalently a strong completeness theorem for
modifications of $\L_{\alpha}$ by either changing the syntax and /or guarding semantics 
obtaining an oprtimal fit}

The method of `guarding semantics' goes back to the seminal 1985 paper of Andr\'eka and Thompson 
\cite{AT} which proves possibly the first  (historically) strong positive result on finite schema axiomatizations of agebras of $\alpha$-- ary 
relations, $\alpha$ an ordinal, intercepting the path of a series of negative non--finite axiomatizabilty results.  For $\alpha<\omega$, 
this schema is a finite set of equations axiomatizing the resulting class of representable algebras  
with diagonizable top elements. The term {\it guarding} though was introduced much later in the seminal paper of Andr\'eka et al. \cite{avn}, where Andr\'eka, V. Benthem, and 
N\'emeti propse the guarded fragments of first order logic. 
The article \cite{v} is an excellent account  on this `fruitful contact between $\sf Crs$'s and guarded logics'.

One way to understand the subtle technique of guarding is to look from above.
One asks oneself what expressive resources will typically lead to undecidability 
and tries to avoid it by manipulating semantics. This view often also diffuses non--finite axiomatizability results.
For the infinite dimensional case we succeed to obtain a strong finite axiomatizability result by guarding semantics, but 
whether this yields decidability of the validity problem with respect to 
this  guarding  is not settled in this paper.

To get rid of the `severe incompleteness' result just proved for the common algebraic formalism of first order logic, 
one  guards semantics in the infinite dimension case, too, 
in analogy to the finite dimensional case.  More specfically, we will 
obtain an exact infinite analogue of the variety ${\sf G}_n$; significantly distinct from ${\sf G}_{\omega}$ in that 
it admits a strictly finite axiomatization. 
We start from the logic side. The following theorem relates the semantics of a (possibly infinitary) formula $\phi$ 
in a {\it generalized model} to the semantics of its {\it guarded version}, denoted by $\sf guard(\phi)$,  in the standard part of the model expanded with the guard. 
Let  $\alpha\leq \omega$. Let $\L_{\alpha}$ denote the algebraizable formalism corresponding to $\CA_{\alpha}$ as defined above allowing a sequence of $\alpha$--many variables, cf.
\cite[\S 4.3]{HMT2}. We dealt with the special $\L_{\omega}$ in the statement of the previous theorem (but 
the proof covered $\L_{\alpha}$  for any 
infinite ordinal $\alpha$.) 
By induction on the complexity of formulas the following can be proved:
\begin{theorem}\label{adm} 
Let $L$ be a signature taken in $\L_{\alpha}$. 
Let $(M, V)$ be a generalized model in $L$,
that  is, $M$ is an $L$--structure and $V\subseteq {}^{\alpha}M$ is the set of {\it admissible assignments}. Assume that $R$ is an $\alpha$--ary
relation symbol outside $L$. For $\phi$ in $L$,  let $\sf guard(\phi)$ be the formula obtained from $\phi$ by relativizing all quantifiers
to one and the same atomic formula $R(\bar{x})$ and let 
${\sf Guard}(M,V)$ be the model expanding $M$ to $L\cup \{R\}$ by
interpreting $R$ via $R(s)\iff s\in V$. 
Then the following holds:  
$$M, V, s\models \phi\Longleftrightarrow {\sf Guard} (M, V),s\models {\sf guard}(\phi),$$
where $s\in V$ and $\phi$ is a formula.

\end{theorem}

For $\alpha$ finite we have already dealt with the $\alpha$--variable guardred fragments of $L_{\omega, \omega}$ 
where the admissable assignments where  sets of the form $V\subseteq {}^{\alpha}U$, $U$ a non--empty set. 
We considered the two cases when $V$ is diagonizable and when $V$ is locally square.
Such choices of `guards' gave the finitely axiomatizable varieties ${\sf D}_{\alpha}$ and ${\sf G}_{\alpha}$ 
whose modal (set) algebras have top elements $V$ as specified, respectively.  
Such varieties have a decidable universal theory, too, witness 
theorem \ref{m} and \cite{AT, Fer, mlq}.  

To get finite axiomatizability we use the so--called 
{\it semigroup approach} in algebraic logic that proved efficient in solving the $\sf FP$ for first order logic without equality.
The choice of the semigroup controls the signature of our algebras expanding the signature of 
$\CA_{\omega}$s. 
One takes only those substitution operators indexed by transformations in this fixed in advance subsemigroup, call it $\T$ 
of $(^{\omega}\omega, \circ)$ to be interpreted in set algebras the usual way like in polyadic set algebrs.  If $\T$  happens
to be finitely presented;  for a start one gets a finite signature and a potential complete finite axiomatization. 
This finite signature {\it expands}
that of $\CA_{\omega}$ in the sense that the infinitely many operations of 
$\CA_{\omega}$  become term
definable in this (finite new) signature.  But this potential can (and will) be attained. 
In this hitherto obtained finite signature, one stipulates a {\it finite} set 
of equations,  enforcing representablity  (in the polyadic equality sense). 
This represenatbility result is proved using a neat embedding theorem analogous to Henkn's neat embedding theorem for $\CA_{\omega}$s; 
it is a relativized version thereof. In this context, cylindrifiers and diagonal elements are interpreted in the representing class 
of set algebras like in cylindric set algebras, and the substitutions operators, as indicated,  are interpreted like in polyadic set algebras.

The difference here is that the top elements of the newly obtained set algebras are a union of $\omega$--dimensional cartesian spaces that may not be 
disjoint. But such set algebras plainly share the intuitive geometric nature of generalized cylindric set algebras of the same dimension $\omega$, 
whose top elements are unions of such 
$\omega$--dimensional cartesian spaces that {\it are disjoint}. Dropping  the condition of disjointness kills 
commutativity of cylindrifiers,  but not completely.  In this new guarded semantics a
weaker commutativty property involving cylindrifiers and substitutions 
given in  item $\sf C_9^*$ of \cite[Definition 6.3.7]{Fer} is preserved.

We formulate the next theorem \ref{c} fully proved in \cite{mlq} 
as a  Stone--like representability result for algebras of 
relations of infinite rank drawing the analogy with theorem \ref{m}.
But first we give an example of a concrete 
{\it rich semigroup}; rather than giving the general definition. 
Such semigroups 
are defined (abstractly) in \cite{Sain, AU}.  The concrete instance recalled next  makes the idea of 
{\it generating infinitely many operations using only finitely many}
more tangible. The key idea here is the presence of 
a successor like element among the elements of the semigroup.  Iterating this operator generates $\omega$-- many extra dimensions 
paving the way for a neat embedding 
theorem. (This idea can be traced back to the work of Craig addressng finitizability 
attempts as well). 

The semigroup $\sf T$ generated by the set of 
transformations 
$\{[i|j], [i,j], i, j\in \omega, \sf suc, \sf pred\}$ defined on $\omega$ is an
example of a countable rich subsemigroup of $(^\omega\omega, \circ)$.
Here $\sf suc$ abbreviates the {\it successor function}
on $\omega$, ${\sf suc}(n)=n+1$, and $\sf \sf pred$ acts as its quasi--right inverse, the 
{\it predecessor
function} on $\omega$, defined by
${\sf pred}(0)=0$ and for other $n\in \omega$, ${\sf pred}(n)=n-1.$
In every rich semigroup there are two elements $\pi$ and $\sigma$ called {\it distinguished elements} where $\pi$ is a `successor--like' transformation and 
$\pi$ is its quasi--inverse, of which  
the transformation $\sf pred$ 
is a special case.

Recall that $\mathcal{B}(V)$ denotes the 
Boolean algebra $\langle\wp(V), \cup, \cap, \sim, \emptyset, V\rangle$.
The next theorem is an (algebraic) solution to the finitizability problem for first order logic with equality:

\begin{theorem}\label{c} Let $\T$ be a countable rich finitely presented subsemigroup 
of $(^{\omega}\omega, \circ)$ 
with distinguished elements $\pi$ and $\sigma$. 
Assume that $\sf T$ is presented by the finite set of transformations $\sf S$ such that $\sigma\in \sf S$. 
Then the  class $\sf Gp_T$ of all $\omega$--dimensional set algebras of the form
$\langle\mathcal{B}(V), {\sf C}_0, {\sf D}_{01}, {\sf S}_{\tau}\rangle_{\tau\in \sf S},$
where $V\subseteq {}^{\omega}U,$  $V$ a non--empty union of cartesian spaces, is a finitely axiomatizable variety.
All the operations ${\sf C}_i, {\sf D}_{ij}$, $i,j\in \omega\sim \{0\}$ are term definable.
\end{theorem}

{\bf It is quite reasonable to take the concept of {\it a finite proof} as the most fundamental concept in logic. 
From this predominantly philosophically motivated point of view semantics come second and 
perhaps merely as
a theoretical tool. Manipulating semantics to achieve finite Hilbert style axiomatizations no longer becomes mere tactical opportunism nor 
a challenging mental exercise, nor an intellectual enterprise. 
It rather becomes a  pressing need.}

In this respect, the logical counterpart of the first part of the previous theorem avoiding the severe incompleteness theorem obtained in theorem \ref{negative} 
is:
\begin{theorem}\label{f}
Let $\T$ be a semigroup as specified in the previous 
theorem. Let $\L_{\T}$ be the algebraizable logic corresponding to ${\sf Gp}_{\T}$ (in the Blok--Pigozzi sense). 
Then the
satisfiability relation $\models_w$ induced by ${\sf Gp}_{\T}$
admits a  finite recursive sound and  complete
proof calculus for the set of type--free valid formula schemata which involves only type--free valid formula schemata  $\vdash$ say,
with respect to $\models_w$, so that $\Gamma\models_w \phi\iff \Gamma\vdash \phi$.  This recursive complete 
axiomatization is a  Hilbert style axiomatization, and there is a translation recursive function $\sf tr$ mapping
$L_{\omega, \omega}$ formulas  to  formulas in $\L_{\T}$
preserving $\models_w$ (but not the usual validity $\models$).
\end{theorem}
Here {\it type--free valid formula schemata} is the plural of {\it type--free valid formula schema}. This is a new notion of validity defined by Henkin et 
al. \cite[Remark 4.3.65, Problem 4.16]{HMT2}, \cite[p. 487]{HHbook}. 

The solution in \cite{mlq} where Tarskian semantics are  broadened, but  only 
slightly, possibly stands against
Henkin, Monk and Tarski's expectations, for the second problem in the first quote \cite[pp.416]{HMT1} {\it does not} prohibit the option of changing  the semantics, 
that is alter the notion of representability,
as long as it is `concrete and intuitive' enough. This, in turn, possibly indicates that their conjecture as 
formulated in the last two lines of their quote from \cite[pp.416]{HMT1} (recalled above) was either too hasty or/ and unfounded.
The aforementioned positive finitizabitity result 
formulated in theorem \ref{c} is an infinite 
analogue of the polyadic equality analogue of the classical 
Andr\'eka--Thompson--Resek theorem \cite{AT} proved by 
Ferenczi \cite{Fer}. 

{\bf In set algebras considered the relativization ({\it guarding}) is the same.}

For first order logic the {\it Entscheidungsproblem}
posed by Hilbert has a negative answer: The validity problem of first order logic is undecidable. 
This is inherited by its finite variable fragments as long as the number of variables used is at 
least three. The validity problem for $\L_{\T}$ is not known  when 
$\T$ rich and finitely presented. There could be one such $\sf T$ that renders decidability. 
Algebraically, we do not know whether the equational theory
of $\bold I{\sf Gp}_{\T}$ is decidable  or not.  Many examples of such semigroups are given 
in \cite{Sain, mlq}.

We refer to the guarding dealt with in theorem \ref{adm} as {\it global guarding}. 
as opposed to 
{\it local guarding} a notion that we describe in the next.

Fix $2<n<\omega$. We have dealt with the cases when $V={}^nU$ (here there is no guarding) 
and the case when $V$ is {\it locally square}; a guarding that we saw diffused the negative properties of `square Tarskian semantics'.
This way of {\it global guarding} has led to the discovery of a whole landscape of multimodal logics having nice modal behaviour (like decidability) 
with the multimodal logic whose modal algebras are the class of relativized set algebras `at the bottom' 
and Tarskian semantics with its undesirable properties (like undecidability) is only the top of an iceberg.
Below the surface a treasure of nice multimodal logics was discovered. 
Viewed differently, first order logic with $n$ variables is a {\it dynamic logic} of variable assignments, 
whose atomic processes shift values in registers $x_0, \ldots, x_{n-1}$. 
This view opens up a heirarchy of  fine structures underneath standard predicate logic using $n$ variables,  
the latter becomes the undecidable theory of one particular mathematical class of   `{rich assignment models}'
or squares. Furthermore, it lacks a completeness theorem using Hilbert style axiomatizations.

But there are other assignments models that are  {\it not as rich, hence not as expressive} 
hence potentially decidable.  Guarding can be viewed as the process of finding logics in this landscape that are reasonably  expressive,  
share positive metalogical properties 
of first order, like interpolation and even improve 
on this by 
completeness and decidability.
The idea is that we want to find a semantics that give just the right action 
while additional effects of square set--theoretic representations are separated out as negotiable decisions of formulation 
that can threaten completeness and decidability.
Using square semantics is a voluntary commitment to one particular mathematical
implication whose complexity seems to be an overkill. An insidious term often confuses this issue 
is the `concreteness of set theoretic models' and the pre--assumption of the {\it canonicity} of `simple' square ones. 
Such a conventional view can be harmful; preventing us from avoiding 
hereditory mistakes of old paradigms.

When we free ourselves from such prejudices, 
such `extended first order logics' suggest further interesting applications. It becomes of practical interest just how 
much of predicate logic is used in mathematical proofs. 
Can decidable  fragments of predicate logic be used instead? 
For example working logicians in linguistics or computer science have the gut feeling that the phenomena at hand are largely 
decidable. It is hard though to pin down mathematically such `feelings' or intuition.
Likewise there could be useful decidable systems
of arithmetic or other parts
of mathematics using such ideas. For example what is the theory
of natural numbers with all possible families of variable assignments?
The thrust of this line
of research suggests that the {\it genuine logical core} of first order logic 
may well be decidable and that undecidability resulted from 
using `more than needed'.
Proving decidability for guarded 
fragments of first order logic went historically via the {\it mosaic method} of 
N\'emeti's, later developed for guarded fragments to so--called quasi--models, which is a mixture of {\it filtration}, a well known technique for proving decidability results
in modal logic, mosaics and 
semantic tableaus for first order logic. 
Such a technique also works for the loosely guarded fragments \cite{v} the result we used in the proof of 
theorem \ref{c}.

\subsection{Local guarding}

Such locally guarded representations essentially amounts to working with the varieties $\bold S{\sf Nr}_n\CA_m$ ($n\leq m<\omega$) 
as defined in \cite[Definitions 2.6.27]{HMT1} with $m$ measuring how close or rather `how far' we are from an ordinary representation.  
An algebra in $\bold S{\sf Nr}_n\CA_m$, for $n<m\leq \omega$ possesses a so--called {\it $m$--flat relativized representation} 
which is an `$m$--approximation' to an ordinary representation. 
We stipulate that an $\omega$--flat 
representation is just an ordinary one for algebras with countably many atoms. Roughly, the parameter $m$ measures how much we have to `zoom in 
by a movable window', so that we mistake the $m$--flat representation for an ordinary (genuine) one. 
Such notions are discussed and studied extensively for relation algebras in 
\cite[Chapter 13]{HHbook}. In \cite{mlq} analogous investigations for cylindric--like  algebras were intiated.
To the best of our knowledge, such investigations (before \cite{mlq}) seem to be lacking or at 
best very rare.
By Henkin's Neat Embedding theorem, namely, using the notation of \cite[Theorem 3.2.10]{HMT1}, 
$\RCA_n={\bold S}\sf Nr_n\CA_{\omega}$ one might be tempted to attribute such negative results for $2<n<\omega$,
to the existence of infinitely many spare dimensions expressed in the superscript 
$\omega$. But we shall see below that  one should not run  to such an unfounded and hasty 
conclusion. 
The presence of {\it only $n+3$ spare dimensions} suffices to obtain negative results proved to hold for ${\sf RCA}_n$.
Indeed, as it happens, the classical negative results in \cite{Hodkinson, HH} on {\it atom--canonicity and complete representations} 
will be generalized to include the varieties $\bold S{\sf Nr}_n\CA_{n+k}$ for $2<n<\omega$ and $k\geq 3$ in theorem \ref{can}.
 
Using the semantical notion of flatness, together with the results obtained in  theorem \ref{can} and those stated in 
the forthcoming theorem \ref{flat}, we 
infer that for any $m\geq n+3$ the variety of algebras having $m$--flat representations is not atom--canonical, and the 
class of algebras having 
{\it complete} $m$--flat representations is not first order definable.
We learn from the above results, that for such proper approximations of ${\sf RCA}_n$ negative properties persist. 
Here by `proper approximations' we mean that for $2<n<\omega$ and finite $m>n$ 
${\sf RCA}_n\subsetneq \bold S{\sf Nr}_n\CA_{m}$  and 
$\bigcap_{k\in \omega, k\geq 3} \bold S{\sf Nr}_n\CA_{n+k}={\sf RCA}_n$ \cite[Theorem 2.6.34]{HMT1}. Furthermore, for $2<n<\omega$, 
the sequence $(\bold S{\sf Nr}_n\CA_m: n+1<m\leq \omega)$ is {\it strictly decreasing} (with respect to inclusion of classes) `converging' to ${\sf RCA}_n$, 
cf. \cite[Problem 2.12]{HMT1} and its answer provided 
in \cite[Theorem 15.1]{HHbook}.
This last view is purely syntactical. 

{\it Till the end of this subsection fix $2<n<m<\omega$.}

Given $\A\in \CA_n$ having an  $m$--flat representation with domain $M$, 
one forms 
an $m$--dilation $\B$ of $\A$ as follows: The algebra $\B$ has top element
the so--called {\it $n$--Gaifmann hypergraph}, having set of hypereges the set
${\sf C}^n(M)=\{s\in {}^{m}M: \rng(s) \text { is an $n$--clique}\}$; so $\B$ is a relativized set algebra of dimension $m$, with 
the operations in $\B$ are induced by 
the so-called {\it clique 
guarded semantics}; denoted by $\models_c$.  Here an $n$--clique is a natural generalization of the notion of 
cliques in graph theory to hypergraphs; every 
$n$--tuple from $\rng(s)$ is labelled by the top element of the algebra $\A$. 

Let $L(A)^n$ is the first order language using $m$ 
variables in a signature consisting of one $n$--ary relation symbol for 
each  element of $\A$, cf. \cite[Propsition 19.4--19.5]{HHbook}. 

\begin{definition}\label{clique} The {\it clique guarded semantics $\models_c$} are defined inductively. 
For atomic formulas and Boolean connectives they are defined
like the classical case and for existential quantifiers
(cylindrifiers) they are defined as follows:
for $\bar{s}\in {}^mM$, $i<m$, $M, \bar{s}\models_c \exists x_i\phi$ $\iff$ there is a $\bar{t}\in {\sf C}^n(M)$, $\bar{t}\equiv_i \bar{s}$ such that
$M, \bar{t}\models \phi$.
$M$  is  {\it $m$--square}, if  $\bar{s}\in {\sf C}^n(M), a\in \A$, $i<n$,
and   $l:n\to m$ is an injective map, $M\models {\sf c}_ia(s_{l(0)},\ldots, s_{l(n-1)})$,
$\implies$ there is a $\bar{t}\in {\sf C}^n(M)$ with $\bar{t}\equiv _i \bar{s}$,
and $M\models a(t_{l(0)}, \ldots, t_{l(n-1)})$.

Finally, $M$ is said to be {\it $m$--flat} if  it is $m$--square and
for all $\phi\in \L(\A)^m$, for all $\bar{s}\in {\sf C}^n(M)$, for all distinct $i,j<m$,
$M\models_c [\exists x_i\exists x_j\phi\longleftrightarrow \exists x_j\exists x_i\phi](\bar{s}).$
\end{definition}

So if $M$ is an $m$--flat representation of $\A$, then in the constructed $m$--dilation $\B$ of $\A$ with top element 
${\sf C}^n(M)$ and operations induced by the clique guarded (flat) semantics,  
cylindrifiers 
commute, so $\B\in \CA_m$, and hence $\A\in \bold S\sf Nr_n\CA_m$.

Conversely, from an $m$--dilation $\B$ of the {\it canonical extension}
of $\A\in \CA_n$, one constructs an $m$--dimensional 
hyperbasis \cite[Definition 12.11]{HHbook} modified (in a straightforward manner) to the $\CA$ case.
This $m$--dimensional hyperbasis can be viewed as a saturated set of $m$--dimensional hypernetworks (mosaics) 
that can be glued together in a step--by--step
manner to build the required $m$--flat representation of $\A$. 
For the relation algebra case witness
\cite[Lemmata 13.33-34-35, Proposition 36]{HHbook}.

We summarize the above discussion in the following theorem. 
For a class $\K$ of having a Boolean reduct let $\K\cap \bf At$ be the class consisting 
of atomic algebas in $\K$. Now it is proved in \cite{mlq} that:
\begin{theorem} \label{flat} Let $\A\in \CA_n$. Then the following hold:
\begin{enumerate}
 \item $\A\in \bold S{\sf Nr}_n\CA_m\iff \A$ has 
an $m$--flat representation.
\item $\A\in \bold S_c{\sf Nr}_n(\CA_m\cap \bf At)\iff $ has 
a complete $m$--flat representation.
\end{enumerate}
\end{theorem}

If an algebra $\A$ has an 
$m$--square representation,
then the algebra neatly embeds into another $m$--dimensional algebra $\B$.
The  '$m$-dilation' $\B$ is formed the same way as for $m$--flatness. In particular, the top element of $\B$ 
is ${\sf C}^n(M)$ where $M$ is the $m$--square representation and the operations like before are induced by the clique guarded semantics.
We have $\B\in {\sf Crs}_m$, but $\B$ is not necessarily a $\CA_m$ for  it may fail commutativity of cylindrifiers. 
This discrepancy in the formed  dilations blatantly manifests itself in
a very important property.
The `Church Rosser condition' of commutativity of cylindrifiers in the formed dilation in case of $m$--flatness when $m\geq n+3$,
makes this clique guarded fragment strongly undecidable.

Let us formulate the latter result and some related ones for three dimensions.  Assume that  $m\geq 6$. Then it is undecidable to tell
whether a finite algebra in  $\CA_3$ has an $m$--flat representation, and the variety
$\bold S{\sf Nr}_3\CA_m$ cannot be finitely axiomatizable
in $k$th order logic for any positive $k$.
This can be proved by lifting the analogous results for relation algebras
\cite[Theorem 18.13, Corollaries 18.14, 18.15, 18.16]{HHbook}. One uses the construction of Hodkinson in \cite{AUU}
which associates recursively to every atomic relation algebra $\sf R$, an atomic  $\A\in \CA_3$
such that ${\sf R}\subseteq \sf Ra\A$, 
the latter is the relation algebra reduct of $\A$, cf. \cite[Definition 5.3.7, Theorem 5.3.8]{HMT2}.
The idea for the second part on non--finite axiomatizability is that the existence of any such finite axiomatization in 
$k$th order logic for any positive $k$,  
gives a decision procedure for telling whether a finite algebra is 
in $\bold S{\sf Nr}_3\CA_m$
or not \cite{HHbook} which is impossible as just shown.

Furthermore, there are finite algebras that have infinite $m$--flat representations, but do not
have finite ones, equivalently 
they {\it do not have a finite $m$--dimensional hyperbasis.}
Roughly an $m$--dimensional hyperbasis consists of a `saturated set' of 
$m$--dimensional hyernetworks. An $m$--dimensional hypernetwork is an extension of an $m$--dimensional networks with 
labelled  hyperedges. These $m$--dimensional hypernetworks, in an $m$--dimensional hyperbasis, 
satisfy certain closure consistency conditions analagous to basic $m\times m$ matrices 
in an $m$--dimensional cylindric basis \cite[Definition 12.11]{HHbook2}.
To see why, assume for contradiction that every finite algebra in $\bold S{\sf Nr}_3\CA_m$ has a finite
$m$--dimensional hyperbasis. 
We claim that there is an algorithm that decides membership in $\bold S{\sf Nr}_3\CA_m$ for finite algebras which we know 
is impossible:
\begin{itemize}
\item Using a recursive axiomatization of $\bold S\sf Nr_3\CA_m$ (exists), recursively enumerate all isomorphism types of
finite $\CA_3$s that are not in $\bold S{\sf Nr}_3\CA_m.$

\item Recursively enumerate all finite algebras in $\bold S{\sf Nr}_3\CA_m$.
For each such algebra, enumerate all finite sets of $m$--dimensional hypernetworks over $\A$,
using $\N$ as hyperlabels, and check  to
see if it is a hyperbasis. When a hypebasis is located specify $\A$.
This recursively enumerates  all and only the finite algebras in $\bold S{\sf Nr}_3\CA_m$.
Since any finite $\CA_3$ is in exactly one of these enumerations, the process will decide
whether or not it is in  ${\bold S}{\sf Nr}_3\CA_m$ in a finite time.
\end{itemize}
We have shown that there are finite algebras that have infinite $m$--flat representations, but do not
have finite ones (this cannot happen with $m$--squareness).

We end this section with the following theorem:
\begin{theorem}\label{packed} Let $\A\in \CA_n$ and 
$M$ be an $m$--flat representation of $\A$.
Then  
$$M, s\models_c \phi  \iff\  M,s\models {\sf packed}(\phi),$$
for all $s\in {\sf C}^n(M)$ and every $\phi\in L(A)^n$, where $\sf packed(\phi$) denotes the translation of $\phi$ to the packed fragment \cite[Definion 19.3]{HHbook}.
\end{theorem}

In the sense of the previous theorem, the {\it the clique guarded fragments}, which are the $n$--variable fragments of first order order with clique 
(locally) guarded semantics are an alternative formulation of {\it the $n$--variable 
packed fragments}   of first order logic \cite[\S 19.2.3]{HHbook}.

\section{Products; a construction orthogonal to guarding} 

Now  we present a construction  that can be seen as {\it orthogonal to global guarding}.
In the latter process {\it the states are altered} but the accessibility relations (corresponding to cylindrifiers and diagonal
elements) {\it are the same}. In the present construction, the states are kept as they are 
but the {\it accessibility relations along
the components are altered}.

Fix $2<n<\omega$. In this subsection, though 
we follow \cite{k} for terminlogy, the subsection is fairly 
self contained.

{\bf Combination of unimodal logics; fusions and products:} 
Many multimodal logics can be considered as a combination of unimodal logics.
So in a sense any result on multimodal logic sheds light on combining modal logics.
But dually, we can start with `the components' and form a modal logic
that somehow encompasses them or extends them; we seek a multimodal logic in which they
`embed'.

In such a process it is very natural to ask about {\it transfer results}, namely, these properties  of the components that transfer
to the combination, like axiomatizability (completeness),
and decidability that involves complexity of the satisfiability
problem. 

There are two versions depending on whether
the combination method is syntactic or semantical, namely, 
{\it fusion} and {\it products}, respectively.
In fusions the components do not interact which makes transfer results from the components to the fusion
easy to handle. In fact, the fusion of consistent modal logics is
a conservative extension
of the components, and the fusion of finitely many logics (this is well defined because the fusion operator is associative)
has the finite model property if each of its components does. Fusion also preserves decidability.
However, determining {\it degrees of complexity}
is quite intricate here. For example, it is not known {\it $\sf PSPACE$ or $\sf EXPTIME$ completeness} transfer under
formation  of fusions.

Product logics  are far more complex. They are designed  semantically.
A product logic is {\it the multimodal logic  of products of
Kripke complete frames}, so by definition it is also Kripke complete.
It is not hard to see that the fusion of logics is contained in their product.
But in products {\it the modalities interact}, and this very interaction obviously adds to its components.
In fusions such interactions
are simply non--existent.
In such a process negative properties persist. The reason basically is that
products reflect the interaction of modalities; which is to be blamed
for the negative result. If they miss on anything
then they miss only on the uni--dimensional aspects of modalities and
these do not really contribute to negative results.
Negative results are caused
by the interaction of modalities not by their uni--dimensional properties.

{\bf Transfer results:} For example, in a product of two uni-modal logics, a precarious Church Rosser condition
on modalities is created via $\Diamond_i\Box_jp=\Box_j\Diamond_ip$  ($i,j <n$).
Indeed, it is
known that the theory of two commuting confluence
closure operators is undecidable.
Nevertheless, {\it commuting closure operations} alone can be harmless like in the case of
many cylindric--like algebras of dimension $2$,
but the interaction of the two modalities
expressed by the confluence is potentially harmful.
The product logic of two countable time flows is not even recursively enumerable, furthermore the modal logic of
$(\mathbb{N}, <)$ is undecidable.

Compared to fusions, there are very few general transfer results for products,  in fact here {\it the exact opposite occurs}.
Nice properties do not transfer, rather the lack of transfer is the norm, particularly
concerning finite axiomatizability and decidability.
An interesting example here is ${\bold  K}^n$  for $n\geq 3$.
Obviously the components are finitely axiomatizable, and their modal logics are decidable.
However, such a logic viewed as a product of frames of the form $(U, R)$ where $R\subseteq U\times U$
is an arbitrary relation has the finite model property,
but it encodes the tiling problem and so it is undecidable. 

In fact, for there are three dimensional
formulas that are valid in all higher dimensional finite products, but can be falsified on an infinite
frame. Here validity in higher dimensions is meaningful, because if $n<m$ then the modal logic ${\bf K}^n$
embeds into ${\bf K}^m$.
Furthermore, it is undecidable to tell whether a finite frame
is a frame for this logic, and this gives
strong non--finite axiomatizability results, and obviously implies undecidability.
It is known \cite{k} that between ${\bf K}^n$ and 
${\bf S5}^n$ for $n\geq 3$ is quite complicated.
The class of modal algebras for ${\sf S5}^n$ is just the class of
diagonal free cylindric algebras of dimension
$n$. For an overview of such results, and more, the reader is referred to \cite{k}.
It is known \cite{k} that one can add diagonal constants
to $n$ products forming so called {\it $\delta$ products}.
The fact that all three dimensional modal logics are undecidable can be intuitively
and indeed best
explained by the undecidability of the product ${\sf S5}^3$ and its relation to
the undecidable fragment of
first order logic with $3$ variables represented algebraically by ${\sf CA}_3$ whose equational theory is undecidable; a result of Maddux 
\cite{HMT2}.

Square Tarskian semantics can be seen as a {\it limiting case of  (i) relativized semantics, (ii) locally guarded or clique guarded semantics,
(ii) products

(i) In the first case the limit is taken on a varying set of worlds (states) approaching the square
$^nU$.} 

(ii) Let $2<n<m$. In the second case, viewed semantically  $m$ is allowed to grow.  
Here $m$--flatness witnesses of cylindrifiers
are allowed more space. Syntactically more and more degrees of freedom or dimensions
are created, so that an algebra $\A\in \CA_n$ has an $m$--flat representation
$\iff$ $\A$ {\it neatly embeds} into an $m$-- dimensional algebra, which is a {\it trunacted} neat embedding theorem 
as indicated above.

(iii) In products, {\it the accessibility relations along the components are changed} approaching $\bf S5^n$
where all accessibility relations along the components are the universal one.

\section*{Part 2, Getting more technical:}

\section{When rainbows and neat embeddings meet}

Here we obtain results on classes of algebras having a neat embedding property using rainbow constructions. In the process, we 
generalize the seminal results on atom--canonicity and  complete 
representations proved in \cite{Hodkinson, HH}, respectively.

\subsection{Rainbow constructions}

We need  the notions of {\it atomic networks} and {\it atomic games} \cite{HHbook, HHbook2}:

Let $i<n$. For $n$--ary sequences $\bar{x}$ and $\bar{y}$ we write $\bar{x}\equiv_ i\bar{y}$
$\iff \bar{y}(j)=\bar{x}(j)$ for all $j\neq i$.
 
\begin{definition}\label{game} Fix finite $n>1$. 

\begin{enumarab}
\item An {\it $n$--dimensional atomic network} on an atomic algebra $\A\in \CA_n$  is a map $N: {}^n\Delta\to  \At\A$, where
$\Delta$ is a non--empty set of {\it nodes}, denoted by $\nodes(N)$, satisfying the following consistency conditions: 
\begin{itemize}
\item If $\bar{x}\in {}^n\nodes(N)$, and $i<j<n$, then $N(x)\leq {\sf d}_{ij}\iff x_i=x_j$.
\item If $\bar{x}, \bar{y}\in {}^n\nodes(N)$, $i<n$ and $\bar{x}\equiv_i \bar{y}$, then  $N(\bar{x})\leq {\sf c}_iN(\bar{y})$.
\end{itemize}
For $n$--dimensional atomic networks $M$ and $N$, 
we write $M\equiv_i N\iff M(\bar{y})=N(\bar{y})$ for all $\bar{y}\in {}^{n}(n\sim \{i\})$.

\item   Assume that $\A\in \CA_n$ is  atomic and that $m, k\leq \omega$. 
The {\it atomic game $G^m_k(\At\A)$, or simply $G^m_k$}, is the game played on atomic networks
of $\A$
using $m$ nodes and having $k$ rounds \cite[Definition 3.3.2]{HHbook2}, where
\pa\ is offered only one move, namely, {\it a cylindrifier move}: 
\begin{itemize}
\item Suppose that we are at round $t>0$. Then \pa\ picks a previously played network $N_t$ $(\nodes(N_t)\subseteq m$), 
$i<n,$ $a\in \At\A$, $x\in {}^n\nodes(N_t)$, such that $N_t(x)\leq {\sf c}_ia$. For her response, \pe\ has to deliver a network $M$
such that $\nodes(M)\subseteq m$,  $M\equiv _i N$, and there is $y\in {}^n\nodes(M)$
that satisfies $y\equiv _i x$, and $M(y)=a$.  
\end{itemize}
\item  We write $G_k(\At\A)$, or simply $G_k$, for $G_k^m(\At\A)$ if $m\geq \omega$.
The {\it atomic game $F^m(\At\A)$, or simply $F^m$}, is like $G^m_{\omega}(\At\A)$ except that
\pa\ has the advantage to reuse the available
$n$ nodes during the play.
\end{enumarab}
\end{definition}

One can show that for $2<n<m<\omega$, the game $F^m$ 
tests neat embeddability in the following sense:
\begin{lemma}\label{n}\cite{mlq} If $\A$ is atomic and 
$\A\in \bold S_c{\sf Nr}_n\CA_m$, then \pe\ has a \ws\ in $F^m(\At\A)$ 
In particular, if $\A\in {\sf Nr}_n\CA_{\omega}$, 
then \pe\ has a \ws\ in $F^{\omega}(\At\A)$ and $G_{\omega}(\At\A)$, and if 
$\A$ is finite and \pa\ has a \ws\ in $F^m(\At\A)$, then $\A\notin \bold S{\sf Nr}_n\CA_{m}$. 
\end{lemma}
\begin{proof}\cite{r, mlq}. The last part follows by observing that for any $\C\in \CA_n$, if $\C\in \bold S{\sf Nr}_n\CA_m\implies \C^+\in \bold S_c{\sf Nr}_n\CA_m$
(where $\C^+$ is the canonical extension of $\C$) and if  $\C$ is finite,  then of course $\C=\C^+$.
\end{proof}
In the first item of theorem \ref{can}, we will use the following lemma:
\begin{lemma}\label{square} Let $2<n<m<\omega$. Let $\A\in \CA_n$ be finite. 
Then \pe\ has a \ws\ in  $G^m(\At\A)\iff \A$ has an $m$--square representation.
\end{lemma}

Now we recall `rainbow constructions' as introduced in algebraic logic by Hirsch and Hodkinson \cite{HHbook2}.
Fix $2<n<\omega$. {\it Coloured graphs} \cite{HH} are complete graphs whose edges are labelled by the rainbow colours, $\g$ (greens), $\r$ (reds), and 
$\w$ (whites) satisfying certain consistency conditions. The greens are 
$\{\g_i: 1\leq i< n-1\}\cup \{\g_0^i: i\in \sf G\}$  and the reds are $\{\r_{ij}: i,j \in \sf R\}$ where
$\sf G$ and $\sf R$ are two relational structures. The whites are $\w_i: i\leq n-2$.

In coloured graphs certain triangles are forbidden.
For example a green triangle (a triangle whose edges are all green)  is forbidden. Not all red triangles are allowed.  
In consistent (allowed) red triangle the indices 
`must match' satisfying a certain `consistency condition'. Also, in coloured graphs some $n-1$ tuples (hyperedges) are labelled by shades of yellow \cite{HH}. 

Given relational structures 
$\sf G$ and $\sf R$ the rainbow 
atom structure of dimension $n$ are equivalence classes of surjective maps $a:n\to \Delta$, where $\Delta$ is a coloured graph
in the rainbow signature, and the equivalence relation relates two such maps $\iff$  they essentially define the same graph \cite[4.3.4]{HH};
the nodes are possibly different but the graph structure is the same. We let $[a]$ denote the equivalence class containing $a$.
The accessibility binary relation corresponding 
to the $i$th  cylindrifier $(i<n)$ is defined by:  $[a] T_i [b]\iff a\upharpoonright n\sim \{i\}=b\upharpoonright n\sim \{i\},$ 
and the accessibility unary relation corresponding to the $ij$th diagonal element ($i<j<n$) 
is defined by: $[a]\in D_{ij}\iff a(i)=a(j)$. We refer to the atom $[a]$ ($a:n\to \Delta$) as an atom and sometimes as an $n$ --coloured graph.
We denote the complex algebra of the rainbow atom structure based on 
$\sf G$ and $\sf R$ by $\CA_{\sf G, R}$. The dimension of $\CA_{\sf G, R}$ will be clear 
from context. Identifying $[a]$ with $a$, we sometimes refer to an atom $[a]:n\to \Delta$ of $\CA_{\sf G, R}$ as 
an $n$--coloured graph. 

If $\A$ is an (atomic)  rainbow $\CA_n$, 
and $k, m\leq \omega$, then the games $G_k(\At\A)$ and $F^m(\At\A)$ 
translate to games on coloured graphs ($\sf CGR$) \cite[p.27--29]{HH}.
Certain special finite coloured graphs play an essential role in `rainbow games' whose board  consists of coloured graphs.
Such special coloured graphs are called {\it cones}:

{\it Let $i\in {\sf G}$, and let $M$ be a coloured graph consisting of $n$ nodes
$x_0,\ldots,  x_{n-2}, z$. We call $M$ {\it an $i$ - cone} if $M(x_0, z)=\g_0^i$
and for every $1\leq j\leq m-2$, $M(x_j, z)=\g_j$,
and no other edge of $M$
is coloured green.
$(x_0,\ldots, x_{n-2})$
is called  the {\bf base of the cone}, $z$ the {\bf apex of the cone}
and $i$ the {\bf tint of the cone}.}

The \ws\ of \pa\ in the rainbow game played on coloured graphs played between \pe\ and \pa\  
is bombarding \pe\ with $i$--cones, $i\in \sf G$, having the 
same base 
and distinct green tints.  To respect the rules of the game \pe\ has to choose a red label for appexes of two succesive cones.  
Eventually, running out of `suitable reds',  \pe\ is forced to play an inconsistent triple of reds where indices do not match.
Thus \pa\ wins on a red clique (a graph all of whose edges are lablled by a red).  
Such a \ws\ is dictated by a simple \ef\ forth  
game played on the relational structures $\sf G$ and $\sf R$ denoted by 
${\sf EF}_p^r(\sf G, R)$ in  \cite[Definition 16.2]{HHbook2} where $r$ is the number of rounds and $p$ is number of pairs of pebbles on board. 

\subsection{Atom--canonicity, first order definability and finite axiomatizability}

Unless otherwise indicated, in this section $n$ is fixed to be finite $n>2$. 
To formulate 
and prove the next theorem, we need to fix some notation.
We write $\bold S_c$ for the operation of forming  complete sublgebras and $\bold S_d$ for the operation of forming 
dense subalgbras. Let $\bold K$ be a class of $\sf BAO$s and 
$\A$, $\B\in \bold K$ such that $\A\subseteq \B$. Then $\A$ is a {\it complete subalgebra} of $\B$ if 
$\A\subseteq \B$ and for $X\subseteq \A$, $\sum ^{\A}X=1\implies \sum ^{\B}X=1$. 
(Recall that) $\A$ is {\it dense} in $\B$ if $(\forall b\in B\sim \{0\})(\exists a\in \A\sim \{0\})(a\leq b)$.
Observe that the two definitions depend only on the Boolean part of the algebras considered. Furthermore, it is not hard to show
that for any such $\bold K$, $\bold K\subseteq \bold S_d\bold K\subseteq \bold S_c\bold K$. In fact, 
the last two inclusions are proper when $\bold K$ is the class
of Boolean algebras (without operators).

Now we are ready to generalize the main results 
in \cite{Hodkinson, HH}. In the last reference it is shown that the class of completely representable $\CA_n$s, 
briefly ${\sf CRCA}_n$ when $2<n<\omega$ is not elementary, and in the former, it is shown that $\RCA_n$ 
is not atom--canonical.
Recall that Andr\'eka \cite{Andreka} 
splits atoms in a representable algebra to get a non--representable one. 
But one can split atoms to do the exact opposite. One can undergo {\it the reverse process} passing from a 
non--representable algebra to a representable one via so--called {\it blow up and 
blur constructions} \cite{ANT, mlq}, a task to be implemented next.
Recall that $n$ is fxed to be finite $>2$. Let $k\geq 3$.  For an infinite ordinal $\alpha$, we stipulate that a (complete) 
$\alpha$--flat representation is just a (complete)
representation.

\begin{theorem}\label{can} 
\begin{enumerate}
\item The variety $\bold S{\sf Nr}_n\CA_{n+k}$ is not atom--canonical.  Hence the variety of algebras having $n+k$--flat representations is not atom--canonical. 
In particular,
${\sf RCA}_n$ is not atom--canonical \cite{Hodkinson}. Furthermore, the variety of algebras having $n+k$--square representations is not atom--canoical.
\item For any class $\bold K$,
such that ${\sf CRCA}_n\cap {\bf S}_d{\sf Nr}_n\CA_{\omega}\subseteq \bold K\subseteq \bold S_c{\sf Nr}_n\CA_{n+k}$, 
$\bold K$ is not elementary. The class of algebras having complete $n+k$--flat representations is not elementary. 
Furthemore, any class $\bold L$ such that $\At({\sf Nr}_n\CA_{\omega})\subseteq \bold L\subseteq \bold \At(\bold S_c{\sf Nr}_n\CA_{n+3})$ 
is not elementary. Finally, ${\bf El}{\sf Nr}_n\CA_{\omega}\nsubseteq {\bf S}_d{\sf Nr}_n\CA_{\omega}\iff$ 
any class $\bold L$ such that  ${\sf Nr}_n\CA_{\omega}\subseteq \bold L\subseteq \bold S_c{\sf Nr}_n\CA_{n+3}$, $\bold L$ is not elementary.
\end{enumerate}
\end{theorem}

\begin{proof} 

First item:   
One starts with a finite algebra $\C\in \CA_n$ that is outside $\bold S{\sf Nr}_n\CA_{n+3}$. This algebra is the rainbow algebra
$\CA_{n+1, n}$.  To see why it is outside  $\bold S{\sf Nr}_n\CA_{n+3}$, 
consider the game  \ef\ forth game ${\sf EF}_r^p( n+1,n)$. In this game 
played on the complete irreflexive graphs $n+1$ and $n$, \pa\ has a \ws\ since $n+1$ is `longer'. Here $r$ is the number of rounds and $p$ is the number of pairs of pebbles
on board. Using (any) $p>n$ many pairs of pebbles avalable on the board \pa\ can win this game in $n+1$ many rounds.
In each round $0,1\ldots n$, \pe\ places a new pebble  on  a new element of $n+1$.
The edge relation in $n$ is irreflexive so to avoid losing
\pe\ must respond by placing the other  pebble of the pair on an unused element of $n$.
After $n$ rounds there will be no such element, so she loses in the next round.
It is not hard 
to show that the \ws\ of \pa\  in the private \ef\ game lifts to a \ws\ in the graph game  
$G_k^{n+3}(\At(\CA_{n+1, n}))$ \cite[pp.841]{HH} for some finite $k$.
Hence by lemma \ref{square}, $\CA_{n, n+1}$ has no $n+3$--square representation. But plainly \pa\ 
has a \ws\ in $F_k^{n+3}(\At(\CA_{n+1, n}))$ (same $k$) 
from which it follows using lemma \ref{n}  that $\CA_{n+1, n}\notin \bold S{\sf Nr}_n\CA_{n+3}$.
 
Then one {\it blows up and blur} $\CA_{n+1, n}$, by splitting
its `(red) atoms' each to infinitely many.
After this splitting one gets a new infinite atom structure $\bf At$. This atom structure 
is like the atom structure of the (set) algebra $\A$ constructed in \cite[Definition 4.1, \S 4.2]{Hodkinson} except 
that the number of greens having $0$ as a subscript involved (in constructing $\bf At$) is only $n+1$,  
rather than $\omega$--many, as  is the case in the construction of \cite{Hodkinson}.

So in the present context, after the splitting `the finitely many red colours' replacing each such red colour $\r_{kl}$, $k<l<n$  by $\omega$ many 
$\r_{kl}^i$, $i\in \omega$, the rainbow signature for the resulting rainbow theory as defined in \cite[Definition 3.6.9]{HHbook} call this theory $T_{ra}$,
 consists of $\g_i: 1\leq i<n-1$, $\g_0^i: 1\leq i\leq n+1$,
$\w_i: i<n-1$,  $\r_{kl}^t: k<l< n$, $t\in \omega$,
binary relations, and $n-1$ ary relations $\y_S$, $S\subseteq_{\omega} n+k-2$ or $S=n+1$. 

The representation of $\Tm\bf At$  can be built exactly like in \cite{Hodkinson}. The representing set algebra of dimension $n$ has
base an $n$--homogeneous  model $M$ of another theory $T$ whose signature expands that of 
$T_{ra}$ by an additional binary relation (a shade of red) $\rho$.  
In this new signature $T$ is obtained from $T_{ra}$ by 
some axioms  (consistency conditions) extending $T_{ra}$. Such axioms (consistency conditions) 
specfify consistent triples involving $\rho$.
This model $M$ is constructed as a countable limit of finite models of $T$ 
using a game played between \pe\ and \pa.  We call the models of $T$ extended coloured graphs. 
In particular, $M$ is an extended coloured graph. Here, unlike the extended $L_{\omega_1, \omega}$ theory 
dealt with in \cite{Hodkinson},  $T$ is a {\it first order one}
because the number of greens used are finite.

As is the case with `rainbow games' \cite{HH, HHbook} 
\pa\ challenges \pe\  with  {\it cones} having  green {\it tints $(\g_0^i)$}, 
and \pe\ wins if she can respond to such moves. This is the only way that \pa\ can force a win.  \pe\ 
has to respond by labelling {\it appexes} of two succesive cones, having the {\it same base} played by \pa.
By the rules of the game, she has to use a red label. The \ws\ is implemented by \pe\  using the red label $\rho$ (outside the rainbow signature) 
that comes to her rescue  whenever she runs out of `rainbow reds', so she can respond with an extended coloured graphs. 
It turns out  inevitable, that some edges in $M$ 
are  labelled by $\rho$ during the play; in fact these edges labelled by $\rho$ will form an {\it infinite red clique}
(an infinite complete extended graph whose edges are all labelled by $\rho$.)

Now $\CA_{n, n+1}$ {\it embeds into} $\Cm\bf At$ by mapping 
every $n$--coloured graph to the join of its copies.
Let us describe the embedding more rigorously. 
Let ${\sf CRG}_f$ denote  the class of coloured graphs on 
${\At}\CA_{n, n+1}$ and $\sf CRG$ be the class of coloured graph on 
${\bf At}={\bf At}\A$. For the sake of brevity denote ${\At}\CA_{n, n+1}$ by $\bf At_f$. 
We 
can assume that  ${\sf CRG}_f\subseteq \sf CRG$.
Write $M_a$ for the atom that is the (equivalence class of the) surjection $a:n\to M$, $M\in \sf CGR$.
Here we identify $a$ with $[a]$; no harm will ensue.
We define the (equivalence) relation $\sim$ on $\bf At$ by
$M_b\sim N_a$, $(M, N\in {\sf CGR}):$
\begin{itemize}
\item $a(i)=a(j)\Longleftrightarrow b(i)=b(j),$

\item $M_a(a(i), a(j))=\r^l\iff N_b(b(i), b(j))=\r^k,  \text { for some $l,k$}\in \omega,$

\item $M_a(a(i), a(j))=N_b(b(i), b(j))$, if they are not red,

\item $M_a(a(k_0),\dots, a(k_{n-2}))=N_b(b(k_0),\ldots, b(k_{n-2}))$, whenever
defined.
\end{itemize}
We say that $M_a$ is a {\it copy of $N_b$} if $M_a\sim N_b$. 
We say that $M_a$ is a {\it red atom} if it has at least one edge labelled by a red rainbow colour $\r_{ij}^l$ for some $i<j<n$ and $l\in \omega$. 
Clearly every red atom $M_a$ has infinitely countable many red copies, which we denote by $\{M_a^{(j)}: j\in \omega\}$.
Now we define a map $\Theta: \CA_{n+`, n}=\Cm{\bf At_f}$ to $\Cm\bf At$,
by  specifing  first its values on ${\bf At}_f$,
via $M_a\mapsto \sum_jM_a^{(j)}$; each atom maps to the suprema of its 
copies.  

If $M_a$ is not red,   then by $\sum_jM_a^{(j)}$,  we understand $M_a$.
This map is extended to $\CA_{n+k-2, n}$ by $\Theta(x)=\bigcup\{ \Theta(y):y\in \At\CA_{n+k-2, n}, y\leq x\}$. The map
$\Theta$ is well--defined, because $\Cm\bf At$ is complete. 
It is not hard to show that the map $\Theta$ 
is an injective homomorphim. 
(Injectivity follows from the fact that $M_a\leq f(M_a)$, hence $\Theta(x)\neq 0$ 
for every atom $x\in \At(\CA_{n+1, n})$.)
These precarious joins prohibiting membership in $\bold S{\sf Nr}_n\CA_{n+3}$ 
{\it do not} exist in the term algebra $\Tm\bf At$, the subalgebra
of $\Cm\bf At$ generated by the atoms, only 
joins of finite or cofinite subsets of the atoms do.

We have shown that $\Tm {\bf At}\in {\sf RCA}_n$  
but its \de\ completion $\Cm\bf At$ is outside $\bold S{\sf Nr}_n\CA_{n+3}$
and $\Cm\bf At$ does not have an $n+3$--square representation proving the required.

We give another approach to proving the above result on non atom--canonicity.
Start  by the rainbow atom structure $\bf At$ (of the atomic set algebra $\A$), 
whose atoms are surjections from $n$ to coloured graphs (briefly $n$--coloured graphs). 
labelled by the rainbow colours (corresponding to the rainbow signature), greens, rainbow reds, whites..etc specified above.
Weak representability of $\bf At$ can be proved without the need of the additional red label $\rho$.  The term algebra 
$\Tm \bf At$ {\it can be represented intrinsically}, using the set of ultrafilters  $\omega\cup \{\sf RUf\}$ as `colours', 
where $\omega$ `codes' the countably many {\it principal ultrafilters} of $\Tm\bf At$; each such ultrafilter
is generated by an atom, namely, an $n$--coloured graph,  and $\sf RUf$ is the ` Red Ultrafilter' of $\Tm\bf At$ containing the filter 
generated by the co--finite sets  of $n$--coloured red graphs.

Here {\sf RUf} and $\rho$ are {\it literally} two sides of the same coin.   In fact,  $\rho$ (and $\sf RUf$) will have to be used 
infinitely many times during the play so 
that the base of any {\it potential representation} of $\Cm\bf At$ will have an infinite red clique  ${\sf RC}=(a_n: n\in \omega$), say, whose edges are labelled by $\rho$.    
It is precisely, this $\sf RC$  that prohibited a representation
of $\C=\Cm\bf At$ as illustrated in some detail on \cite[pp. 9--10]{Hodkinson}. 
though $\sf RC$ exists in the representation 
of $\Tm\bf At$, since $\Tm\bf At$ is not complete, 
$\sf RC$ will not cause any trouble as far as the representability of 
$\Tm \bf At$ is concerned.

(2) For the second item.  
Fix finite $n>2$. One takes a rainbow --like algebra based on the ordered structure $\Z$ and $\N$, that is similar but not identical to $\CA_{\Z, \N}$; call this (complex) algebra $\C$.
The reds ${\sf R}$ is the set $\{\r_{ij}: i<j<\omega(=\N)\}$ and the green colours used 
constitute the set $\{\g_i:1\leq i <n-1\}\cup \{\g_0^i: i\in \Z\}$. 
In complete coloured graphs the forbidden triples are like 
in usual rainbow constructions; 
more specifically the following are forbidden triangles in coloured graphs \cite[4.3.3]{HH}:
\vspace{-.2in}
\begin{eqnarray}
&&\nonumber\\
(\g, \g^{'}, \g^{*}), (\g_i, \g_{i}, \w_i),
&&\mbox{any }1\leq i\leq  n-2  \\
(\g^j_0, \g^k_0, \w_0)&&\mbox{ any } j, k\in \sf G\\
\label{forb:match}(\r_{ij}, \r_{j'k'}, \r_{i^*k^*})&&\mbox{unless }i=i^*,\; j=j'\mbox{ and }k'=k^*,
\end{eqnarray}
but now 
the triple  $(\g^i_0, \g^j_0, \r_{kl})$ is also forbidden if $\{(i, k), (j, l)\}$ is not an order preserving partial function from
$\Z\to\N$.
It can be proved that \pe\ has a \ws\ $\rho_k$ in the $k$--rounded game $G_k(\At\C)$ for all $k\in \omega$ \cite{mlq}.
Hence, using ultrapowers and an elementary chain argument  \cite[Corollary 3.3.5]{HHbook2}, one gets a countable 
(completely represenatble) algebra $\B$
 such that $\B\equiv \A$, and  \pe\ has a \ws\ in $G_{\omega}(\At\B)$. 

On the other hand, one can show that \pa\ has a \ws\ in $F^{n+3}(\At\C)$.
The idea here, is that, as is the case with \ws's of \pa\ in rainbow constructions, 
\pa\ bombards \pe\ with cones having distinct green tints demanding a red label from \pe\ to appexes of succesive cones.
The number of nodes are limited but \pa\ has the option to re-use them, so this process will not end after finitely many rounds.
The added order preserving condition relating two greens and a red, forces \pe\ to choose red labels, one of whose indices form a decreasing 
sequence in $\N$.  In $\omega$ many rounds \pa\ 
forces a win, 
so by lemma \ref{n} $\C\notin \bold S_c{\sf Nr}_n\CA_{n+3}$.

He plays as follows \cite{mlq}: In the initial round \pa\ plays a graph $M$ with nodes $0,1,\ldots, n-1$ such that $M(i,j)=\w_0$
for $i<j<n-1$
and $M(i, n-1)=\g_i$
$(i=1, \ldots, n-2)$, $M(0, n-1)=\g_0^0$ and $M(0,1,\ldots, n-2)=\y_{\Z}$. This is a $0$ cone.
In the following move \pa\ chooses the base  of the cone $(0,\ldots, n-2)$ and demands a node $n$
with $M_2(i,n)=\g_i$ $(i=1,\ldots, n-2)$, and $M_2(0,n)=\g_0^{-1}.$
\pe\ must choose a label for the edge $(n+1,n)$ of $M_2$. It must be a red atom $r_{mk}$, $m, k\in \N$. Since $-1<0$, then by the `order preserving' condition
we have $m<k$.
In the next move \pa\ plays the face $(0, \ldots, n-2)$ and demands a node $n+1$, with $M_3(i,n)=\g_i$ $(i=1,\ldots, n-2)$,
such that  $M_3(0, n+2)=\g_0^{-2}$.
Then $M_3(n+1,n)$ and $M_3(n+1, n-1)$ both being red, the indices must match.
$M_3(n+1,n)=r_{lk}$ and $M_3(n+1, r-1)=r_{km}$ with $l<m\in \N$.
In the next round \pa\ plays $(0,1,\ldots n-2)$ and re-uses the node $2$ such that $M_4(0,2)=\g_0^{-3}$.
This time we have $M_4(n,n-1)=\r_{jl}$ for some $j<l<m\in \N$.
Continuing in this manner leads to a decreasing
sequence in $\N$.

But we can go further. We define another $k$--rounded atomic game stronger than $G_k$ call it $H_k$, for $k\leq \omega$,  
so that if $\D\in \CA_n$ is countable and atomic and \pe\ has a \ws\ in $H_{\omega}(\At\D)$, 
then 
(*) $\At\D\in \At{\sf Nr}_n\CA_{\omega}$ and $\Cm\At\D\in {\sf Nr}_n\CA_{\omega}$ (these two conditions taken together do not imply that $\D\in {\sf Nr}_n\CA_{\omega}$, witness example \ref{SL2}).  

It can be shown that \pe\ has a \ws\ in $H_k(\At\C)$ for all 
$k\in \omega$, hence using ultrapowers and an elementary chain argument, we get that $\C\equiv \D$, 
for some countable completely representable $\D$ that satisfies the two conditions in (*).
Since $\D\subseteq_d \Cm\At\D$, we get the required result, because $\D\in \bold S_d{\sf Nr}_n\CA_{\omega}$
and as before $\C\notin \bold S_c{\sf Nr}_n\CA_{n+3}$ and $\C\equiv \D$. 
For the following part, let $\bold L$ be as specified and $\D$ and $\C=\CA_{\Z, \N}$ be the algebras in the last paragraph. Since an 
atom structure of an algebra is first order interpretable in the algebra, then we have $\D\equiv \C\implies \At\D\equiv \At\C$.
Furthermore $\At\D\in \At({\sf Nr}_n\CA_{\omega})\subseteq \bold L$ (as stated above 
$\D$ might not be in ${\sf Nr}_n\CA_{\omega}$) 
and $\At\C\notin  \At(\bold S_c{\sf Nr}_n\CA_{n+3})\supseteq \bold L$. 
The latter follows from the fact that
if $\F\in \CA_n$ is atomic, 
then $\At\F\in   \At(\bold S_c{\sf Nr}_n\CA_{n+3})\iff \F\in \bold S_c{\sf Nr}_n\CA_{n+3}$. 
We conclude that $\bold L$ is not elementary.

For the last part, It suffices to consider classes between ${\sf Nr}_n\CA_{\omega}$ and $\bold S_d{\sf Nr}_n\CA_{\omega}$, because the former class is not 
elementary \cite[Theorem 5.4.1]{Sayedneat} and as just shown any class between $\bold S_d{\sf Nr}_n\CA_{\omega}$ 
and ${\bold S}_c{\sf Nr}_n\CA_{n+3}$ is not elementary. 
One implication, namely, $\Longleftarrow$  is trivial. For the other less trivial (but still easy) implication, 
assume for contradiction that there is such a class $\bold K$ that is elementary.
Then ${\bf El}{\sf Nr}_n\CA_{\omega}\subseteq \bold K$, because $\bold K$ is elementary.
It readily follows that  ${\sf Nr}_n\CA_{\omega}\subseteq {\bf El}{\sf Nr}_n\CA_{\omega}\subseteq \bold K\subseteq \bold S_d{\sf Nr}_n{\sf CA}_{\omega}$,
which is impossible by the 
given assumption that ${\bf El}{\sf Nr}_n\CA_{\omega}\subsetneq \bold S_d{\sf Nr}_n{\sf CA}_{\omega}$.
\end{proof}
From a result of Venema's \cite[Theorem 2.96]{HHbook} by noting that varieties of $\CA_n$s are conjugated, 
we readily conclude from the first item that $\bold S{\sf Nr}_n\CA_{n+k}$
cannot be 
axiomatized by Sahlqvist equtions.

The next theorem taken from \cite{SL} highlights the limitations of the scope of the prevous proof on first order definability. 
These limitations for the relation algebra analogue were crossed in \cite{r}, an error 
that  was corrected in \cite{r2} 
by weakening the result in \cite{r}. 
\begin{theorem}\label{SL} Let $\alpha$ be any ordinal $>1$.
Then for every infinite cardinal $\kappa\geq |\alpha|$, there exist completely representable algebras
$\B, \A\in \CA_{\alpha}$, that are weak set algebras, such that $\At\A=\At\B$, $|\At\B|=|\B|=\kappa$,
$\B\notin {\bf El}{\sf Nr}_{\alpha}\CA_{\alpha+1}$,
$\A\in {\sf Nr}_{\alpha}\CA_{\alpha+\omega}$,  and $\Cm\At\B=\A$,
so that $|\A|=2^{\kappa}$. In particular,
${\sf Nr}_{\alpha}\CA_{\beta}\subsetneq \bold S_d{\sf Nr}_{\alpha}\CA_{\beta}\subseteq \bold S_c{\sf Nr}_{\alpha}\CA_{\beta}$. 
\end{theorem}
\begin{proof} 
Fix an infinite cardinal $\kappa\geq |\alpha|$.  Let $FT_{\alpha}$ denote the set of all finite transformations on an ordinal $\alpha$.
Assume that $\alpha>1$.  Let $\F$ be field of characteristic $0$ such that $|\F|=\kappa$,
$V=\{s\in {}^{\alpha}\F: |\{i\in \alpha: s_i\neq 0\}|<\omega\}$ and let
${\A}$ have universe $\wp(V)$ with the usual concrete operations.
Then clearly $\wp(V)\in {\sf Nr}_{\alpha}\sf CA_{\alpha+\omega}$.
Let $y$ denote the following $\alpha$--ary relation:
$y=\{s\in V: s_0+1=\sum_{i>0} s_i\}.$
Let $y_s$ be the singleton containing $s$, i.e. $y_s=\{s\}.$
Let ${\B}=\Sg^{\A}\{y,y_s:s\in y\}.$ Clearly $|\B|=\kappa$.
Now $\B$ and $\A$ having same top element $V$, share the same atom structure, namely, the singletons.
Thus $\Cm\At\B=\A$. 
As  proved in \cite{SL}, we have
$\B\notin {\bold  El}{\sf Nr}_{\alpha}{\sf CA}_{\alpha+1}$.  
\end{proof}

\begin{corollary}\label{SL2} Let $\alpha$ be any ordinal $>1$.  Then there exists  a completely representable $\B\in \CA_{\alpha}$ 
such that $\At\B\in \At{\sf Nr}_{\alpha}\CA_{\alpha+\omega}$, $\Cm\At\B\in {\sf Nr}_{\alpha}\CA_{\alpha+\omega}$, 
but $\B\notin {\sf Nr}_{\alpha}\CA_{\alpha+1}(\supseteq {\sf Nr}_{\alpha}\CA_{\alpha+\omega})$.
\end{corollary}
We conclude from the previous corollary that we cannot lift the result proved in theorem \ref{can} for the class 
$\bold L$ to the agebra level deleting $\At$. The result proved for atom structures is stronger than that proved on the 
level of algebras. The condition imposed in the last part of the theorem is both sufficient and necessary for 
lifting the result
to the algebra level.

What was proved in \cite{r2} is the $\sf Ra$ analogue of  the following which follows from item (2) of theorem \ref{can}.
\begin{corollary}\label{rain} Any class between ${\sf CRCA}_n$ and $\bold S_c\Nr_n\CA_{n+3}$ is not elementary. In particular, any class between $\bold S\Nr_n\CA_{\omega}\cap \bf At$ and
$\bold S_c\Nr_n\CA_{n+3}$ is not elementary.
\end{corollary}
For relation algebras a yet strictly stronger result than that that proved in \cite{r2}, but weaker than that alledged in \cite{r} can be proved:  
\begin{corollary}\label{rain2} Any class between ${\sf RRA}\cap \bold S_d\sf Ra\CA_{\omega}$ 
and $\bold S_c{\sf Ra}\CA_{k}$, $k\geq 6$,  is not elementary.
\end{corollary}
The strictness of the inclusion $\bold S_d\sf Ra\CA_{\omega}\subsetneq \bold S_c\sf Ra\CA_{\omega}$ is proved in \cite{r}.
However, in corollary \ref{rain2} we do not know whether we can replace $\bold S_d\sf Ra\CA_{\omega}$ by $\sf Ra\CA_{\omega}$, which we will succeed to do 
in the $\CA$ case in a moment. 
As a matter of fact, we do not know in the first place whether the last two classes are distinct or not, i.e whether the inclusion $\sf Ra\CA_{\omega}\subseteq \bold S_d\sf Ra\CA_{\omega}$ is strict or not.
However, by example \ref{SL} we know that the $\CA$ analogue of this  inclusion
is strict; the strictness witnessed the algebra denoted by $\B$  in {\it op.cit}.

To summarize: Baring in mind 
the following facts:  

\begin{itemize}

\item $\Nr_n\CA_{\omega}\subsetneq \bold S_d\Nr_n\CA_{\omega}\subsetneq \bold S_c\Nr_n\CA_{\omega}$ by example \ref{SL}, 
and the construction of $\A$ and $\B$ used in the proof of item (1) of the forthcoming theorem \ref{t}, 

\item that neither of the classes ${\sf CRCA}_n$ 
and $\bold S_d\Nr_n\CA_{\omega}\cap \bf At$ 
is contained in the other by the same algebras used in item (1) of theorem \ref{t}, and the construction in \cite{bsl}. From the construction in  \cite{bsl}, 
it can be directly distilled, that for every infinite cardinal $\kappa$, there  an atomic 
algebra $\B\in \Nr_n\CA_{\omega}\sim {\sf CRCA}_n$ having $2^{\kappa}$ many atoms.  

\item that neither of the classes ${\sf CRCA}_n$ and $\Nr_n\CA_{\omega}\cap \bf At$ are contained in each other by the same aforementioned references in the previous item,

\item and finally that $\bold S_d\Nr_n\CA_{\omega}\cap {\sf CRCA}_n\subsetneq \bold S_c\Nr_n\CA_{\omega}\cap {\sf CRCA}_n$,  with 
$\B$  used in item (1) of theorem \ref{t} witnessing the strictness of the inclusion, 
\end{itemize}
we get that item (2) of theorem  \ref{can} is stronger than corollary \ref{rain}, that item (2) of the next theorem is 
stronger than item (2) of theorem  \ref{can} 
too, and lastly that item (1)  in the following 
theorem is different than the latter two results. 

Now we prove the {\it strictly} 
stronger result from that proved in tem (2) of theorem \ref{can}, obatined by repalcing $\bold S_d\Nr_n\CA_{\omega}$ by the (smaller) class 
$\Nr_n\CA_{\omega}$. 
using the intervention of another construction.
\begin{theorem}\label{t}
\begin{enumerate}
\item Any class between $\Nr_n\CA_{\omega}\cap  {\sf CRCA}_n$ and $\bold S_d\Nr_n\CA_{n+1}$ is not elementary
\item Any class between $\Nr_n\CA_{\omega}\cap {\sf CRCA}_n$ and $\bold S_c\Nr_n\CA_{n+3}$ is not elementary
\end{enumerate}
\end{theorem}
\begin{proof}
For item (1): We  slighty modify the construction in \cite[Lemma 5.1.3, Theorem 5.1.4]{Sayedneat} lifted to any finite $n>2$. 
The algebras $\A$ and $\B$ constructed in {\it op.cit} satisfy that
$\A\in {\sf Nr}_n\CA_{\omega}$, $\B\notin {\sf Nr}_n\CA_{n+1}$ and $\A\equiv \B$.
As they stand, $\A$ and $\B$ are not atomic, but it
can be  fixed that they are atomic, giving the same result, by interpreting the uncountably many $n$--ary relations in the signature of 
$\M$ defined in \cite[Lemma 5.1.3]{Sayedneat} for $n=3$, which is the base of $\A$ and $\B$ 
to be {\it disjoint} in $\M$, not just distinct. In fact the construction is presented in this way in \cite{IGPL}.

Let us see show why. We work with $2<n<\omega$ instead of only $n=3$. The proof presented in {\it op.cit} lifts verbatim to any such $n$.
Let $u\in {}^nn$. Write $\bold 1_u$ for $\chi_u^{\M}$ (denoted by $1_u$ (for $n=3$) in \cite[Theorem 5.1.4]{Sayedneat}.) 
We denote by $\A_u$ the Boolean algebra $\Rl_{\bold 1_u}\A=\{x\in \A: x\leq \bold 1_u\}$ 
and similarly  for $\B$, writing $\B_u$ short hand  for the Boolean algebra $\Rl_{\bold 1_u}\B=\{x\in \B: x\leq \bold 1_u\}.$
Using that $\M$ has quantifier elimination we get, using the same argument in {\it op.cit} 
that $\A\in \Nr_n\CA_{\omega}$.  The property that $\B\notin \Nr_n\CA_{n+1}$ is also still maintained.
To see why consider the substitution operator $_{n}{\sf s}(0, 1)$ (using one spare dimension) as defined in the proof of \cite[Theorem 5.1.4]{Sayedneat}.

Assume for contradiction that 
$\B=\Nr_{n}\C$, with $\C\in \CA_{n+1}.$ Let $u=(1, 0, 2,\ldots, n-1)$. Then $\A_u=\B_u$
and so $|\B_u|>\omega$. The term  $_{n}{\sf s}(0, 1)$ acts like a substitution operator corresponding
to the transposition $[0, 1]$; it `swaps' the first two co--ordinates.
Now one can show that $_{n}{\sf s(0,1)}^{\C}\B_u\subseteq \B_{[0,1]\circ u}=\B_{Id},$ 
so $|_{n}{\sf s}(0,1)^{\C}\B_u|$ is countable because $\B_{Id}$ was forced by construction to be 
countable. But $_{n}{\sf s}(0,1)$ is a Boolean automorpism with inverse
$_{n}{\sf s}(1,0)$, 
so that $|\B_{Id}|=|_{n}{\sf s(0,1)}^{\C}\B_u|>\omega$, contradiction. 
One proves that $\A\equiv \B$ exactly like in \cite{Sayedneat}.  

Now we show that $\B$ is actually outside the bigger class $\bold S_d\Nr_n\CA_{n+1}$.
Take the cardinality $\kappa$ specifying the signature of $\M$ to be $2^{2^{\omega}}$ and assume for contradiction that  
$\B\in \bold S_d\Nr_n\CA_{n+1}\cap \bf At$. 
Then $\B\subseteq_d \mathfrak{Nr}_n\D$, for some $\D\in \CA_{n+1}$ and $\mathfrak{Nr}_n\D$ is atomic. For brevity, 
let $\C=\mathfrak{Nr}_n\D$. Then $\B_{Id}\subseteq_d \Rl_{Id}\C$; the last algebra is the Boolean algebra with universe 
$\{x\in \C:  x\leq Id\}$.
Since $\C$ is atomic,   $\Rl_{Id}\C$ is also atomic.  
Using the same reasoning as above, we get that $|\Rl_{Id}\C|>2^{\omega}$ (since $\C\in \Nr_n\CA_{n+1}$). 
By the choice of $\kappa$, we get that $|\At\Rl_{Id}\C|>\omega$. 
By $\B\subseteq_d \C$, we get that $\B_{Id}\subseteq_d \Rl_{Id}\C$ and   
that $\At\Rl_{Id}\C\subseteq \At\B_{Id}$, so $|\At\B_{Id}|\geq |\At\Rl_{Id}\C|>\omega$.   
But by the construction of $\B$, we have  $|\B_{Id}|=|\At\B_{Id}|=\omega$,   
which is a  contradiction and we are done.

For the second required, namely, item (2):  We need to remove the $\bold S_d$ from item (2) of theorem  \ref{can}.  
Example \ref{SL} alerts us to the fact that 
we are not done yet. We still have some work to do.
For this purpose, we need the intervention of the construction used in item (1). 
We use  
the (modified) algebras $\A$ and $\B$ that  we know  
satisfy: $\A\in {\sf Nr}_n\CA_{\omega}$, $\B\notin {\sf Nr}_n\CA_{n+1}$, $\A\equiv \B$
and both $\A$ and $\B$ are atomic. 

In item (2) of theorem \ref{can}, we have excluded any first order definable class between 
the two classes $\bold S_d\Nr_n\CA_{\omega}\cap {\sf CRCA}_n$ and $\bold S_c\Nr_n\CA_{n+3}$. 
So hoping for a contradiction, we can only 
assume that there is a class 
$\bold M$ between $\Nr_n\CA_{\omega}\cap {\sf CRCA_n}$ and $\bold S_d\Nr_n\CA_{\omega}\cap {\sf CRCA}_n$ that is first order definable. 
Then ${\bf El}(\Nr_n\CA_{\omega}\cap {\sf CRCA}_n)\subseteq \bold M\subseteq \bold S_d\Nr_n\CA_{\omega}\cap {\sf CRCA}_n$. 
We have, $\B\equiv \A$, and $\A\in \Nr_n\CA_{\omega}\cap {\sf CRCA}_n$, hence $\B\in {\bf El}(\Nr_n\CA_{\omega}\cap {\sf CRCA}_n)\subseteq \bold S_d\Nr_n\CA_{\omega}\cap {\sf CRCA}_n$. 
But in item (1), we showed that $\B$  
is in fact outside $\bold S_d\Nr_n\CA_{n+1}\supseteq \bold S_d\Nr_n\CA_{\omega}\supseteq \bold S_d\Nr_n\CA_{\omega}\cap {\bf At}\supseteq \bold S_d\Nr_n\CA_{\omega}\cap {\sf CRCA}_n$,  
getting the hoped for contradiction,
and consequently the required.  
\end{proof}
\begin{corollary} Let $2<n<\omega$ and $k\geq 3$. Then the following classes, together with the intersection of any two of them, the last four taken at the same $k$, are not elementary:
${\sf CRCA}_n$ \cite{HH}, $\Nr_n\CA_{n+k}$ \cite[Theorem 5.4.1]{Sayedneat}, 
$\bold S_d\Nr_n\CA_{n+k}$, $\bold S_c\Nr_n\CA_{n+k}$ and the class of algebras
having complete $k$--flat representations \cite{mlq}.
\end{corollary}

As indicated above, we do not know whether the $\sf Ra$ analogous of item (2) proved in 
theorem \ref{t}  holds or not. But the closely related following result is proved in \cite{bsl2} using instead Monk--like 
algebras. The integral relation algebra (in which $\sf Id$ is an atom) defined next
by listing its forbidden triples, will be used in the proof of the next theorem.

\begin{example}
Take $\R$ to be a symmetric, atomic relation algebra with atoms
$$\Id, \r(i),
\y(i), \bb(i):i<\omega.$$
Non-identity atoms have colours, $\r$ is red,
$\bb$ is blue, and $\y$ is yellow. All atoms are self-converse.
Composition of atoms is defined
by listing the forbidden triples.
The forbidden triples are (Peircean transforms)
or permutations of $(\Id, x, y)$ for $x\neq y$, and
$$(\r(i), \r(i), \r(j)), \; (\y(i), \y(i), \y(j)), \; (\bb(i), \bb(i), \bb(j))\; \; i\leq j < \omega$$
$\R$ is the complex algebra over this atom structure.

Let $\alpha$ be an ordinal.  $\R^\alpha$ is obtained from $\R$ by
splitting the atom $\r(0)$ into $\alpha$ parts $\r^k(0):k<\alpha$
and then taking the full complex algebra.
In more detail, we put red atoms $r^{k}(0)$ for $k<\alpha.$
In the altered algebra the forbidden triples are
$(\y(i), \y(i), \y(j)), (\bb(i), \bb(i), \bb(j)), \ \   i\leq j<\omega,$
$(\r(i), \r(i), \r(j)),  \ \  0<i\leq j<\omega,$
$(\r^k(0), \r^l(0), \r(j)), \ \   0<j<\omega, k,l<\alpha,$\\
$(\r^k(0), \r^l(0), \r^m(0)), \ \  k,l,m<\alpha.$
These algebras were used in \cite{bsl} to show that $\Ra\CA_k$ for all $k\geq 5$ is not elementary.
\end{example}

\begin{theorem}\label{raa} 
Any class between $\Ra\CA_{\omega}$ and $\Ra\CA_5$, as well as the class
 $\sf CRRA$, is not closed under $\equiv_{\infty, \omega}$. 
\end{theorem}
\begin{proof} In $\R^\alpha$, we use the
following abbreviations:
$r(0) = \sum_{k<\alpha}\r^k(0),$
$\r = \sum_{i<\omega}\r(i)$,
$\y = \sum_{i<\omega}\y(i)$ and 
$\bb = \sum_{i<\omega}\bb(i).$
These suprema exist because they are taken in the complex algebras which are complete.
The \emph{index} of $\r(i), \y(i)$ and $\bb(i)$ is $i$ and the index of
$\r^k(0)$ is also $0$.
Now let  $\B = \R^\omega$ and $\A=\R^{\mathfrak{n}}$.
We claim that
$\B\in \Ra\CA_{\omega}$ and $\A\equiv \B$.
For the first required, it is shown in \cite{bsl} that $\B$ has a cylindric basis by exhibiting a \ws\ for  \pe\
in the the cylindric-basis game, which is a simpler version of the hyperbasis game
\cite[Definition 12.26]{HHbook}. 
Now, let $\H$ be an $\omega$-dimensional cylindric basis for $\B$. Then  
$\Ca\H\in \CA_{\omega}$.  Consider the
cylindric algebra $\C = \Sg^{\Ca\H}\B$, the subalgebra of $\Ca\H$ generated by $\B$.  In principal, new two dimensional elements that 
were not originally in $\B$,  
can be created in $\C$ using the spare dimensions in $\Ca(\H)$.
But in fact $\B$ exhausts the $2$--dimensional elements of  $\mathfrak{Ra}\C$, more concisely, 
we have $\B=\mathfrak{Ra}\C$ \cite{bsl2}. 
Like the proof of theorem \ref{t}, we can show that $\A\equiv_{\infty, \omega}\B$ by replacing $\bold 1_{\sf Id}$ by the newly splitted $\r(0)$.
We have proved that $\B\in \Ra\CA_{\omega}$ and $\A\equiv \B$. 
In \cite{bsl2}, it is proved that $\A\notin \Ra\CA_5$. So we get the first required, namely, that any class $\bold K$, such that 
$\Ra\CA_\omega\subseteq \bold K\subseteq \Ra\CA_5$ is not 
closed under $\equiv_{\infty, \omega}$.

Now we show that $\sf CRRA$ is not closed under $\equiv_{\infty, \omega}$, strengthening the result in \cite{HH} that 
only shows that 
$\sf CRRA$ is not closed under elementary equivalence proving 
the remaining required. 
Since $\B\in \Ra\CA_{\omega}$ has countably many atoms, 
then $\B$ is completely representable \cite[Theorem 29]{r}.
For this purpose, we show that $\A$ is not completely representable. We work with the term algebra, $\Tm\At\A$, since the latter is completely representable 
$\iff$ the complex algebra is.
Let  $\r = \{\r(i): 1\leq i<\omega\}\cup \{\r^k(0): k<2^{\aleph_0}\}$, $\y = \{\y(i):  i\in \omega\}$, $\bb^+ = \{\bb(i): i\in \omega\}.$
It is not hard to check every element of $\Tm\At\A\subseteq \wp(\At\A)$ has the form  
$F\cup R_0\cup B_0\cup Y_0$, where $F$ is a finite set of atoms, $R_0$ is either empty or a co-finite subset of $\r$, $B_0$ 
is either empty or a co--finite subset of $\bb$, and $Y_0$ is either empty or a co--finite subset 
of $\y$. 
We show  that the existence of a complete representation necessarily forces a 
monochromatic triangle, that we avoided at the start when defining $\A$.
Let $x, y$ be points in the
representation with $M \models \y(0)(x, y)$.  For each $i< 2^{\aleph_0}$, there is a
point $z_i \in M$ such that $M \models {\sf red}(x, z_i) \wedge \y(0)(z_i, y)$ (some red $\sf red\in \r$).
Let $Z = \set{z_i:i<2^{\aleph_0}}$.  Within $Z$ each edge is labelled by one of the $\omega$ atoms in
$\y^+$ or $\bb^+$.  The Erdos-Rado theorem forces the existence of three points
$z^1, z^2, z^3 \in Z$ such that $M \models \y(j)(z^1, z^2) \wedge \y(j)(z^2, z^3)
\wedge \y(j)(z^3, z_1)$, for some single $j<\omega$  
or three  points $z^1, z^2, z^3 \in Z$ such that $M \models \bb(l)(z^1, z^2) \wedge \bb(l)(z^2, z^3)
\wedge \bb(l)(z^3, z_1)$, for some single $l<\omega$.  
This contradicts the
definition of composition in $\A$ (since we avoided monochromatic triangles).
We have proved that $\sf CRRA$ is not closed under $\equiv_{\infty, \omega}$, since $\A\equiv_{\infty, \omega}\B$, 
$\A$ is not completely representable,   but $\B$ is completely representable.
\end{proof}

{\bf Remark:} Fix $2<n<\omega$. The algebra $\A\in \Nr_n\CA_{\omega}$  used in theorem \ref{t} can be viewed as splitting 
the atoms of the atom structure ${\bf At}=({}^nn, \equiv_i, D_{ij})_{i,j<n}$ each to $\kappa$--many atoms so 
$\A$ can be denoted 
${\sf split}(\bold 1_{Id}, {\bf At}, \kappa)$ $(\kappa$ an uncountable cardinal).  
The algebra $\B\notin \Nr_n\CA_{n+1}$ can be viewed as splitting the same atom structure, each  atom -- except for one atom that is split into countably many atoms -- 
is also split into $\kappa$--many atoms, so $\B$ can be denoted by 
${\sf split}(\bold 1_{Id}, {\bf At}, \omega).$  By the same token, for an ordinal $\alpha$, $\R^{\alpha}$ can be denoted by
${\sf split}(\r(0), \R, \alpha)$ short for splitting (the red atom) 
$\r(0)$ into  $\alpha$ parts.

\begin{corollary} \label{Ra}Let $k\geq 5$. Then the classes 
$\sf CRRA$, $\sf Ra\CA_k$, $\bold S_d{\sf Ra}\CA_k$ and $\bold S_c{\sf Ra}\CA_k$ are not elementary \cite{HH, bsl2, r}. 
The first two classes are not closed under
$\equiv_{\infty\omega}$, but are closed under ultraproducts. Furthermore, $\Ra\CA_\omega\subseteq \bold S_d\Ra\CA_{\omega}\subsetneq 
\bold S_c\Ra\CA_{\omega}$.  
\end{corollary}
\begin{proof} The first two classes are closed under ultraproducts because they are psuedo-elementary (reducts of elementary classes), 
cf. \cite[Item (2), p. 279]{HHbook}, 
\cite[Theorem 21]{r}. Proving the strictness of the last inclusion can be easily distilled from the proof of 
\cite[Theorem 36]{r}. 
The rest of the alleged statements can also be found in \cite{r}. 
\end{proof}
For a class $\sf K$ of $\sf BAO$s, recall that $\sf K\cap \bf At$ denotes the class of atomic algebras in $\K$.
Let ${\sf FRRA}=\{\A\in {\sf RRA}: A=\wp(^2U) \text { some non--empty set $U$} \}.$ 
We use the relation algebra analogue of lemma \ref{n} proved in \cite{r}. Following the notation in {\t op.cit}, 
we denote the the $\Ra$ analogue of the game $\bold G^m$ by $F^m$ which is like the usual atomic 
game $G_r^{m}$ \cite{HHbook}, when $r=\omega$, and \pa\ is allowed to reuse the $m$ nodes in play.

From the blow up and blur construction used in \cite{ANT} 
one can show that the existence a finite relation algebra with an $n$--blur and no infinite $m$--dimensional hyperbasis with $2<n<m\leq \omega$
implies that $\bold S{\sf Nr}_n\CA_m$ is not atom--canonical. Assume that $\R$ is  such a relation algebra.
Take $\B={\sf Bb}_n(\R, J, E)$ with last notation as in \cite[Top of p. 78]{ANT}. Then $\B\in \RCA_n$. We claim that $\C=\Cm\At\B\notin \bold S{\sf Nr}_n\CA_{m}$.
Suppose for contradiction that $\C\subseteq \mathfrak{Nr}_n\D$,
where $\D$ is atomic, with $\D\in \CA_{m}$.  Then $\C$ has a (necessarily infinite $m$--flat representation), hence 
$\Ra\C$ has an infinite $m$--flat representation as an $\RA$.  But $\R$ embeds into $\Cm\At({\sf Bb}(\R, J, E)$) which, in turn, 
embeds  into $\Ra\C$, so $\R$ has an infinite $m$--flat representation. 
Hence $\R$ has a $m$--dimensional infinite hyperbases which
contradicts the hypothesis.     

We know that such algebras exist, having even {\it strong $l$--blurs} for any pre--assighned $n\leq l<\omega$
when $m=\omega$, identifying the existence of an 
$\omega$--dimensional infinite hyperbasis with the existence of a representation on an infinite base. 
Such algebras are the Maddux algebras denoted in \cite[Lemma 5]{ANT} by $\mathfrak{E}_k(2, 3)$ with $k$ the finite 
number of non--identity atoms depending recursively on $l$.  Here by a strong $l$--blur, we understand an $l$ blur $(J, E)$ 
($J$ is the set of blure and $E$ is the index blur)
as in  \cite[Definition 3.1]{ANT} satisfying using the notation in {\it op.cit}:  
$(\forall V_1,\ldots V_n, W_2,\ldots W_n\in J)(\forall T\in J)(\forall 2\leq i\leq n)
{\sf safe}(V_i,W_i,T)$ that is, there is for $v\in V_i$, $w\in W_i$ and $t\in T$, $v;w\leq t.$ 
 This last condition  formulated as $(J5)_n$ on \cite[p.79]{ANT}, is obtained from 
$(J4)_n$ formulated in the definition of an $l$--blur \cite[p.72]{ANT}, by replacing $\exists$ by $\forall$ plainly 
giving  a stronger condition.

The following theorem on non--atom canonicity was proved in \cite{weak} using Monk--like algebras.
\begin{theorem}\label{conditional2} Let $2<n<\omega$. Then there exists an atomic  countable
relation algebra $\R$, such that $\Mat_n(\At\R)$ forms an $n$--dimensional cylindric basis, $\A=\Tm{\sf Mat}_n(\At\R)\in {\sf RQEA}_n$, while
even the diagonal free reduct of the \de\ completion of $\A$, namely, $\Cm{\sf Mat_n}(\At\R$) is not representable. In particular, 
${\sf Mat}_n(\At\R)$ is a weakly, but not strongly representable 
$\QEA_n$ atom structure.
\end{theorem}
\begin{proof} Let $\G$ be a graph.
Let $\rho$ be a `shade of red'; we assume that $\rho\notin \G$. Let $L^+$ be the signature consisting of the binary
relation symbols $(a, i)$, for each $a \in \G \cup \{ \rho\}$ and
$i < n$.  Let $T$ denote the following (Monk) theory in this signature:
$M\models T\iff$
for all $a,b\in M$, there is a unique $p\in (\G\cup \{\rho\})\times n$, such that
$(a,b)\in p$ and if  $M\models (a,i)(x,y)\land (b,j)(y,z)\land (c,k)(x,z)$, $x,y, z\in M$, then $| \{ i, j, k \}> 1 $, or
$ a, b, c \in \G$ and $\{ a, b, c\} $ has at least one edge
of $\G$, or exactly one of $a, b, c$ -- say, $a$ -- is $\rho$, and $bc$ is
an edge of $\G$, or two or more of $a, b, c$ are $\rho$.

We denote the class of models of $T$ which can be seen as coloured undirected
graphs (not necessarily complete) with labels coming from
$(\G\cup \{\rho\})\times n$  by $\GG$.

Now specify $\G$ to be either:

\begin{itemize}
\item the graph with nodes $\N$ and edge relation $E$
defined by $(i,j)\in E$ if $0<|i-j|<N$, where $N\geq n(n-1)/2$ is a postive number.

\item or the $\omega$ disjoint union of $N$ cliques, same $N$.
\end{itemize}

In both cases the countably infinite graphs 
contain infinitey many $N$ cliques.  In the first they overlap, in the second they do not.
One shows that there is a countable ($n$--homogeneous) coloured graph  (model) $M\in \GG$
with the following property \cite[Proposition 2.6]{Hodkinson}:

If $\triangle \subseteq \triangle' \in \GG$, $|\triangle'|
\leq n$, and $\theta : \triangle \rightarrow M$ is an embedding,
then $\theta$ extends to an embedding $\theta' : \triangle'
\rightarrow M$.

Here the choice of $N\geq n(n-)/n$ is not 
haphazard it bounds the number of edges of any graph $\Delta$ of size $\leq n$.
This is crucial to show that for any permutation $\chi$ of $\omega \cup \{\rho\}$, $\Theta^\chi$
is an $n$-back-and-forth system on $M$ \cite{weak}, so that the countable atomic  
set algebra $\A$ based on $M$ whose top element is
obtained from $^nM$ by discarding assignments whose edges are labelled by one of $n$--shaded of reds ($(\rho, i): i<n)$) forming $W\subsetneq {}^nM$, 
is classically representable.  The classical semantics of  $L_{\omega, \omega}$ formulas 
and relativized semantics (restricting assignmjents to $W$), coincide, so that 
$\A$ is isomorphic  to a set algebra with top element $^nM$. 

Consider the following relation algebra atom structure 
$\alpha(\G)=(\{{\sf Id}\}\cup (\G\times n), R_{\sf Id}, \breve{R}, R_;)$, where:
The only identity atom is $\sf Id$. All atoms are self converse,
so $\breve{R}=\{(a, a): a \text { an atom }\}.$
The colour of an atom $(a,i)\in \G\times n$ is $i$. The identity $\sf Id$ has no colour. A triple $(a,b,c)$
of atoms in $\alpha(\G)$ is consistent if
$R;(a,b,c)$ holds $(R;$ is the accessibility relation corresponding to composition). Then the consistent triples 
are $(a,b,c)$ where:
One of $a,b,c$ is $\sf Id$ and the other two are equal, or
none of $a,b,c$ is $\sf Id$ and they do not all have the same colour, or
$a=(a', i), b=(b', i)$ and $c=(c', i)$ for some $i<n$ and
$a',b',c'\in \G$, and there exists at least one graph edge
of $G$ in $\{a', b', c'\}$.

$\C$ is not representable because $\Cm(\alpha(\G))$ is not representable  and 
${\sf Mat}_n(\alpha(\G))\cong \At\A$. To see why, for each  $m  \in {\Mat}_n(\alpha(\G)), \,\ \textrm{let} \,\ \alpha_m
= \bigwedge_{i,j<n}  \alpha_{ij}. $ Here $ \alpha_{ij}$ is $x_i =
x_j$ if $ m_{ij} = \Id$ and $R(x_i, x_j)$ otherwise, where $R =
m_{ij} \in L$. Then the map $(m \mapsto
\alpha^W_m)_{m \in {\Mat}_n(\alpha(\G))}$ is a well - defined isomorphism of
$n$-dimensional cylindric algebra atom structures.
Non-representability follows from the fact that  $\G$ is a bad graph, that is, 
$\chi(\G)=N<\infty$ \cite[Definition 14.10, Theorem 14.11]{HHbook}.
The relation algebra atom structure specified above is exactly like the one in Definition 14.10 in {\it op.cit}, except that we have $n$ colours 
rather than just three.
\end{proof}

Now using the construction in \cite{weak, ANT} on `non atom--canonicity', 
we prove `non--finite axiomatizability results for relation and 
cylindric algebra.
Let ${\sf LCA}_n$ denote the elementary class of ${\sf RCA}_n$s satisfying the Lyndon conditions. 
We stipulate that $\A\in {\sf LCA}_n\iff$ $\A$ is atomic and $\At\A$ satifies the Lyndon conditions \cite[Definition 3.5.1]{HHbook2}. 
In the same sense, let $\sf LCRA(\subseteq \sf RRA)$ denote  the elementary class of $\sf RA$s satsfying the Lyndon conditions \cite[Definition pp. 337]{HHbook}.
We now use the above construction
proving non--atom canonicity to prove non--finite 
axiomatizability.
We use `bad' non--representable 
Monk--like algebras converging to a `good' representable one. 
In the process, we recover the results of Monk and Maddux on non--finite axiomatizability of both 
${\sf RCA}_n$  $(2<n<\omega)$  
and $\sf RRA$.

Let $2<n<\omega$. Then $\sf LCRA$ and ${\sf LCA}_n$ are 
not finitely axiomatizable.

{\it First proof:} Let $\A_l$ be the atomic $\sf RCA_n$ 
constructed from $\G_l$, $l\in \omega$ where $\G_l$  which is a disjoint countable union of $N_l$ cliques, such that for $i<j\in \omega$, $n(n-1)/n\leq N_i<N_j.$
Then  $\Cm\A_l$ with $\A_l$ based on $\G_l$, as constructed in the proof of theorem \ref{conditional2} is not representable.
So $(\Cm(\A_l): l\in \omega)$ is a sequence of non--representable algebras,
whose ultraproduct $\B$, being based on the ultraproduct of graphs having arbitrarily large chromatic number, 
will have an infinite clique, and so $\B$ will be completely representable \cite[Theorem 3.6.11]{HHbook2}.
Likewise, the sequence $(\Tm(\A_l): l\in \omega)$ is a sequence of representable, 
but not completely representable algebras, whose ultraproduct is completey representable.
The same holds for the sequence of 
relation algebras $(\R_l:l\in \omega)$ constructed as in \cite{weak}  for 
which $\Tm\At\A_l\cong \Mat_n\At\R_l$.

{\it Second proof:}  For each $2<n\leq l<\omega$, 
let $\R_l$ be the finite Maddux algebra $\mathfrak{E}_{f(l)}(2, 3)$ used in \cite{ANT} with strong $l$--blur
$(J_l,E_l)$ and $f(l)\geq l$ denoted by $l<k<\omega$ in \cite[Lemma 5]{ANT}. The relation algebra $\R_l$ has no representations on an infinite base. 
Let ${\cal R}_l={\sf Bb}(\R_l, J_l, E_l)\in {\sf RRA}$ with atom structure $\bf At$ obtained by blowing up and blurring $\R_l$ 
(with underlying set  denoted by $At$ on \cite[p.73]{ANT}), and let $\A_l=\Nr_n{\sf Bb}_l(\R_l, J_l, E_l)$ as defined on \cite[Top of p. 78]{ANT}; then $\A_l\in \RCA_n$. 
Then  $(\At{\cal R}_l: l\in \omega\sim n)$, and $(\At\A_l: l\in \omega\sim n)$ are sequences of weakly representable atom structures 
that are not strongly representable with a completely representable 
ultraproduct. The (complex algebra) sequences $(\Cm \At{\cal R}_l: l\in \omega\sim n)$, 
$(\Cm\At\A_l: l\in \omega\sim n$) are typical examples of what Hirsch and Hodkinson call `bad Monk (non--representable) algebras' 
converging to  `good (representable) one, 
namely their (non-trivial) ultraproduct. 
Observe that for $2<n\leq k <m<\omega$, $\A_k=\mathfrak{Nr}_k\A_m$. Such sequences witness the non--finite axiomatizability of the class 
representable agebras and the elementary closure of the class completely representable ones, namely, 
the class of algebras satisfying the Lyndon conditions. This recovers Monk's and Maddux's classical results  
on non--finite axiomatizability of $\sf RRA$s and $\RCA_n$s
since algebras considered are generated by a 
single $2$--dimensional elements. 

\subsection{Rainbows versus Monk--like algebras, pros and cons}  

The model--theoretic ideas used in the theorems \ref{can} and the construction in \cite{weak} proving 
theorem \ref{conditional2} are quite close but only superficially.
In the overall structure; they follow closely the model--theoretic framework in \cite{Hodkinson}.
In both cases, we have finitely many shades of red outside a  Monk-like and rainbow
signature,  that were used as labels to construct an $n$-- homogeneous model $M$ in the
expanded signature. Though the shades of reds are {\it outside} the signature, they were used as labels
during an $\omega$--rounded game played on labelled finite graphs--which can be seen as finite models in  the extended signature having size $\leq n$--
in which \pe\ had a \ws,  enabling her to
construct the required $M$ as a countable limit of the finite graphs played during the game. The construction, in both cases, entailed that any subgraph (substructure)
of $M$ of size $\leq n$, is independent of its location in $M$;
it is uniquely determined by its isomorphism type.

A relativized set algebra $\A$ based on $M$ was constructed
by discarding all assignments whose edges are labelled
by these shades of reds,  getting a set of $n$--ary sequences $W\subseteq {}^nM$. This $W$ is definable in $^nM$ by an $L_{\infty, \omega}$ formula
and the semantics with respect to $W$ coincides with classical Tarskian semantics (on $^nM$) for formulas of the
signature taken in $L_n$ (but not for formulas taken in $L_{\infty, \omega}^n$).
This was proved in both cases using certain $n$ back--and--forth systems, thus $\A$ is representable classically,
in fact it (is isomorphic to a set algebra that) has base $M$.
{\it The heart and soul of both proofs is to replace the reds labels by suitable
non--red binary relation symbols within an $n$ back--and--forth system, so that one  can
adjust that the system maps a tuple $\bar{b} \in {}^n M \backslash W$ to a tuple
$\bar{c} \in W$ and this will preserve any formula
containing the non--red symbols that are
`moved' by the system.  In fact, all
injective maps of size $\leq n$ defined on $M$ modulo an appropriate
permutation of the reds will
form an $n$ back--and--forth system.} 

This set algebra $\A$
was further proved to be atomic, countable, and simple (with top element $^nM$). Surjections into 
the subgraphs of size $\leq n$ of $M$ whose edges are not labelled by any shade of red are 
the atoms of $\A$   
and its \de\ completion, $\C=\Cm\At\A$.
Both algebras have top element $W$, but $\Cm\At\A$ is not in $\bold S{\sf Nr}_n\CA_{n+3}$ in case of the rainbow construction, least representable,
and it  is  (only) not representable in the case of the Monk--like algebra.
In case of both constructions `the shades of  red' -- which can be intrinsically identified with
non--principal ultrafilters in $\A$,  were used as colours, together with the principal ultrafilters
to represent completely $\A^+$, inducing a representation of $\A$.  Non--representability of $\Cm\At\A$ in the Monk case, used Ramsey's theory. The non neat--embeddability of
$\Cm\At\A$ in the rainbow case, used {\it the finite number of greens}, that gave us information on
when $\Cm\At\A$   `stops to be representable.'  
The reds in both cases has to do with representing $\A$.
The model theory used for both constructions is almost identical.

Nevertheless,  from the algebraic point of view,
there is a crucial difference.
The non--representability of $\Cm\At\A$ was tested by a game between the two players \pa\ and \pe.
The \ws's of the two players are independent, this is  reflected by the fact
that we have  two `independent parameters' $\sf G$ (the greens)  and $\sf R$ (the reds) that are 
more were finite irreflexive complete graphs. 
In  Monk--like constructions like the one used in \cite{weak} to show that $\RCA_n$ is not atom--canonical by constructing a countable atomic $\A$
the non representability of $\Cm\At\A$ was also tested by a game between \pe\ and \pa. 
But in {\it op.cit}  \ws's are interlinked, one operates through the other; hence only one parameter is the source of colours.
This parameter is a graph $\G$ which is a countable disjoint union of $N$ cliques where $N\geq n(n-1)/2$. Non--representability of the complex algebra $\Cm\At\A$ in
this case depends only on the {\it chromatic number of $\G$}, via an application of Ramseys' theorem; this number is $<\omega$; 
in fact it is $N$.

{\it In both cases two players operate using `cardinality of a labelled graph'.
\pa\ trying to make this graph too large for \pe\ to cope, colouring some of
its edges suitably.} 
For the rainbow case, it is a red clique formed during the play as we have seen in the model theoretic proof of theorem \ref{can}. 
It might be clear in both cases (rainbow and Monk--like algebras),  to see that \pe\ cannot win the infinite game, but what might not
be clear is {\it when does this happens; we know it eventually happen in finitely many round, but till which round \pe\ has not lost yet}.
The non--representability of $\Cm\At\A$ amounts  to that $\Cm\At\A\notin \bold S{\sf Nr}_n\CA_{n+k}$ for some finite $k$ because
$\RCA_n=\bigcap_{k<\omega}\bold S{\sf Nr}_n\CA_{n+k}$. Can we {\bf pin down} the value of $k$? 
getting an estimate that is not `infinitely' loose. 
{\it By adjusting the number of greens in the proof of theorem \ref{can} to be the least possible that outfits the $n$ `reds', 
namely, $n+1$, one gets a finer result than Hodkinson's \cite{Hodkinson} where Hodkinson uses an `overkill' of  
infinitely many greens. 
By truncating  the greens to be $n+1$, which is the smallest number $\lambda(>n)$ (of greens) for which \pa\ has a \ws\ in the private 
\ef\ forth game ${\sf EF}_\lambda(\lambda, n)$, lifted to the rainbow game $G^{n+3}_{\omega}(\At\CA_{n+1, n})$, 
we could tell when $\bold S{\sf Nr}_n\CA_{n+k}$,  $3\leq k\leq \omega$ {\bf `stops to be atom--canonical'. 
This happened when $\Cm\At\A$ {\bf `stopped to be representable at $k=3$'}}. Here the number of nodes, namely, 
$n+3$ used by \pa\ in the graph game tells us that 
$\Cm \bf At$ stopped to have $m$--dilations for any $m\geq n+3$. 
The finite number of rounds $r$ gives no additional information,  as far as non--atom canonicity is concerned, 
so we can save ourselves the hassle of 
finding the least such $r$.  This is simply immaterial in the present context for here it is the (finite) number of nodes
used by \pa\ during the play, which is $n+3$,  not the (finite) number of rounds $<\omega$ 
needed for him to implement his \ws,  that matters to conclude that $\bold S{\sf Nr}_n\CA_{n+3}$ is not atom--canonical.}

Now what if in the `Monk construction' based on a graph having finite chromatic number as given in th construction in theorem \ref{conditional2},  we replace this graph by 
$\G$ for which $\chi(\G)=\infty?$    
Let us approach the problem abstractly. We procced like in \cite[\S 6.3]{HHbook2} using the notation in {\it op.cit}.
Let $\G$ be any graph. One can  define a family of first order structures (labelled graphs)  in the signature $\G\times n$, denote it by $I(\G)$
like we did in the proof of theorem \ref{conditional2} as follows: The first order model $M$ is in $I(\G)$ $\iff$ 
for all $a,b\in M$, there is a unique $p\in \G\times n$, such that
$(a,b)\in p$. If  $M\models (a,i)(x,y)\land (b,j)(y,z)\land (c,l)(x,z)$, then $| \{ i, j, l \}> 1 $, or
$ a, b, c \in \G$ and $\{ a, b, c\} $ has at least one edge
of $\G$.
For any graph $\Gamma$, let $\rho(\Gamma)$ be the atom structure defined from the class of models satisfying the above,
these are maps from $n\to M$, $M\in I(\Gamma)$, endowed with an obvious equivalence relation,
with cylindrifiers and diagonal elements defined as Hirsch and Hodkinson define atom structures from classes of models,
and let $\mathfrak{M}(\Gamma)$ be the complex algebra of this atom structure as defiined  on p.78 right before 
\cite[Lemma 3.6.4]{HHbook2}.

We used the Monk--like algebras in \cite{weak, ANT} to construct a sequence of bad Monk algebras converging to a good one 
proving non finite axiomatizability of $\sf RRA$ and $\RCA_n$.
Let us reverse the process, dealing with good Monk algebras converging to a bad one aspiring to 
prove non--first order definability of the class of strongly representable
atom structures of $\RA$ and $\RCA_n$.

Let $\G$ be a graph. Consider the relation algebra atom structure $\alpha(\G)$ defind exactly like in the proof of theorem 
\ref{conditional2}. Then the relation complex algebra based on this atom structure will  have an $n$--dimensional cylindric basis
and, in fact, the cylindric atom structure of $\mathfrak{M}(\G)$ 
is isomorphic (as a cylindric algebra
atom structure) to the atom structure $\Mat_n(\alpha(\G))$ of all $n$--dimensional basic
matrices over the relation algebra atom structure $\alpha(\G)$. 
It is plausible that one can prove that 
$\alpha(\G)$ is strongly representable $\iff \mathfrak{M}(\G)$ is representable 
$\iff \G$ has infinite chromatic number, so that one gets the result, that the class of strongly representable atom structure 
both $\RA$s and 
$\CA$s of finite dimension at least three, is not elementary in one go.
The underlying idea here is that shade of red $\rho$
will appear in the {\it ultrafilter extension} of $\G$, if it 
has infinite chromatic number, as a reflexive node 
\cite[Definition 3.6.5]{HHbook2} and its $n$--copies $(\rho, i)$,  $i<n$, 
can be used to completely represent
$\mathfrak{M}(\G)^{+}$. 

For a variety $\V$ of Boolean algebras with operators, let ${\sf Str}(\V)$ be the following class of atom structures: $\{\F: \Cm\F\in \V\}$ \cite[Definition 2.89]{HHbook}.
Using a rainbow construction, we showed that $\bold S\sf Nr_n\CA_{n+k}$ is not 
atom--canonical, for any $k\geq 3$.
We know that 
${\sf Str}({\RCA}_n)$
is not elementary \cite[Corollary 3.7.2]{HHbook2}; this is proved using Monk-like algebras. 
Fix finite $k>2$.  Let $\V_k=\bold S{\sf Nr}_n\CA_{n+k}$. Then ${\sf Str}(\bold S{\sf Nr}_n\CA_{n+k})$ is not elementary 
$\implies   \V_k$ is not--atom canonical \cite[Theorem 2.84]{HHbook}.  
{\it But the converse implication, namely, 
$\V_k$ is not atom--canonical $\implies {\sf Str}(\V_k)$ not elementary does not hold in general.}
It is easy to show that there has to be a finite $k<\omega$ such that ${\sf Str}(\V_k)$ is not elementary.
It therefore seems plausible that Rainbows and Monk--like algebras 
can  be reconciled  to control the value of $m$ for which ${\sf Str}(\bold S{\sf Nr}_n\CA_{n+m})$ ($m>2$) 
is not elementary, thereby 
using a combination of variants of Ramsey's theorem 
and Erd\"os probabalistic graphs with a `rainbow intervention' to control the `extra dimensions'. 
The latter probabalistic
technique is used  in the aforementioned \cite[Corollary 3.7.2]{HHbook}.

\end{document}